\documentclass[12pt]{article}
\usepackage{amsmath, amssymb,amsfonts,bm,bbm,verbatim,multirow,color,xcolor,mathtools}
\usepackage{epic,eepic,psfrag,epsfig}
\usepackage[sf,bf,SF,footnotesize]{subfigure}
\usepackage{graphicx}
\usepackage{amsthm,breakcites}
\usepackage{amscd}
\usepackage{epsfig}
\usepackage{fullpage}
\usepackage{algorithm}
\usepackage{algpseudocode}
\usepackage{aligned-overset}
\usepackage{tikz-cd}
\usepackage{upgreek}

\usepackage{tikz}
\tikzstyle{edge} = [fill,opacity=.5,fill opacity=.5,line cap=round, line join=round, line width=50pt]
\pgfdeclarelayer{background}
\pgfsetlayers{background,main}
\usetikzlibrary{calc}
\usetikzlibrary{arrows,arrows.meta}

\usepackage[round]{natbib}

\usepackage{longtable}
\usepackage{setspace}
\usepackage{enumerate}
\usepackage{array}
\usepackage[small]{caption}
\RequirePackage[colorlinks,citecolor=blue,urlcolor=blue,breaklinks]{hyperref}

\newtheorem{thm}{Theorem}
\newtheorem{defn}[thm]{Definition}
\newtheorem{lemma}[thm]{Lemma}

\newtheorem{prop}[thm]{Proposition}
\newtheorem{cor}[thm]{Corollary}

\newtheorem*{remark}{Remark}

\newtheorem{assumption}{Assumption}

\DeclareMathOperator*{\argmin}{argmin}

\DeclareMathOperator*{\tr}{tr}
\DeclareMathOperator*{\Var}{Var}
\DeclareMathOperator*{\Cov}{Cov}
\makeatletter
\newcommand{\ostar}{\mathbin{\mathpalette\make@circled\star}}
\newcommand{\make@circled}[2]{%
  \ooalign{$\m@th#1\smallbigcirc{#1}$\cr\hidewidth$\m@th#1#2$\hidewidth\cr}%
}
\newcommand{\smallbigcirc}[1]{%
  \vcenter{\hbox{\scalebox{0.77778}{$\m@th#1\bigcirc$}}}%
}
\makeatother

\newlength{\widebarargwidth}
\newlength{\widebarargheight}
\newlength{\widebarargdepth}

\newcommand\numberthis{\addtocounter{equation}{1}\tag{\theequation}}

\def\hat{\widehat}
\def\tilde{\widetilde}

\usepackage{pgfplots}
\pgfplotsset{width=12cm,compat=1.9}

\usepackage[draft,inline,nomargin,index]{fixme}
\fxsetup{theme=color,mode=multiuser}

\FXRegisterAuthor{rs}{ars}{\color{green} RS}
\FXRegisterAuthor{kv}{akv}{\color{purple} KV}
\FXRegisterAuthor{tm}{atm}{\color{orange} TM}

\allowdisplaybreaks

\begin{document}
\title{High-probability minimax lower bounds}
\author{Tianyi Ma$^*$, Kabir A. Verchand$^{*, \circ}$
and Richard J. Samworth$^*$\\ \\
$^*$Statistical Laboratory, University of Cambridge\\
$^\circ$Industrial and Systems Engineering, Georgia Institute of Technology}
\maketitle
\begin{abstract}
    The minimax risk is often considered as a gold standard against which we can compare specific statistical procedures.  Nevertheless, as has been observed recently in robust and heavy-tailed estimation problems, the inherent reduction of the (random) loss to its expectation may entail a significant loss of information regarding its tail behaviour.  In an attempt to avoid such a loss, we introduce the notion of a minimax quantile, and seek to articulate its dependence on the quantile level.  To this end, we develop high-probability variants of the classical Le Cam and Fano methods, as well as a technique to convert local minimax risk lower bounds to lower bounds on minimax quantiles.  To illustrate the power of our framework, we deploy our techniques on several examples, recovering recent results in robust mean estimation and stochastic convex optimisation, as well as obtaining several new results in covariance matrix estimation, sparse linear regression, nonparametric density estimation and isotonic regression.  Our overall goal is to argue that minimax quantiles can provide a finer-grained understanding of the difficulty of statistical problems, and that, in wide generality, lower bounds on these quantities can be obtained via user-friendly tools.  
\end{abstract}
\section{Introduction}

Over the last 15 years or so, minimax lower bounds, which used to be regarded as primarily a branch of information theory, have entered the statistical mainstream.  This transition has been brought about by an increased awareness of what they have to offer to statisticians, 
in terms of providing a benchmark against which we can compare the performance of different statistical procedures.  Their increased prominence can be attributed in part to the surge of interest in high-dimensional statistical problems over the last 25 years, which has encouraged statisticians to move away from traditional diverging sample size asymptotics (with other problem parameters held fixed) and to focus instead on finite-sample performance bounds.

To fix the decision-theoretic foundations that underpin this work, suppose that we are interested in estimating a parameter $\theta$ taking values in a parameter space $\Theta$, based on data~$X$ taking values in a measurable space $\mathcal{X}$.  The notation here might suggest that $\Theta$ is a subset of a Euclidean space, so it is worth emphasising that we may be interested in estimating more complicated objects like density or regression functions, in which case $\Theta$ would be an appropriate function space.  Likewise, the data need not be a single observation, but could be a sample of size $n$, for example.  Let $\hat{\Theta}$ denote the set of estimators of $\theta \in \Theta$, i.e.~the set of measurable functions from $\mathcal{X}$ to $\Theta$.  Our loss function is a measurable function $L: \Theta \times \Theta \rightarrow [0,\infty)$.

Since $L\bigl(\hat{\theta}(X),\theta\bigr)$ is random when $\hat{\theta} \in \hat{\Theta}$, it is common to work with the \emph{risk function}
\[
R(\hat{\theta},\theta) \coloneqq \mathbb{E}_\theta L\bigl(\hat{\theta}(X),\theta\bigr),
\]
a deterministic quantity.  
A typical ideal goal, then, is to find an estimator $\hat{\theta}$ that achieves the \emph{minimax risk} in the sense that 
\[
\sup_{\theta \in \Theta} R(\hat{\theta},\theta) = \inf_{\tilde{\theta} \in \hat{\Theta}}\sup_{\theta \in \Theta} R(\tilde{\theta},\theta).
\]

In this paper, we revisit one of the early choices in the development of this paradigm, namely the reduction of the random $L\bigl(\hat{\theta}(X),\theta\bigr)$ to the deterministic risk function.  This simplification may fail to tell the whole story regarding the variability and tail behaviour of the loss function for different data realisations.  For instance, consider the simple problem of estimating a univariate mean $\theta$ with respect to squared error loss $L(\hat{\theta},\theta) \coloneqq (\hat{\theta} - \theta)^2$, based on a sample $X_1,\ldots,X_n$ from a distribution having variance at most $\sigma^2$.  It is well-known that the minimax risk in this problem is $\sigma^2/n$, and is attained by the sample mean $\bar{X} \coloneqq n^{-1}\sum_{i=1}^n X_i$.  On the other hand, \citet[][Proposition~6.2]{catoni2012challenging} proved that the sample mean has undesirable tail behaviour in the sense that for each $\delta \leq 1/e$, there exists a distribution with mean~$\theta$ and variance $\sigma^2$ such that with probability at least $\delta$,
\[
L(\bar{X},\theta) \geq \frac{\sigma^2}{n\delta}\biggl(1 - \frac{e\delta}{n}\biggr)^{n-1} > \frac{\sigma^2}{en\delta}.
\]
This lower bound admits an interpretation as a lower bound on the $(1 - \delta)$th quantile of $L(\bar{X}, \theta)$.  By contrast, the performance of the median-of-means estimator $\hat{\theta}^{\mathrm{MoM}}_{\delta}$~\citep[][p.~243]{nemirovski1983problem} can be controlled via an upper bound on the $(1-\delta)$th quantile of $L(\hat{\theta}^{\mathrm{MoM}}_{\delta}, \theta)$: with probability at least $1 - \delta$,
\[
L(\hat{\theta}^{\mathrm{MoM}}_{\delta}, \theta) \leq C \cdot  \frac{\sigma^2\log(1/\delta)}{n},
\]
for some universal constant $C > 0$.  

This stark gap in the dependence on $\delta$ in the above bounds, and hence the guarantees that can be made with high confidence for the two estimators, makes clear that the reduction to the minimax risk may entail significant information loss.  We can retain this information by defining, for $\delta \in (0,1]$, the \emph{minimax $(1-\delta)$th quantile}
\[
\mathcal{M}(\delta) \coloneqq \inf_{\hat{\theta} \in \hat{\Theta}} \sup_{\theta \in \Theta} \inf\Bigl\{r \in [0,\infty):\mathbb{P}_\theta\bigl\{L\bigl(\hat{\theta}(X),\theta\bigr) > r\bigr\} \leq \delta\Bigr\}.
\]
It is often the case that upper bounds on the minimax risk are obtained by finding an estimator $\hat{\theta} \in \hat{\Theta}$ and integrating a tail bound of the form
\[
\sup_{\theta \in \Theta} \mathbb{P}_\theta\bigl\{L\bigl(\hat{\theta}(X),\theta\bigr) > r_\delta\bigr\} \leq \delta;
\]
in such circumstances, we immediately have $\mathcal{M}(\delta) \leq r_\delta$, so upper bounds on the minimax quantile are frequently implicit in upper bound arguments for the minimax risk.  The main aim of this work is to introduce user-friendly tools to facilitate lower bounds on $\mathcal{M}(\delta)$; we give a variety of examples where these lower bounds match corresponding upper bounds up to constants, and thereby provide a finer-grained understanding of these statistical challenges.   

As two further illustrations of the benefits of minimax quantiles, we first remark that the basis for the analysis of several estimators is a union bound over (potentially exponentially) many different events.  Examples include the sample covariance matrix as an estimator of its population counterpart in operator norm~\citep[][Theorem 4.6.1]{vershynin2018high} and the Tukey median \citep{tukey1975mathematics,chen2018robust}, as well as tournament and skeleton estimators (\citet[][Chapters~6 and~7]{devroye2001combinatorial} and \citet{chen2016general}).  In such cases, it is the behaviour of extreme quantiles of simpler estimators that govern ultimate performance.  Second, consider the problem of robust mean estimation with respect to squared error loss in a Huber contamination model \citep{huber1964robust}.  There, the unboundedness of the loss function, combined with the fact that there is a positive probability with which all data are contaminated, ensures that the minimax risk is infinite, so uninformative from the point of view of comparing estimators; on the other hand, in Corollary~\ref{cor:huber-mean} we will determine (the finite) $\mathcal{M}(\delta)$ for a broad interval of quantile levels $\delta$.  

In the high-probability minimax framework, an estimator achieving an upper bound on the minimax quantile $\mathcal{M}(\delta)$ may in principle depend on $\delta$.  Indeed, this is the case in the median-of-means example discussed above.  Nevertheless, there is the potential for a single estimator $\hat{\theta}$ to attain the minimax quantile (up to constants) simultaneously over a wide range of quantile levels.  Such adaptation to the quantile level may be interpreted as a further justification for the use of $\hat{\theta}$.

To outline our contributions in greater detail, in Section~\ref{sec:setup}, we formally define the minimax quantile $\mathcal{M}(\delta)$, and a closely-related minorant $\mathcal{M}_-(\delta)$, which we call the \emph{lower minimax quantile}; see Definitions~\ref{defn:minimax-quantile} and~\ref{def:lower-minimax} below.  This latter quantity turns out to be especially amenable to analysis via standard tools in the minimax lower bound literature, and in Lemma~\ref{lemma:high-prob-le-cam-tv} and Lemma~\ref{lemma:fano's-minimax-quantile} respectively, we apply Le Cam's two-point lemma and Fano's inequality to obtain lower bounds on the lower minimax quantile.  In fact, in Theorem~\ref{lemma:expectation-lb-to-high-prob-lb}, we provide a reduction from a local minimax lower bound on the risk to a lower bound on $\mathcal{M}_-(\delta)$, so the full range of minimax lower bound techniques, including Assouad's lemma, can be applied.  

In Section~\ref{sec:examples}, we provide a suite of examples to illustrate the power of the framework and give a sense of the breadth of problems to which it may be applied.  Indeed, the settings we consider include (robust) multivariate mean estimation, covariance matrix estimation, sparse linear regression, nonparametric density estimation, isotonic regression and stochastic convex optimisation.  In all of these examples, we determine the minimax quantile up to multiplicative constants\footnote{Strictly speaking, in sparse linear regression, these constants depend on a restricted eigenvalue condition, while in nonparametric density estimation, they depend on smoothness parameter of a H\"older class.}.  One of the main messages we aim to convey through this selection of examples is that it is relatively straightforward to apply our results from Section~\ref{sec:setup}, and thereby to enhance the strength of the justification for many statistical procedures through the minimax quantile paradigm.  

Proofs of our main results in Section~\ref{sec:setup} are given in Section~\ref{sec:proofs-setup}, with the exception of one very short corollary.  The Appendix comprises some standard minimax risk lower bounds, as well as the statements and proofs of several auxiliary lemmas.





\subsection{Related work}

To the best of our knowledge, our work represents the first systematic study of high-probability minimax lower bounds, though isolated results in certain settings have started to appear in recent years.  The importance of the dependence on the quantile level of upper bounds on the minimax quantile has been emphasised in the literature on robust and heavy-tailed estimation \citep[e.g.][]{catoni2012challenging,lugosi2019sub-gaussian,hopkins2020mean}.  On the other hand, lower bounds are scarcer, though \citet{devroye2016subgaussian} established a lower bound on $\mathcal{M}_{-}(\delta)$ in univariate mean estimation, and \citet{depersin2022optimal} extended this result to the multivariate case using techniques developed by~\citet{lecue2013learning} based on the Gaussian shift theorem.  In particular, our Proposition~\ref{prop:mean-estimation} provides an alternative, streamlined proof of \citet[][Theorem~4]{depersin2022optimal}.  Contemporaneously with our work, \citet{elhanchi2024minimax} obtained an exact but less explicit expression for the minimax quantile in this mean estimation problem (and in linear regression) via a Bayesian argument.

In the context of linear functional estimation, \citet{donoho1991geometrizing} characterised $\mathcal{M}_-(\delta)$ in terms of a Hellinger modulus of continuity.  Under similar assumptions, explicit universal constants for matching upper and lower bounds on the minimax quantile $\mathcal{M}(\delta)$ were provided by \citet{polyanskiy2021dualizing}; see also \citet{juditsky2009convex}, whose primary focus is on affine estimators.  A quantity similar to $\mathcal{M}_-(\delta)$ was introduced by~\citet{donoho1994statistical} to study minimax confidence interval length in the setting of linear functional estimation.   He showed that it could be expressed in terms of another modulus of continuity~\citep[see also][Section 5]{cai2012minimax}.  

Recently, in a different yet still univariate setting,~\citet{wang2024anytime} and \citet{duchi2024information} invoked information-theoretic techniques to obtain lower bounds on the minimax length of an anytime-valid $(1-\delta)$-level confidence interval.  \citet[][Theorem 2.3]{tsybakov2009introduction} established a lower bound on $\mathcal{M}_-(\delta)$ in the context of univariate nonparametric regression.  See also~\citet{ehrenfeucht1989general} for early work on sample complexity lower bounds for fixed quantile levels in classification problems within the learning theory community.  A closely-related concept to minimax quantiles is that of the Accuracy Confidence (AC) function; indeed, this may be regarded as a generalised inverse of the lower minimax quantile function.  Accurate bounds on the AC function in regression and classification problems have been obtained by~\citet{devore2006approximation}, \citet{temlyakov2008approximation} and~\citet{kerkyacharian2014optimal}.    

In stochastic convex optimisation,~\citet{davis2021low} consider heavy-tailed stochastic gradients and design methods with logarithmic dependence on the quantile level.  Very recently,~\citet{carmon2024price} introduced the notion of a minimax quantile in this setting and applied information-theoretic tools to obtain lower bounds that match their upper bounds up to multiplicative universal constants.  In Section~\ref{sec:SCO}, we provide a different proof of their lower bound based on our general theory developed in Section~\ref{sec:setup} in conjunction with Le Cam's two-point method.

\subsection{Notation}
Given $a, b \in \mathbb{R}$, we write $a \vee b \coloneqq \max(a,b)$ and $a \wedge b \coloneqq \min(a,b)$.  For $n\in\mathbb{N}$, we define $[n] \coloneqq \{1,\ldots,n\}$.
We let $\mathbb{S}^{d-1}$ denote the unit Euclidean sphere in $\mathbb{R}^d$ and for $R > 0$, write $\mathbb{B}_2^{d}(R)$ for the closed Euclidean ball in $\mathbb{R}^d$ of radius $R$.  For $v=(v_1,\ldots,v_d)^\top \in\mathbb{R}^d$ and $q>0$, we let $\|v\|_q \coloneqq \bigl( \sum_{j=1}^d |v_j|^q \bigr)^{1/q}$, $\|v\|_0 \coloneqq \sum_{j=1}^d \mathbbm{1}_{\{v_j \neq 0\}}$ and $\|v\|_{\infty} \coloneqq \max_{j\in[d]} |v_j|$.  We let $\mathcal{S}^{d \times d}_{+}$ and $\mathcal{S}_{++}^{d \times d}$ denote the cones of symmetric, positive semi-definite and symmetric, positive definite matrices in $\mathbb{R}^{d \times d}$ respectively.  For symmetric matrices $A, B \in \mathbb{R}^{d \times d}$, we write $A \succeq B$ if $A - B \in \mathcal{S}^{d \times d}_{+}$ and $A \preceq B$ if $B - A \in \mathcal{S}^{d \times d}_{+}$.  For a matrix $A = (A_{ij})_{i \in [n], j \in [d]} \in \mathbb{R}^{n \times d}$, we let $\| A \|_{\infty} \coloneqq \max_{i \in [n], j \in [d]} \lvert A_{ij} \rvert$ and $\|A\|_{\mathrm{op}} \coloneqq \sup_{u\in\mathbb{S}^{n-1}, v\in\mathbb{S}^{d-1}} u^\top A v$.  Given $(a_1, \ldots, a_d)^\top \in\mathbb{R}^d$, let $\mathrm{diag}(a_1, \ldots, a_d) \in \mathbb{R}^{d \times d}$ denote the diagonal matrix with entries $a_1, \ldots a_d$.  We let $I_{d} \coloneqq \mathrm{diag}(1, \ldots, 1)$ denote the identity matrix in $\mathbb{R}^{d \times d}$ and, for $j \in [d]$, let $e_j \in \mathbb{R}^d$ denote the $j$th standard basis vector. Given a measurable space $(\mathcal{X}, \mathcal{A})$, we will denote the space of all probability measures on $\mathcal{X}$ as $\mathcal{Q}(\mathcal{X})$.  For $\omega=(\omega_1,\ldots,\omega_d)^\top, \omega'=(\omega_1',\ldots,\omega_d')^\top \in\{0,1\}^d$, we define their Hamming distance as $d_{\mathrm{H}}(\omega,\omega') \coloneqq \sum_{j=1}^d \mathbbm{1}_{\{\omega_j \neq \omega_j'\}}$.  We let $\mathrm{Pow}(S)$ denote the power set of an arbitrary set $S$.

If $I$ is an arbitrary index set, then for functions $f,g:I \rightarrow \mathbb{R}$, we write $f \gtrsim g$ if there exists a universal constant $c>0$ such that $f(i) \geq cg(i)$ for all $i \in I$, and write $f \lesssim g$ if there exists a universal constant $C>0$ such that $f(i) \leq Cg(i)$ for all $i \in I$. We write $f \asymp g$ if both $f \lesssim g$ and $f\gtrsim g$.  Similarly, we write $f \asymp_{\alpha} g$ if there exist $c(\alpha),C(\alpha)>0$, depending only on $\alpha$, such that $c(\alpha) g(i) \leq f(i) \leq C(\alpha) g(i)$ for all $i\in I$.

\section{Problem setup and main results} \label{sec:setup}


Let $(\Theta,d)$ be a non-empty pseudo-metric\footnote{$(\Theta,d)$ is a pseudo-metric space if $d:\Theta\times\Theta \to [0,\infty)$ satisfies for all $\theta_1,\theta_2,\theta_3 \in \Theta$ that $d(\theta_1,\theta_1) = 0$, $d(\theta_1,\theta_2) = d(\theta_2,\theta_1)$ and $d(\theta_1,\theta_2) \leq d(\theta_1,\theta_3) + d(\theta_2,\theta_3)$; in contrast to a metric space, we may have $d(\theta_1,\theta_2) = 0$ with $\theta_1 \neq \theta_2$.} space, equipped with the topology induced by the open balls in $d$ and the corresponding Borel $\sigma$-algebra.  It will be convenient for some of our examples to consider a slightly more general setting than that outlined in the introduction.  In particular, for $\theta\in\Theta$, let $\mathcal{P}_{\theta}$ denote a family of probability measures on a measurable space $(\mathcal{X},\mathcal{A})$, and define $\mathcal{P}_{\Theta} \coloneqq \{\mathcal{P}_{\theta} : \theta\in\Theta\}$.  Although in many applications, each $\mathcal{P}_\theta$ will consist of a single distribution, the additional generality allows us to cover heavy-tailed and Huber contamination models with our theory.  

Let $g:[0,\infty) \to [0,\infty)$ be an increasing (i.e.~non-decreasing) function and define a loss function $L : \Theta\times\Theta \to [0,\infty)$ by
\begin{align*}
    L(\theta_1, \theta_2) \coloneqq g\bigl( d(\theta_1,\theta_2) \bigr),
\end{align*}
for $\theta_1,\theta_2 \in \Theta$.  Writing $\hat{\Theta}$ for the set of all measurable functions from $\mathcal{X}$ to $\Theta$, for $\delta\in(0,1]$, $\theta\in\Theta$, $P_{\theta}\in\mathcal{P}_{\theta}$ and $\hat{\theta} \in \hat{\Theta}$, we then define the \emph{$(1-\delta)$th quantile} of $L(\hat{\theta}, \theta)$ under $P_{\theta}$ as
\begin{align*}
    \mathrm{Quantile} \bigl( 1-\delta; P_{\theta},L(\hat{\theta}, \theta) \bigr) \coloneqq \inf \Bigl\{ r\in [0,\infty) : P_{\theta} \bigl\{ L(\hat{\theta}, \theta) \leq r \bigr\} \geq 1-\delta \Bigr\}.
\end{align*}
In our more general setting, we have the following definition of minimax quantiles.
\begin{defn}[Minimax quantile]\label{defn:minimax-quantile}
    For $\delta\in(0,1]$, we define the \emph{minimax $(1-\delta)$th quantile} of the loss $L$ for estimating $\theta$ over $\mathcal{P}_{\Theta}$ as
    \begin{align*}
        \mathcal{M}(\delta) \equiv \mathcal{M}(\delta, \mathcal{P}_{\Theta}, L) \coloneqq \inf_{\hat{\theta}_{\delta} \in \hat{\Theta}} \sup_{\theta\in\Theta} \sup_{P_{\theta} \in \mathcal{P}_{\theta}}\, \mathrm{Quantile} \bigl( 1-\delta; P_{\theta},L(\hat{\theta}_{\delta}, \theta) \bigr).
    \end{align*}
\end{defn}
A high-probability upper bound for an estimator implies an upper bound on the minimax quantile. That is, if there exists $\hat{\theta}_{\delta} \in \hat{\Theta}$ such that with $P_{\theta}$-probability at least $1-\delta$, we have $L(\hat{\theta}_{\delta}, \theta) \leq \mathrm{UB}(\delta)$ for all $\theta\in\Theta$ and $P_{\theta}\in\mathcal{P}_{\theta}$, then $\mathcal{M}(\delta) \leq \mathrm{UB}(\delta)$.  The main goal of this section is to provide general lower bounds on minimax quantiles. The following basic result shows that a lower bound on the minimax quantile $\mathcal{M}(\delta)$ implies a lower bound on the minimax risk. 

\begin{prop} \label{lemma:minimax-quantile-to-minimax-risk}
    For $\delta\in(0,1]$, we have
    \begin{align*}
        \inf_{\hat{\theta} \in \hat{\Theta}} \sup_{\theta \in \Theta} \sup_{P_{\theta} \in \mathcal{P}_{\theta}} \mathbb{E}_{P_{\theta}} L(\hat{\theta}, \theta) \geq \delta \mathcal{M}(\delta).
    \end{align*}
\end{prop}
Key to our development is the notion of a lower minimax quantile, which turns out to be more amenable to direct analysis.
\begin{defn}[Lower minimax quantile] \label{def:lower-minimax}
For $\delta\in(0,1]$, we define the \emph{lower minimax $(1-\delta)$th quantile} of the loss $L$ for estimating $\theta$ over $\mathcal{P}_{\Theta}$ as 
\[
\mathcal{M}_{-}(\delta) \equiv \mathcal{M}_-(\delta,\mathcal{P}_{\Theta}, L) \coloneqq \inf \Bigl\{ r\in [0,\infty) : \inf_{\hat{\theta} \in \hat{\Theta}} \sup_{\theta\in\Theta} \sup_{P_{\theta} \in \mathcal{P}_{\theta}}\, P_{\theta} \bigl\{ L(\hat{\theta}, \theta) > r \bigr\} \leq \delta \Bigr\}.
\]
\end{defn}
The following theorem relates minimax quantiles to their lower analogues.  In particular, it reveals that the lower minimax quantile is not only a minorant of the corresponding minimax quantile, but also that the two quantities coincide almost everywhere.  
\begin{thm}\label{lemma:important-lower-bound}
    For all $\delta\in(0,1]$ and $\xi \in (0,\delta)$, we have
    \begin{align}
    \label{Eq:LowerUpperBound}
        \mathcal{M}_-(\delta) \leq \mathcal{M}(\delta) \leq \mathcal{M}_{-}(\delta - \xi).
    \end{align}
Moreover, $\mathcal{M}(\delta) = \mathcal{M}_-(\delta)$ at all but countably many points $\delta \in (0,1]$.
\end{thm}
The convenience of the lower bound in Theorem~\ref{lemma:important-lower-bound} arises from the fact that if 
\begin{align*}
    \inf_{\hat{\theta} \in \hat{\Theta}}\sup_{\theta\in\Theta} \sup_{P_{\theta} \in \mathcal{P}_{\theta}}\, P_{\theta} \bigl\{ L(\hat{\theta}, \theta) > r \bigr\} > \delta,
\end{align*}
then $\mathcal{M}(\delta) \geq \mathcal{M}_{-}(\delta) \geq r$. 

With these notions in place, we now establish high-probability versions of standard information-theoretic tools, beginning with a high-probability variant of Le Cam's two-point lemma.  
\begin{lemma}[High-probability Le Cam's two-point lemma]
\label{lemma:high-prob-le-cam-tv}
    Let $\delta\in(0,1/2)$, and suppose that $\theta_1,\theta_2 \in \Theta$, $P_{1} \in \mathcal{P}_{\theta_1}$ and $P_{2} \in \mathcal{P}_{\theta_2}$ satisfy $\mathrm{TV}(P_1, P_2) < 1-2\delta$. Then, writing $\eta \coloneqq d(\theta_1,\theta_2)/2$, we have $\mathcal{M}_{-}(\delta) \geq g(\eta)$.
\end{lemma}
In our applications in Section~\ref{sec:examples}, we will typically apply this lemma to two product distributions for which the total variation distance may be difficult to compute.  The following corollary provides a more user-friendly version of the above lemma, stated in terms of the Kullback--Leibler divergence.
\begin{cor}\label{cor:high-prob-le-cam-kl}
    Let $\delta\in(0,1/2)$, and suppose that $\theta_1,\theta_2 \in \Theta$, $P_{1} \in \mathcal{P}_{\theta_1}$ and $P_{2} \in \mathcal{P}_{\theta_2}$ satisfy $\mathrm{KL}(P_1, P_2) < \log\bigl(\frac{1}{4\delta(1-\delta)}\bigr)$. Then, writing $\eta \coloneqq d(\theta_1,\theta_2)/2$, we have $\mathcal{M}_{-}(\delta) \geq g(\eta)$.
\end{cor}

\begin{proof}
    By the Bretagnolle--Huber inequality \citep{bretagnolle1979estimation}, 
    \begin{align*}
        \mathrm{TV}(P_1, P_2) \leq \sqrt{1 -  \exp\bigl\{ -\mathrm{KL}(P_1, P_2) \bigr\}} < 1-2\delta.
    \end{align*}
    The result then follows by Lemma~\ref{lemma:high-prob-le-cam-tv}.
\end{proof}
As we will see in Section~\ref{sec:examples}, in several examples, the minimax quantile decomposes into a sum of a term reflecting the dependence on the quantile level $\delta$, and another capturing the minimax risk (that is independent of $\delta$).  Lemma~\ref{lemma:high-prob-le-cam-tv} and Corollary~\ref{cor:high-prob-le-cam-kl} are effective for controlling the first of these terms.  On the other hand, our next two results are well-suited for the minimax risk term, with the first providing lower bounds on the minimax quantile via a specialisation of Fano's lemma.
\begin{lemma}\label{lemma:fano's-minimax-quantile}
    Let $M \geq 2$, $\theta_1,\ldots,\theta_M \in \Theta$, and $P_j \in \mathcal{P}_{\theta_j}$ for $j\in[M]$. Suppose that there exists $\epsilon \in(0,1]$ such that 
    \begin{align*}
        \frac{M^{-1} \inf_{Q\in\mathcal{Q}(\mathcal{X})} \sum_{j=1}^M \mathrm{KL}(P_{j},Q) + \log(2-M^{-1})}{\log M} \leq 1-\epsilon.
    \end{align*}
    Then, writing $\eta \coloneqq \min_{1\leq j < k \leq M} d(\theta_j, \theta_k)/2$, we have $\mathcal{M}_{-}(\delta) \geq g(\eta)$ for all $\delta\in(0,\epsilon)$.
\end{lemma}


The second of our results demonstrates how to convert lower bounds on the minimax risk over bounded subsets of the parameter space into lower bounds on the lower minimax quantile. 
\begin{thm} \label{lemma:expectation-lb-to-high-prob-lb}
Let $\Theta_0 \subseteq \Theta$ be non-empty and suppose that $D\coloneqq \sup_{\theta_1,\theta_2\in\Theta_0} d(\theta_1,\theta_2)$ satisfies $g(D) > 0$. Let $\Delta \in [0,\infty)$ be such that 
    \begin{align*}
        \inf_{\hat{\theta} \in \hat{\Theta}} \sup_{\theta_0 \in \Theta_0} \sup_{P_{\theta_0} \in \mathcal{P}_{\theta_0}} \mathbb{E}_{P_{\theta_0}} L( \hat{\theta}, \theta_0) \geq \Delta.
    \end{align*}
    Then for every $\epsilon > 0$ and $\delta \in \bigl(0, \frac{\Delta - g(\epsilon  D)}{g((1+\epsilon )D)} \bigr)$, we have $\mathcal{M}_{-}(\delta) \geq g(\epsilon D)$.
\end{thm}
In order to apply this theorem, we note that standard information-theoretic minimax risk lower bound techniques (including Assouad's lemma) proceed by constructing a carefully-chosen finite subset $\Theta_0$ of the parameter space $\Theta$, and hence provide local minimax risk lower bounds over $\Theta_0$; see Propositions~\ref{prop:mean-estimation} and~\ref{prop:isotonic-book}.  During the preparation of this manuscript, we became aware of the recent result of~\citet[][Lemma~6]{chhor2024generalized}, obtained independently, which provides a similar lower bound. 

The upper bound on $\delta$ in Theorem~\ref{lemma:expectation-lb-to-high-prob-lb} may be small in some cases; Proposition~\ref{cor:boost-delta} provides a way to relax this to any universal constant less than $1/2$.  Recall that a \emph{Polish space} is a separable topological space that may be equipped with a metric that generates the topology and under which the space is complete; for example, any open or closed subset of $\mathbb{R}^d$ is a Polish space.  
\begin{prop} \label{cor:boost-delta}
    For $m\in\mathbb{N}$, let $\mathcal{P}_{\theta}^{\otimes m} \coloneqq \{ P_{\theta}^{\otimes m} \in \mathcal{Q}(\mathcal{X}^m) : P_{\theta} \in \mathcal{P}_{\theta} \}$ and $\mathcal{P}_{\Theta}^{\otimes m} \coloneqq \{\mathcal{P}_{\theta}^{\otimes m} : \theta\in\Theta\}$. Suppose that $\Theta$ is a non-empty Polish space, that $g$ is continuous  
    and that there exists $A>0$ such that $L(\theta_1,\theta_2) \leq A\bigl\{ L(\theta_1,\theta_3) + L(\theta_2,\theta_3) \bigr\}$ for all $\theta_1,\theta_2,\theta_3\in\Theta$.  
    Define $h:[0,1] \rightarrow [0,1]$ by $h(x) \coloneqq x^3+3x^2(1-x)$ and for $m\in\mathbb{N}$, let $h^{\circ m}$ denote the function obtained by composing $h$ with itself $m$ times.  Let $\delta_-\in(0,1/2)$ and $\delta_+\in(\delta_-,1/2)$.  Then there exists $k \equiv k(\delta_-,\delta_+)\in\mathbb{N}$ such that $h^{\circ k}(\delta_+) \leq \delta_-$. Moreover, for any such $k$, we have for $\delta \in (0,\delta_+]$ that
    \begin{align*}
        \mathcal{M}_{-}(\delta,\mathcal{P}_{\Theta},L) \geq \frac{1}{(2A)^k} \cdot \mathcal{M}_{-}\bigl(\delta_-, \mathcal{P}_{\Theta}^{\otimes 3^k},L\bigr).
    \end{align*}
\end{prop}

\section{Examples} \label{sec:examples}

With the general techniques of Section~\ref{sec:setup} now in place, the aim of this section is to determine expressions for the minimax quantiles of loss functions in several examples.  In each of the statements to follow, we will consider quantile levels $\delta \leq 1/4$; this value is chosen simply for convenience and could be relaxed to any constant less than $1/2$ by, e.g., a further application of Proposition~\ref{cor:boost-delta}.  Throughout this section, we will slightly abuse notation when the loss is of the form $L(\theta, \theta') = g(\| \theta - \theta' \|)$ for some norm $\|\cdot\|$, by writing $\mathcal{M}\bigl(\delta, \mathcal{P}_{\Theta}, g(\| \cdot \|)\bigr)$ in place of $\mathcal{M}(\delta, \mathcal{P}_{\Theta}, L)$.

\subsection{Multivariate mean estimation} \label{sec:mean-estimation}

For this basic problem, we determine the minimax quantile in three settings, one for each of the next subsections.
  
\subsubsection{Gaussian mean estimation in squared Euclidean norm}

Proposition~\ref{prop:mean-estimation} below recovers the minimax $(1 - \delta)$th quantile in multivariate Gaussian mean estimation, up to multiplicative universal constants.  The matching upper bound is attained by the sample mean, which is a $\delta$-independent estimator.  This result therefore demonstrates the optimality of the sample mean in Gaussian mean estimation in a much stronger manner than usual minimax risk lower bounds. 
\begin{prop}\label{prop:mean-estimation}
Let $n,d\in\mathbb{N}$, $\delta \in (0, 1/4]$, $\Sigma \in \mathcal{S}_{++}^{d \times d}$, $\Theta \coloneqq \mathbb{R}^d$ and $\mathcal{P}_{\theta} \coloneqq \{ \mathsf{N}_d(\theta, \Sigma)^{\otimes n} \}$ for $\theta \in \Theta$.  Then 
\[
\mathcal{M}(\delta, \mathcal{P}_{\Theta}, \| \cdot \|_2^2) \asymp \frac{\tr(\Sigma)}{n} + \frac{ \| \Sigma \|_{\mathrm{op}} \log(1/\delta)}{n}. 
\]
\end{prop}
\begin{proof}
Let $v_1,\ldots,v_d \in \mathbb{R}^{d}$ denote orthonormal eigenvectors of $\Sigma$ with corresponding eigenvalues $\lambda_1 \geq \ldots \geq \lambda_d$.  We capture each of the two additive terms separately, starting with the $\delta$-dependent term.  To this end, take $\theta_1 \coloneqq 0$, $\theta_2 \coloneqq \sqrt{\frac{\lambda_1}{n}\log\bigl(\frac{1}{4\delta(1-\delta)}\bigr)}\, v_1$ and set $P_1 = \mathsf{N}_d(\theta_1, \Sigma)^{\otimes n}$ and $P_2 = \mathsf{N}_d (\theta_2, \Sigma )^{\otimes n}$.  Then 
\[
\mathrm{KL}(P_1, P_2) = n \mathrm{KL}\bigl(\mathsf{N}_d(\theta_1, \Sigma), \mathsf{N}_d(\theta_2, \Sigma)\bigr) = \frac{1}{2}\log\Bigl(\frac{1}{4\delta(1-\delta)}\Bigr) < \log\Bigl(\frac{1}{4\delta(1-\delta)}\Bigr).
\]
Thus, applying Theorem~\ref{lemma:important-lower-bound} in conjunction with Corollary~\ref{cor:high-prob-le-cam-kl}, we deduce that for $\delta \in (0,1/4]$,
\begin{align}\label{ineq:mean-est-high-prob-lb}
\mathcal{M}(\delta, \mathcal{P}_{\Theta}, \| \cdot \|_2^2) \geq \mathcal{M}_{-}(\delta, \mathcal{P}_{\Theta}, \| \cdot \|_2^2) \geq \frac{\| \Sigma \|_{\mathrm{op}} \log\bigl(\frac{1}{4\delta(1-\delta)}\bigr)}{4n} \geq \frac{\| \Sigma \|_{\mathrm{op}} \log(1/\delta)}{20n}.
\end{align}
We turn now to the other term in the lower bound.  For $\theta,\theta' \in \mathbb{R}^d$, we have 
\[
\|\theta - \theta'\|_2^2 = \biggl \| \sum_{j=1}^{d} \bigl\{(v_j^{\top}\theta)  v_j -  (v_j^{\top}\theta') v_j\bigr\} \biggr\|_2^2 = \sum_{j=1}^{d} \bigl\{v_j ^{\top}(\theta - \theta') \bigr\}^2 = \sum_{j=1}^{d} g\bigl(\rho_j(\theta, \theta')\bigr),
\]
where we defined the pseudo-metrics $\rho_j:\Theta \times \Theta \rightarrow [0,\infty)$ by  $\rho_j(\theta, \theta) \coloneqq \lvert v_j^\top (\theta - \theta') \rvert$ for $j \in [d]$ and the increasing function $g:[0,\infty) \rightarrow [0,\infty)$ by $g(x) \coloneqq x^2$.  Now let $\Phi \coloneqq \{0, 1\}^{d}$ and for $j \in [d]$ set $a_j \coloneqq \frac{4}{3} \sqrt{\frac{\lambda_j}{n}}$.  For $\phi = (\phi_1,\ldots,\phi_d)^\top \in \Phi$, take $\theta_{\phi} \coloneqq \sum_{j =1}^{d} a_j \phi_j v_j$ and let $P_{\phi} \coloneqq \mathsf{N}_d(\theta_{\phi}, \Sigma)^{\otimes n}$. Define $\Theta_0 \coloneqq \{\theta_{\phi} : \phi\in\Phi\} \subseteq \Theta$ with diameter $D\coloneqq \frac{4}{3} \sqrt{\frac{\tr(\Sigma)}{n}}$.  For $\phi, \phi' \in \Phi$ that differ only in the $j$th coordinate, we have $\rho_j(\theta_{\phi}, \theta_{\phi'}) = \lvert a_j \rvert$ and, by Pinsker's inequality~\citep[Lemma 15.2]{wainwright2019high}, $\mathrm{TV}(P_{\phi}, P_{\phi'}) \leq \bigl\{\frac{n}{2} \mathrm{KL}\bigl(\mathsf{N}_d(\theta_{\phi}, \Sigma), \mathsf{N}_d(\theta_{\phi'}, \Sigma)\bigr)\bigr\}^{1/2} = 2/3$.  Thus, by Assouad's lemma (Lemma~\ref{lemma:assouad}), 
\[
\inf_{\hat{\theta} \in \hat{\Theta}} \max_{\theta_0 \in \Theta_0} \sup_{P_{\theta_0} \in \mathcal{P}_{\theta_0}} \mathbb{E}_{P_{\theta_0}} \bigl( \| \hat{\theta} - \theta_0 \|_2^2 \bigr) \geq \frac{4 \tr(\Sigma)}{27n}.
\]
Taking $\epsilon = 3/40$ in Theorem~\ref{lemma:expectation-lb-to-high-prob-lb}, we deduce for $\delta \in (0,1/15]$ that
\begin{align}\label{ineq:mean-est-assouad-lb}
\mathcal{M}_{-}(\delta, \mathcal{P}_{\Theta}, \| \cdot \|_2^2) \geq g(\epsilon D) =  \frac{\tr(\Sigma)}{100n}.
\end{align}
Now, recalling the function $h$ from Proposition~\ref{cor:boost-delta}, since $h\bigl(h(1/4)\bigr)<1/15$, we may apply Proposition~\ref{cor:boost-delta} with $k=2$ and~\eqref{ineq:mean-est-assouad-lb} to deduce that for any $\delta \in (0,1/4]$,
\begin{align} 
\label{eq:mean-est-boosted}
\mathcal{M}_-(\delta, \mathcal{P}_{\Theta}, \| \cdot \|_2^2) \geq \frac{\tr(\Sigma)}{2^6 \cdot 3^2 \cdot 5^2 \cdot n}.
\end{align}
Combining Theorem~\ref{lemma:important-lower-bound} with~\eqref{ineq:mean-est-high-prob-lb} and~\eqref{eq:mean-est-boosted} therefore yields that for $\delta \in (0,1/4]$,
\begin{align*}
\mathcal{M}(\delta, \mathcal{P}_{\Theta}, \| \cdot \|_2^2) \geq \mathcal{M}_-(\delta, \mathcal{P}_{\Theta}, \| \cdot \|_2^2) \geq \frac{\tr(\Sigma)}{2^7 \cdot 3^2 \cdot 5^2 \cdot n} + \frac{\| \Sigma \|_{\mathrm{op}} \log(1/\delta)}{40n}.
\end{align*}

For the upper bound, consider the sample mean as an estimator.  Set $\bar{X} \coloneqq n^{-1} \sum_{i=1}^n X_i$, where $X_1,\ldots,X_n \stackrel{\mathrm{iid}}{\sim} \mathsf{N}(\theta,\Sigma)$, and note that by Jensen's inequality, $\mathbb{E}(\|\bar{X}-\theta\|_2) \leq \sqrt{\tr(\Sigma)/n}$.  Moreover, $\|\bar{X} - \theta\|_2 = \sup_{u\in\mathbb{S}^{d-1}} u^\top (\bar{X} - \theta)$, so using the fact that $2(a+b) \geq (a^{1/2} + b^{1/2})^2$ for $a,b \geq 0$ by another application of Jensen's inequality, as well as the Borell--TIS inequality~\citep[see, e.g.,][Theorem 2.1.1]{adler2009random}, we have for any $\delta \in (0,1]$ that
\begin{align*}
\mathbb{P}\biggl(\|\bar{X} - \theta\|_2^2 > \frac{2\tr(\Sigma)}{n} &+ \frac{4\|\Sigma\|_{\mathrm{op}}\log(1/\delta)}{n}\biggr) \\
&\leq \mathbb{P}\biggl(\|\bar{X} - \theta\|_2 > \mathbb{E}(\|\bar{X}-\theta\|_2) + \sqrt{\frac{2\|\Sigma\|_{\mathrm{op}}\log(1/\delta)}{n}}\biggr) \leq \delta,
\end{align*}
as required.
\end{proof} 

\subsubsection{Robust mean estimation in squared Euclidean norm}

Recent literature has studied high-probability upper bounds for mean estimation with heavy-tailed data (see, e.g., the review article of~\citet{lugosi2019mean} and the references therein), or more generally, data from the Huber $\varepsilon$-contamination model~\citep[see, e.g.,][]{lugosi21robust,depersin2022robust}.  These  upper bound are attained using $\delta$-dependent estimators.  As a corollary of Proposition~\ref{prop:mean-estimation}, we provide a matching lower bound on the minimax $(1 - \delta)$th quantile up to universal constants in these settings.  
\begin{cor} \label{cor:huber-mean}
    Let $n,d\in\mathbb{N}$, $\varepsilon\in [0,1)$ $\Sigma \in \mathcal{S}_{++}^{d \times d}$, $\Theta = \mathbb{R}^d$ and 
    \[
    \mathcal{P}_{\theta} \coloneqq \Bigl\{R^{\otimes n}:\, R \in \mathcal{Q}(\mathbb{R}^d),\, R = (1-\varepsilon)P + \varepsilon Q, \, \mathbb{E}_P(X) = \theta,\, \mathrm{Cov}_{P}(X) = \Sigma,\, Q\in\mathcal{Q}(\mathbb{R}^d) \Bigr\}.
    \]
    There exists a universal constant $c>0$ such that if $e^{-cn(1-300\varepsilon)} \leq \delta \leq 1/4$, then
\[
\mathcal{M}(\delta, \mathcal{P}_{\Theta}, \| \cdot \|_2^2) \asymp \frac{\tr(\Sigma)}{n} + \| \Sigma \|_{\mathrm{op}}\varepsilon + \frac{ \| \Sigma \|_{\mathrm{op}} \log(1/\delta)}{n}. 
\]
\end{cor}
\begin{proof}
    We first consider the lower bound.  When $\varepsilon = 0$, the claimed lower bound follows from Proposition~\ref{prop:mean-estimation}, so without loss of generality we assume that $\varepsilon > 0$.  First, since $\mathsf{N}_d(\theta,\Sigma)^{\otimes n} \in \mathcal{P}_{\theta}$, we have by Proposition~\ref{prop:mean-estimation} that for $\delta \in (0,1/4]$,
    \begin{align} \label{eq:huber-lb-first-term}
        \mathcal{M}(\delta, \mathcal{P}_{\Theta}, \| \cdot \|_2^2)  \gtrsim \frac{\tr(\Sigma)}{n} + \frac{ \| \Sigma \|_{\mathrm{op}} \log(1/\delta)}{n}. 
    \end{align}
    By the spectral theorem, $\Sigma = V \Lambda V^{\top} \in \mathbb{R}^{d \times d}$ where $\Lambda \coloneqq \mathrm{diag}(\lambda_1, \ldots, \lambda_d) \in \mathbb{R}^{d \times d}$, $V \in \mathbb{R}^{d \times d}$ is orthogonal and $\lambda_1 \geq \ldots \geq \lambda_d$.  Let $a \coloneqq (\varepsilon+\varepsilon^2)/2$ and $b \coloneqq (3\varepsilon+\varepsilon^2)/2$.  Define random vectors $X = (X_1, \ldots, X_d)^\top \sim P_X$ and $Y = (Y_1, \ldots, Y_d)^\top \sim P_Y$ with independent components satisfying
    \begin{align*}
        X_1 \coloneqq \begin{cases}
            -\sqrt{\frac{\lambda_1}{2\varepsilon}} \quad &\text{with probability } \varepsilon\\
            0 &\text{with probability } 1-2\varepsilon,\\
            \sqrt{\frac{\lambda_1}{2\varepsilon}} \quad &\text{with probability } \varepsilon
        \end{cases} \ \
        Y_1 \coloneqq \begin{cases}
            -\sqrt{\frac{\lambda_1}{2\varepsilon}} \quad &\text{with probability } a\\
            0 &\text{with probability } 1 - a - b \\
            \sqrt{\frac{\lambda_1}{2\varepsilon}} \quad &\text{with probability } b,
        \end{cases}
    \end{align*}
    and $X_j \overset{d}{=} Y_j \sim \mathsf{N}(0, \lambda_j)$ for all $j \in \{2, \ldots, d\}$.  Then 
    \[
    \Var(X_1) = \lambda_1, \quad \Var(Y_1) = \frac{\lambda_1}{2\varepsilon}\{(a+b) - (b-a)^2\} = \lambda_1.
    \]
    Thus $\Cov(X) = \Cov(Y) = \Lambda$, and
    \begin{align*}
        \mathrm{TV}(P_X,P_Y) = \frac{1}{2} \biggl( \frac{\varepsilon - \varepsilon^2}{2} + \varepsilon^2 + \frac{\varepsilon + \varepsilon^2}{2} \biggr) = \frac{\varepsilon + \varepsilon^2}{2} \leq \varepsilon.
    \end{align*}
    Define $X' \coloneqq VX$, $Y' \coloneqq VY$ and let $P_{X'}$, $P_{Y'}$ denote the distributions of $X'$, $Y'$ respectively.  Then $\Cov(X') = \Cov(Y') = \Sigma$, and by the data processing inequality~\citep[e.g.,][Lemma 2.1]{gerchinovitz2020fano}, 
    \begin{align*}
        \mathrm{TV}(P_{X'},P_{Y'}) \leq \mathrm{TV}(P_X,P_Y) \leq \varepsilon \leq \frac{\varepsilon}{1-\varepsilon}.
    \end{align*}
    Then, let $\mathcal{R}_{\theta} \coloneqq \bigl\{R \in \mathcal{Q}(\mathbb{R}^d):\, \mathbb{E}_{R}(X) = \theta, \mathrm{Cov}_R(X) = \Sigma\bigr\}$ and $\mathcal{R}_{\Theta} \coloneqq \{ \mathcal{R}_{\theta} : \theta\in\Theta \}$.  Then, writing $\theta_1 \coloneqq \mathbb{E}(X')$ and $\theta_2 \coloneqq \mathbb{E}(Y')$, we have by Theorem~\ref{lemma:important-lower-bound} and Lemma~\ref{thm:general-lower-bound-huber} that
    \begin{align} \label{eq:huber-lb-second-term}
        \mathcal{M}(\delta,\mathcal{P}_{\Theta},\|\cdot\|_2^2) \geq \mathcal{M}_-(\delta,\mathcal{P}_{\Theta},\|\cdot\|_2^2) \geq \frac{\|\theta_1 - \theta_2\|_2^2}{4} = \frac{(b - a)^2 \cdot \| \Sigma \|_{\mathrm{op}}}{8\varepsilon}  = \frac{\|\Sigma\|_{\mathrm{op}}\varepsilon}{8}.
    \end{align}
    Combining \eqref{eq:huber-lb-first-term} and \eqref{eq:huber-lb-second-term} proves that the stated lower bound holds for all $\delta \in (0,1/4]$.
    
    The matching upper bound follows from~\citet[Theorem 2.1]{depersin2022robust}, by taking $u = \infty$ in their notation (note that the lower bound on $\delta$---which in turn implies an upper bound on $\varepsilon$---is implicit from their theorem).
\end{proof}

\subsubsection{Gaussian mean estimation in \texorpdfstring{$\ell_{\infty}$}{ell-infinity}-norm}
In our next result, we switch to $\ell_\infty$-loss, and invoke Fano's lemma (Lemma~\ref{lemma:fano's-minimax-quantile}) for the lower bound.   As in Proposition~\ref{prop:mean-estimation}, the matching upper bound is achieved by the sample mean.
\begin{prop} \label{prop:lb-ell-infinity-norm}
    Let $n,d\in\mathbb{N}$, $\sigma > 0$, $\delta\in (0, 1/4]$, $\Theta \coloneqq \mathbb{R}^d$ and $\mathcal{P}_{\theta} \coloneqq \bigl\{ \mathsf{N}_d (\theta, \Sigma)^{\otimes n} : \Sigma \in \mathcal{S}^{d\times d}_{++},\, \|\Sigma\|_{\infty} \leq \sigma^2 \bigr\}$. Then
    \begin{align*}
        \mathcal{M}(\delta, \mathcal{P}_{\Theta}, \|\cdot\|_{\infty}) \asymp \sigma \sqrt{\frac{\log(d/\delta)}{n}}.
    \end{align*}
\end{prop}
\begin{remark} In our lower bound construction below, we only use distributions whose covariance matrices are scalar multiples of the identity matrix.  In that setting, we see from Proposition~\ref{prop:lb-ell-infinity-norm} that the term involving the quantile level $\delta$ dominates when $\delta \ll 1/d$, whereas with squared Euclidean error, Proposition~\ref{prop:mean-estimation} reveals that this happens when $\delta \ll e^{-d}$.
\end{remark}
\begin{proof}
    Define $\theta_1 \coloneqq \sigma\sqrt{\frac{1}{n} \log\bigl(\frac{1}{4\delta(1-\delta)}\bigr)} \,e_1$, $\theta_2 \coloneqq 0$, and $P_\ell \coloneqq \mathsf{N}_d(\theta_\ell,\sigma^2 I_d)^{\otimes n} \in \mathcal{P}_{\theta_\ell}$ for $\ell\in\{1,2\}$. Then $\mathrm{KL}(P_1, P_2) = \frac{1}{2} \log\bigl(\frac{1}{4\delta(1-\delta)}\bigr) < \log\bigl( \frac{1}{4\delta(1-\delta)} \bigr)$. Hence, by Theorem~\ref{lemma:important-lower-bound} and Corollary~\ref{cor:high-prob-le-cam-kl}, we have for $\delta\in(0,1/4]$ that
    \begin{align} \label{eq:ell-infinity-lb2}
        \mathcal{M}(\delta,\mathcal{P}_{\Theta},\|\cdot\|_{\infty}) \geq \mathcal{M}_-(\delta,\mathcal{P}_{\Theta},\|\cdot\|_{\infty}) \geq \sigma\sqrt{\frac{1}{4n} \log\biggl(\frac{1}{4\delta(1-\delta)}\biggr)} \geq \sigma\sqrt{\frac{\log(1/\delta)}{20n}}.
    \end{align}
    Now assume for the time being that $d\geq 4$.  For $j \in [d]$, define $\theta^{(j)} \coloneqq \sigma \sqrt{\frac{\log d}{2n}} \, e_j$, $P^{(j)} \coloneqq \mathsf{N}_d(\theta^{(j)}, \sigma^2 I_d)^{\otimes n} \in \mathcal{P}_{\theta^{(j)}}$, and define $Q \coloneqq \mathsf{N}_d(0,\sigma^2 I)^{\otimes n} \in \mathcal{Q}\bigl((\mathbb{R}^d)^n\bigr)$. Then $\mathrm{KL}(P^{(j)}, Q) = n\mathrm{KL}\bigl( \mathsf{N}_d(\theta^{(j)}, \sigma^2 I_d), \mathsf{N}_d(0,\sigma^2 I_d) \bigr) = \frac{1}{4}\log d$ and thus
    \begin{align*}
        \frac{d^{-1} \sum_{j=1}^d \mathrm{KL}(P^{(j)},Q) + \log(2-d^{-1})}{\log d} \leq \frac{1}{4} + \frac{\log(2-d^{-1})}{\log d} \leq \frac{7}{10}.
    \end{align*}
    Therefore, by applying Theorem~\ref{lemma:important-lower-bound}, followed by Lemma~\ref{lemma:fano's-minimax-quantile} with $\epsilon = 3/10$, we deduce that for $\delta\in(0,1/4]$,
    \begin{align} \label{eq:ell-infinity-lb1}
        \mathcal{M}(\delta,\mathcal{P}_{\Theta},\|\cdot\|_{\infty}) \geq \mathcal{M}_-(\delta,\mathcal{P}_{\Theta},\|\cdot\|_{\infty}) \geq  \sigma\sqrt{\frac{\log d}{8n}}.
    \end{align}
    Combining \eqref{eq:ell-infinity-lb2} and \eqref{eq:ell-infinity-lb1} yields that for $\delta \in (0,1/4]$ and $d\geq 4$,
    \begin{align} \label{eq:ell-infty-norm-lb1}
        \mathcal{M}(\delta,\mathcal{P}_{\Theta},\|\cdot\|_{\infty}) \geq \frac{\sigma}{2} \biggl( \sqrt{\frac{\log d}{8n}} + \sqrt{\frac{\log(1/\delta)}{20n}} \biggr) \geq \sigma \sqrt{\frac{\log(d/\delta)}{80n}}.
    \end{align}
    Now consider the case $d \leq 3$.  By~\eqref{eq:ell-infinity-lb2}, we have for $\delta \in (0,1/4]$ that
    \begin{align} \label{eq:ell-infty-norm-lb2}
        \mathcal{M}(\delta,\mathcal{P}_{\Theta},\|\cdot\|_{\infty}) \geq \sigma\sqrt{\frac{\log(1/\delta)}{20n}} \geq \sigma \sqrt{\frac{\log(d/\delta)}{40n}}.
    \end{align}
    Combining~\eqref{eq:ell-infty-norm-lb1} and~\eqref{eq:ell-infty-norm-lb2} proves the lower bound.

    For the upper bound, let $X_1, \ldots, X_n \overset{\mathrm{iid}}{\sim} P_{\theta} \in \mathcal{P}_{\theta}$ and $\bar{X} \coloneqq n^{-1} \sum_{i=1}^n X_i$. Then for $j\in[d]$, 
    \[
    \mathbb{P}\biggl( |\bar{X}_j - \theta_j| > \sigma\sqrt{\frac{2\log(d/\delta)}{n}} \biggr) = 2\int_{\sqrt{2\log(d/\delta)}}^\infty \frac{1}{\sqrt{2\pi}} e^{-x^2/2} \, \mathrm{d}x \leq \frac{\delta}{d}.
    \]
    Applying a union bound, we deduce that
    \begin{align*}
        \mathbb{P}\biggl( \|\bar{X} - \theta\|_{\infty} > \sigma\sqrt{\frac{2\log(d/\delta)}{n}} \biggr) \leq \delta,
    \end{align*}
    so $\mathcal{M}_-(\delta,\mathcal{P}_{\Theta},\|\cdot\|_{\infty}) \leq \mathcal{M}(\delta,\mathcal{P}_{\Theta},\|\cdot\|_{\infty}) \leq \sigma\sqrt{\frac{2\log(d/\delta)}{n}}$ for $\delta \in (0,1]$.
\end{proof}

\subsection{Covariance matrix estimation} \label{sec:covariance-matrix}

Here we consider the estimation of a covariance matrix in operator norm based on a sample of size $n$, and determine the minimax quantile up to universal constants.  As we show in the proof of Proposition~\ref{prop:covariance-matrix-estimation} below, the upper bound is attained by the sample covariance matrix, applying the careful analysis of \citet{koltchinskii2017concentration} (see also~\citet{zhivotovskiy2024dimension}) when the quantile level $\delta$ satisfies $\delta \geq e^{-n}$, and the trivial zero estimator when $\delta < e^{-n}$.  By contrast, the matrix Bernstein inequality yields sub-optimal concentration as compared to the minimax quantile in Proposition~\ref{prop:covariance-matrix-estimation}; see Lemma~\ref{lemma:matrix-bernstein} for details.  
Recall that the \emph{effective rank} of $\Sigma \in \mathcal{S}_{+}^{d\times d}$ is defined as $\mathbf{r}(\Sigma) \coloneqq \tr(\Sigma) / \|\Sigma\|_{\mathrm{op}}$, with the convention that $0/0 \coloneqq 0$.
\begin{prop} \label{prop:covariance-matrix-estimation}
    Let $n,d \in \mathbb{N}$, $\delta\in (0,1/4]$, $\sigma>0$, $r\in[1,n \wedge d]$, $\Theta \coloneqq \bigl\{ \Sigma \in \mathcal{S}_+^{d\times d} : \|\Sigma\|_{\mathrm{op}} \leq \sigma^2,\, \mathbf{r}(\Sigma) \leq r \bigr\}$, and for $\Sigma \in \Theta$, define $\mathcal{P}_{\Sigma} \coloneqq \bigl\{ \mathsf{N}_d(0, \Sigma)^{\otimes n} \bigr\}$.
    Then
    \begin{align*}
        \mathcal{M}(\delta,\mathcal{P}_{\Theta},\|\cdot\|_{\mathrm{op}}) \asymp \sigma^2 \sqrt{\frac{r+\log(1/\delta)}{n}} \wedge \sigma^2.
    \end{align*}
\end{prop}

\begin{proof}
    Let $r_0 \coloneqq \lfloor r \rfloor$, so that $r_0 \geq (r/2) \vee 1$. 
    First consider the case $\delta \in [e^{-n}/3,1/4]$, and set $a\coloneqq 1 - \sqrt{\frac{1}{4n}\log\bigl(\frac{1}{4\delta(1-\delta)}\bigr)} \in [1/2, 1]$. Define $\Sigma^{(1)}, \Sigma^{(2)} \in \mathcal{S}_+^{d\times d}$ by
    \[
    \Sigma^{(1)} \coloneqq \begin{pmatrix}
        a^2\sigma^2 & 0 & 0\\
        0 & \sigma^2 I_{r_0 - 1} & 0\\
        0 & 0 & 0
    \end{pmatrix}
    \quad \text{and} \quad 
        \Sigma^{(2)} \coloneqq \begin{pmatrix}
        \sigma^2 I_{r_0} & 0\\
        0 & 0
    \end{pmatrix}.
    \]
    Then, for $\ell \in \{1, 2\}$, we have $\|\Sigma^{(\ell)}\|_{\mathrm{op}} \leq \sigma^2$ and $\mathbf{r}(\Sigma^{(\ell)}) \leq r_0 \leq r$, so that $\Sigma^{(\ell)} \in \Theta$. Letting $P_{\ell} \coloneqq \mathsf{N}_d(0,\Sigma^{(\ell)})^{\otimes n}$ for $\ell\in\{1,2\}$, we find by Taylor's theorem with the integral form of the remainder that
    \begin{align*}
        \mathrm{KL}(P_1,P_2) &= n\cdot \frac{\log(1/a^2) - (1 \!-\!a^2)}{2} = \frac{n}{2} \cdot \Bigl\{2(1 - a)^2 + 2 \int_{a}^{1} \frac{1}{t} \Bigl(1 - \frac{a}{t}\Bigr)^{2} \mathrm{d}t\Bigr\}\\
        &\leq 3n(1-a)^2
        = \frac{3}{4} \log\biggl( \frac{1}{4\delta(1-\delta)} \biggr) < \log\biggl( \frac{1}{4\delta(1-\delta)} \biggr).
    \end{align*}
    Therefore, combining Theorem~\ref{lemma:important-lower-bound} and Corollary~\ref{cor:high-prob-le-cam-kl}, we find that
    \begin{align*}
        \mathcal{M}(\delta,\mathcal{P}_{\Theta},\|\cdot\|_{\mathrm{op}}) \geq \frac{\|\Sigma^{(1)} - \Sigma^{(2)}\|_{\mathrm{op}}}{2} = \frac{(1-a^2)\sigma^2}{2} \geq \frac{\sigma^2}{4\sqrt{n}} \sqrt{\log\Bigl( \frac{1}{4\delta(1\!-\!\delta)} \Bigr) } \geq \frac{\sigma^2}{9} \sqrt{\frac{\log(1/\delta)}{n}}.
    \end{align*}
    On the other hand, when $\delta \in (0,e^{-n}/3)$, since $\mathcal{M}(\delta)$ is decreasing in $\delta$, we have
    \begin{align*}
        \mathcal{M}(\delta,\mathcal{P}_{\Theta},\|\cdot\|_{\mathrm{op}}) \geq \mathcal{M}(e^{-n},\mathcal{P}_{\Theta},\|\cdot\|_{\mathrm{op}}) \geq \frac{\sigma^2}{9}.
    \end{align*}
    Hence, for $\delta\in(0,1/4]$,
    \begin{align} \label{eq:covariance-matrix-delta-term}
        \mathcal{M}(\delta,\mathcal{P}_{\Theta},\|\cdot\|_{\mathrm{op}}) \gtrsim \sigma^2 \sqrt{\frac{\log(1/\delta)}{n}} \wedge \sigma^2.
    \end{align}

    Now assume that $r_0 \geq 20$ and define $\Phi \coloneqq \{0,1\}^{r_0}$. By the Gilbert--Varshamov lemma \citep[e.g.,][Lemma 4.7]{massart2007concentration}, there exists $\Phi_0 \subseteq \Phi$ with $|\Phi_0| \geq e^{r_0/8}$ and $d_{\mathrm{H}}(\phi,\phi') > r_0/4$ for all distinct $\phi,\phi'\in\Phi_0$, where $d_{\mathrm{H}}$ denotes the Hamming distance. Let $b\coloneqq \frac{\sigma^2}{6\sqrt{nr_0}}$, and for $\phi \in \Phi_0$, define $\Sigma^{(\phi)} \in \mathcal{S}_+^{d\times d}$ as well as $\Sigma^{(0)} \in \mathcal{S}^{d \times d}_{+}$ by
    \begin{align*}
        \Sigma^{(\phi)} \coloneqq \begin{pmatrix}
            \frac{\sigma^2}{2}I_{r_0} + b\phi\phi^\top & 0\\
            0 & 0
        \end{pmatrix} \quad\text{and}\quad 
        \Sigma^{(0)} \coloneqq \begin{pmatrix}
            \frac{\sigma^2}{2}I_{r_0} & 0\\
            0 & 0
        \end{pmatrix}.
    \end{align*}
    Then $\|\Sigma^{(\phi)}\|_{\mathrm{op}} = \frac{\sigma^2}{2} + b\|\phi\|_2^2 \leq \frac{\sigma^2}{2} + br_0 \leq \sigma^2$ and $\mathbf{r}(\Sigma^{(\phi)}) \leq r_0 \leq r$ for all $\phi\in\Phi_0$, so $\Sigma^{(\phi)} \in \Theta$.  For $\phi\in\Phi_0$, define $P^{(\phi)} \coloneqq \mathsf{N}_d(0,\Sigma^{(\phi)})^{\otimes n}$ so that $P^{(\phi)} \in \mathcal{P}_{\Sigma^{(\phi)}}$. Then, taking $Q \coloneqq \mathsf{N}_d(0,\Sigma^{(0)})^{\otimes n}$, we have for $\phi\in\Phi_0$ that
    \begin{align*}
        \mathrm{KL}(P^{(\phi)},Q) &= \frac{n}{2} \Biggl\{ \log\biggl( \frac{(\sigma^2/2)^{r_0}}{(b\|\phi\|_2^2 + \sigma^2/2)(\sigma^2/2)^{r_0-1}} \biggr) + \frac{2b\|\phi\|_2^2}{\sigma^2} \Biggr\}\\
        &= \frac{n}{2} \Biggl\{ \log\biggl( \frac{1}{1 + 2b\|\phi\|_2^2/\sigma^2} \biggr) + \frac{2b\|\phi\|_2^2}{\sigma^2} \Biggr\} \leq \frac{nb^2\|\phi\|_2^4}{\sigma^4} \leq \frac{r_0}{36}, 
    \end{align*}
    where the penultimate inequality follows because $\log(\frac{1}{1+x}) + x \leq x^2/2$ for $x\geq 0$ and the final inequality follows since $\|\phi\|_2^2 \leq r_0$. Therefore, 
    \begin{align} \label{eq:covariance-matrix-fano-term}
         \frac{|\Phi_0|^{-1} \sum_{\phi\in\Phi_0} \mathrm{KL}(P^{(\phi)},Q) + \log(2-|\Phi_0|^{-1})}{\log |\Phi_0|} \leq \frac{r_0/36 + \log(2)}{r_0/8} \leq \frac{1}{2},
    \end{align}
    where the final inequality holds under the assumption $r_0 \geq 20$.  Now fix distinct $\phi = (\phi_1,\ldots,\phi_{r_0})^\top$, $\phi' = (\phi_1',\ldots,\phi_{r_0}')^\top \in \Phi_0$, and define $S\coloneqq \bigl\{j\in[r_0] : \phi_j=1 \text{ and } \phi'_j=0\bigr\}$.  By reversing the roles of $\phi$ and $\phi'$ if necessary, we may assume that $|S| \geq r_0/8$ since $d_{\mathrm{H}}(\phi,\phi') > r_0/4$. Define $v=(v_1,\ldots,v_{r_0})^\top \in\{0,1\}^{r_0}$ by $v_j \coloneqq \mathbbm{1}_{\{j\in S\}}$ for $j\in[r_0]$. Then
    \begin{align} \label{eq:covariance-matrix-separation-term}
        \|\Sigma^{(\phi)} - \Sigma^{(\phi')}\|_{\mathrm{op}} \geq \frac{v^\top (\Sigma^{(\phi)} - \Sigma^{(\phi')}) v}{\|v\|_2^2} \geq \frac{b(r_0/8)^2}{r_0} \geq \frac{\sigma^2}{384} \sqrt{\frac{r_0}{n}}. 
    \end{align}
    Combining \eqref{eq:covariance-matrix-fano-term} and \eqref{eq:covariance-matrix-separation-term} with Theorem~\ref{lemma:important-lower-bound} and Lemma~\ref{lemma:fano's-minimax-quantile}, we deduce that for $\delta \in (0,1/2)$,
    \begin{align} \label{eq:covariance-matrix-expectation-term} 
        \mathcal{M}(\delta,\mathcal{P}_{\Theta},\|\cdot\|_{\mathrm{op}}) \geq \frac{\sigma^2}{768}\sqrt{\frac{r_0}{n}} \geq \frac{\sigma^2}{800}\sqrt{\frac{r}{n}}.
    \end{align}
    Hence, combining~\eqref{eq:covariance-matrix-delta-term} and~\eqref{eq:covariance-matrix-expectation-term} yields that for $r\geq 20$ and $\delta\in(0,1/4]$,
    \begin{align*}
        \mathcal{M}(\delta,\mathcal{P}_{\Theta},\|\cdot\|_{\mathrm{op}}) \gtrsim \sigma^2\sqrt{\frac{r}{n}} + \biggl( \sigma^2 \sqrt{\frac{\log(1/\delta)}{n}} \wedge \sigma^2 \biggr) \gtrsim \sigma^2 \sqrt{\frac{r+\log(1/\delta)}{n}} \wedge \sigma^2.
    \end{align*}
    On the other hand, for $r<20$ and $\delta\in(0,1/4]$, we have by \eqref{eq:covariance-matrix-delta-term} that
    \begin{align*}
        \mathcal{M}(\delta,\mathcal{P}_{\Theta},\|\cdot\|_{\mathrm{op}}) \gtrsim \sigma^2 \sqrt{\frac{\log(1/\delta)}{n}} \wedge \sigma^2 \gtrsim \sigma^2 \sqrt{\frac{r+\log(1/\delta)}{n}} \wedge \sigma^2.
    \end{align*}
    This proves the lower bound.

    For the upper bound, let $X_1, \ldots X_n \overset{\mathrm{iid}}{\sim} \mathsf{N}_d(0,\Sigma)$ where $\Sigma \in \Theta$. When $\delta\in[e^{-n},1/4]$, let $\hat{\Sigma} \coloneqq n^{-1} \sum_{i=1}^n X_iX_i^\top$. Then by \citet[Corollary~2]{koltchinskii2017concentration}, we have with probability at least $1-\delta$ that
    \begin{align*}
        \|\hat{\Sigma} - \Sigma\|_{\mathrm{op}} \lesssim \sigma^2\sqrt{\frac{r+\log(1/\delta)}{n}}.
    \end{align*}
    On the other hand, when $\delta \in (0,e^{-n})$, let $\hat{\Sigma} \coloneqq 0$, so that $\|\hat{\Sigma} - \Sigma\|_{\mathrm{op}} = \|\Sigma\|_{\mathrm{op}} \leq \sigma^2$. Thus
    \begin{align*}
        \mathcal{M}(\delta,\mathcal{P}_{\Theta},\|\cdot\|_{\mathrm{op}}) \lesssim \sigma^2\sqrt{\frac{r+\log(1/\delta)}{n}} \wedge \sigma^2,
    \end{align*}
    for $\delta\in(0,1/4]$.
\end{proof}

\subsection{Sparse linear regression}
 
In this section, we consider sparse linear regression with fixed design $X \in \mathbb{R}^{n \times d}$.  We will require the weighted restricted eigenvalue (WRE) condition of~\citet[][p.\ 3616]{bellec2018slope}.  To recall this, for $v = (v_1,\ldots,v_d)^\top \in \mathbb{R}^d$, let $v^\sharp = (v_1^\sharp,\ldots,v_d^\sharp)^\top \in [0,\infty)^d$ denote the vector obtained by ordering the absolute values of the components of $v$ in decreasing order.  For $j \in [d]$ and some $\sigma > 0$ (which will be taken to be the noise standard deviation), let $\lambda_j \coloneqq 6\sigma \sqrt{n^{-1}\log(2d/j)}$.  Now, for $c_0 > 0$, $s \in [d]$, define the cone
\[
\mathcal{C}_{\mathrm{WRE}}(s,c_0) \coloneqq \biggl\{v = (v_1,\ldots,v_d)^\top \in \mathbb{R}^d: \sum_{j=1}^d \lambda_j v_j^\sharp \leq (1+c_0)\|v\|_2\biggl(\sum_{j=1}^s \lambda_j^2\biggr)^{1/2}\biggr\}
\]
and let
\[
\vartheta(s,c_0) \coloneqq \min_{v \in \mathcal{C}_{\mathrm{WRE}}(s,c_0) \setminus \{0\}} \frac{\|Xv\|_2}{n^{1/2}\|v\|_2}.
\]
The following assumption imposes regularity conditions on our design matrix. 
\begin{assumption}\label{asm:design}
    Each column of $X \in \mathbb{R}^{n \times d}$ has Euclidean norm at most $\sqrt{n}$.  Moreover, there exist $c, C > 0$ such that when $d=1$, $X$ satisfies $\|X\|_2 \geq c\sqrt{n}$; when $d\geq 2$, $X$ satisfies $\vartheta(s \vee 2, 3) \geq c$ and $\max_{T\subseteq [d] : |T| \leq s\vee2} \|X_T\|_{\mathrm{op}} \leq C\sqrt{n}$, where $X_T \in \mathbb{R}^{n \times |T|}$ is the submatrix of~$X$ obtained by extracting the columns with indices in~$T$.
\end{assumption}
Existing results show that if $X\in\mathbb{R}^{n\times d}$ has independent $\mathsf{N}(0,1/2)$ entries, then there exist universal constants $c', C' > 0$ such that, when $d\geq 2$, Assumption~\ref{asm:design} holds with $c=1/2$ and $C = 2$ with probability at least $1-e^{-c'n}$ when $n \geq C's\log(2ed/s)$.  Indeed, the verification of the condition $\vartheta(s \vee 2, 3) \geq 1/2$ is given by~\citet[Theorem~8.3]{bellec2018slope}, and verification of the condition $\max_{T\subseteq [d] : |T| \leq s\vee2} \|X_T\|_{\mathrm{op}} / \sqrt{n} \leq 2$ can be obtained via concentration results for Gaussian ensembles~\citep[e.g.,][Theorem~6.1]{wainwright2019high}.

We note that the matching upper bound in Proposition~\ref{Prop:SparseLR} below is attained by the SLOPE estimator~\citep{bogdan2015slope,bellec2018slope}, which is $\delta$-independent.
\begin{prop}
\label{Prop:SparseLR}
    Let $n, d\in\mathbb{N}$, $s\in[d]$, $\delta\in(0,1/4]$, $\sigma>0$ and $\Theta \coloneqq \bigl\{ \theta \in \mathbb{R}^d : \|\theta\|_0 \leq s \bigr\}$.  Assume that $X \in \mathbb{R}^{n \times d}$ satisfies Assumption~\ref{asm:design} for some $c,C > 0$, and for $\theta \in \Theta$, let $\mathcal{P}_{\theta} \coloneqq \bigl\{\mathsf{N}_n(X\theta, \sigma^2 I_n)\bigr\}$.  
    Then, setting $L(\theta,\theta') \coloneqq n^{-1}\|X(\theta - \theta')\|_2^2$, we have
    \begin{align*}
        \mathcal{M}(\delta,\mathcal{P}_{\Theta},L) \asymp_{c,C} \sigma^2 \cdot \frac{s\log(ed/s) + \log(1/\delta)}{n}.
    \end{align*}
\end{prop}



\begin{proof}
Let $\theta_1 \coloneqq 0 \in \mathbb{R}^d$ and $\theta_2 \coloneqq e_1 \sqrt{\frac{\sigma^2}{n} \log \bigl( \frac{1}{4\delta(1-\delta)}} \bigr) \in \mathbb{R}^d$, so that $\theta_1,\theta_2 \in \Theta$. For $\ell\in\{1,2\}$, let $P_{\ell} \coloneqq \mathsf{N}_n(X\theta_{\ell}, \sigma^2 I_n)$. Then $\mathrm{KL}(P_1,P_2) = \frac{1}{2\sigma^2} \|X\theta_2\|_2^2 \leq \frac{n}{2\sigma^2}\|\theta_2\|_2^2 < \log \bigl( \frac{1}{4\delta(1-\delta)} \bigr)$. Therefore, combining Theorem~\ref{lemma:important-lower-bound}, Assumption~\ref{asm:design} (since $\theta_1 - \theta_2 \in \mathcal{C}_{\mathrm{WRE}}(s\vee2, 3)$ when $d\geq 2$) and Corollary~\ref{cor:high-prob-le-cam-kl}, we have for $\delta \in (0,1/4]$ that
    \begin{align} \label{eq:sparse-linear-regression-delta-term}
        \mathcal{M}(\delta,\mathcal{P}_{\Theta},L) \geq \frac{1}{4n}\|X(\theta_1 - \theta_2)\|_2^2 \geq \frac{c^2}{4}\|\theta_1 - \theta_2\|_2^2 \geq \frac{c^2\sigma^2\log(1/\delta)}{20n}.
    \end{align}

    To capture the first term in the lower bound, let us first consider the case where $d \geq 4e$.  By Lemma~\ref{lemma:GilbertVarshamov}, there exist $M \in \mathbb{N}$ with 
    \[
    \log M \geq \frac{3\bigl(1 \vee \lfloor s/2\rfloor \wedge \lfloor d/(4e) \rfloor\bigr)}{4} \log\biggl(\frac{d}{4\bigl(1 \vee \lfloor s/2\rfloor \wedge \lfloor d/(4e) \rfloor\bigr)} \biggr)
    \]
    and $\mathcal{V} \coloneqq \bigl\{ v^{(1)},\ldots,v^{(M)} \bigr\}$ such that $\mathcal{V}$ forms a Euclidean norm $(1/2)$-packing of the set $\bigl\{ \theta \in \mathbb{R}^d : \|\theta\|_0 \leq 1 \vee \lfloor s/2 \rfloor \wedge \lfloor d/(4e) \rfloor,\, \|\theta\|_2 \leq 1 \bigr\}$. For $m\in[M]$, define $\theta^{(m)} \coloneqq v^{(m)} \sqrt{\frac{\sigma^2 \log M}{4C^2n}}$, so that $\theta^{(m)} \in \Theta$, and define $P^{(m)} \coloneqq \mathsf{N}_n(X\theta^{(m)}, \sigma^2 I_n)$. Further define $Q\coloneqq \mathsf{N}_n(0, \sigma^2 I_n)$ so that $\mathrm{KL}(P^{(m)}, Q) = \frac{1}{2\sigma^2}\|X\theta^{(m)}\|_2^2 \leq \frac{\log M}{8}$ for $m \in [M]$.  Since $\log M \geq 3/4$ and since $M \mapsto \log(2-M^{-1})/\log M$ is decreasing for $M \geq e^{3/4}$, we have 
    \begin{align*}
        \frac{M^{-1}\sum_{m=1}^M \mathrm{KL}(P^{(m)}, Q) + \log(2-M^{-1})}{\log M} \leq \frac{1}{8} + \frac{\log(2-M^{-1})}{\log M} \leq \frac{7}{10}.
    \end{align*}
    Since $\theta^{(m)} - \theta^{(m')} \in \mathcal{C}_{\mathrm{WRE}}(s \vee 2, 3)$ for distinct $m, m' \in [M]$, we have by Lemma~\ref{lemma:fano's-minimax-quantile}, that for $\delta\in(0,1/4]$,
    \begin{align}
        \mathcal{M}(\delta,\mathcal{P}_{\Theta},L) &\geq \frac{1}{4n}\cdot \min_{m\neq m'}\|X(\theta^{(m)} - \theta^{(m')})\|_2^2 \geq \frac{c^2}{4} \min_{m\neq m'}\|\theta^{(m)} - \theta^{(m')}\|_2^2 \nonumber\\
        &> \frac{c^2\sigma^2\{s/4 \wedge 3d/(32e)\}\log\bigl(\frac{d}{4s} \bigr)}{2^6 C^2n} \geq \frac{c^2\sigma^2s\log(ed/s)}{2^{13}C^2n}. \label{eq:sparse-linear-regression-expectation-term-1}
    \end{align}
From~\eqref{eq:sparse-linear-regression-delta-term} and \eqref{eq:sparse-linear-regression-expectation-term-1}, we deduce that for $d \geq 4e$ and $\delta\in(0,1/4]$,
\[
\mathcal{M}(\delta,\mathcal{P}_{\Theta},L) \geq \frac{c^2\sigma^2\log(1/\delta)}{40n} + \frac{c^2\sigma^2s\log(ed/s)}{2^{14}C^2n}.
\]
On the other hand, when $d < 4e$, we have by~\eqref{eq:sparse-linear-regression-delta-term} that for $\delta\in(0,1/4]$,
    \begin{align*}
        \mathcal{M}(\delta,\mathcal{P}_{\Theta},L) \geq \frac{c^2\sigma^2\log(1/\delta)}{20n} \geq c^2\sigma^2\cdot \frac{s\log(ed/s) + \log(1/\delta)}{200n}.
    \end{align*}
    This proves the lower bound.

    The upper bound is achieved, for example, by the SLOPE estimator \citep{bogdan2015slope}, with a $\delta$-independent tuning parameter. In particular, applying \citet[Eq.\ (6.2) and (6.5)]{bellec2018slope} with $\gamma = \tau = 1/4$ in their notation yields that for any $\delta \in (0,1)$,
    \[
    \mathcal{M}(\delta,\mathcal{P}_{\Theta},L) \leq 100(4+\sqrt{2})\sigma^2\biggl(\frac{s\log(ed/s)}{c^2n} + \frac{\log(1/\delta)}{n}\biggr),
    \]
    as required.
\end{proof}

\subsection{Nonparametric density estimation}

Turning to a nonparametric statistical problem, we now consider density estimation at a point.  
To fix the setting, let $\beta,\gamma > 0$ and $m \coloneqq \lceil \beta \rceil - 1$. The \emph{H\"{o}lder class} $\mathcal{F}(\beta,\gamma)$ of density functions is the set of all Borel measurable functions $f: \mathbb{R} \to [0,\infty)$ satisfying that $\int_{-\infty}^{\infty} f(x)\,\mathrm{d}x = 1$, that $f$ is $m$-times differentiable and that
\begin{align*}
    |f^{(m)}(x) - f^{(m)}(x')| \leq \gamma |x-x'|^{\beta-m},
\end{align*}
for all $x,x'\in\mathbb{R}$.

Proposition~\ref{prop:KDE} determines an explicit expression for the minimax quantiles, up to multiplicative factors depending only on $\beta$.  We establish, using Lepski's method, that the upper bound is attainable by a $\delta$-independent estimator in Lemma~\ref{lemma:kde-ub-lepski}. 
\begin{prop}\label{prop:KDE}
    Let $n\in\mathbb{N}$, $\delta\in(0,1/4]$, $\Theta\coloneqq \mathcal{F}(\beta, \gamma)$ be the H\"{o}lder class of density functions and for $f\in\Theta$, define $\mathcal{P}_f \coloneqq \{P_f^{\otimes n}\}$ where $P_f$ is the distribution with density $f$. Fix $x_0 \in \mathbb{R}$ and let $L(f_1,f_2) \coloneqq \{f_1(x_0) - f_2(x_0)\}^2$. Then
    \begin{align*}
        \mathcal{M}(\delta, \mathcal{P}_{\Theta}, L) \asymp_{\beta}  \gamma^{2/(\beta+1)} \cdot \biggl\{ \biggl( \frac{\log(1/\delta)}{n} \biggr)^{2\beta/(2\beta+1)} \wedge 1\biggr\}.
    \end{align*}
\end{prop}
\begin{remark}
\citet[][Corollary~2 and Lemma~16]{polyanskiy2021dualizing} obtain that for the subclass $\mathcal{F}_{[-1,1]}(\beta,\gamma)$ of $\mathcal{F}(\beta,\gamma)$ consisting of densities supported on $[-1,1]$, and for $\delta \in [e^{-2n},1/16]$, we have
\[
\mathcal{M}(\delta) \asymp_{\beta,\gamma} \biggl( \frac{\log(1/\delta)}{n} \biggr)^{2\beta/(2\beta+1)},
\]
with the upper bound attained by a $\delta$-dependent estimator.
\end{remark}
\begin{proof}
By the translation invariance of $\mathcal{F}(\beta,\gamma)$, the minimax quantile $\mathcal{M}(\delta,\mathcal{P}_\Theta,L)$ is the same for all $x_0\in\mathbb{R}$, so assume without loss of generality that $x_0 = 0$. 
Let $m \coloneqq \lceil\beta\rceil - 1$ and take $K: \mathbb{R} \rightarrow [0,\infty)$ to be the infinitely differentiable density given by $K(x)\propto e^{-1/(1-4x^2)}\mathbbm{1}_{\{|x| < 1/2\}}$, which satisfies 
\[
\|K\|_{C^\beta} \coloneqq \sup_{x \neq x'}\frac{|K^{(m)}(x) - K^{(m)}(x')|}{|x-x'|^{\beta - m}} < \infty.
\]
Define $\tilde{\gamma}\coloneqq \gamma/\|K\|_{C^\beta}$, and for $\lambda\in (0,1]$ that will be chosen later to depend on $n$, $\beta$ and $\delta$, define $g,h:\mathbb{R} \rightarrow [0,\infty)$ by
\begin{equation}
\label{Eq:gh}
g(x) \coloneqq \tilde{\gamma}^{1/(\beta+1)}K\bigl(\tilde{\gamma}^{1/(\beta+1)}x\bigr),\qquad h(x) \coloneqq \lambda \tilde{\gamma}^{1/(\beta+1)}K\bigl(\lambda^{-1/\beta}\tilde{\gamma}^{1/(\beta+1)}x\bigr).
\end{equation}
Then $g \in \mathcal{F}(\beta,\gamma)$ because it is $m$-times differentiable, and for $x,x' \in \mathbb{R}$,
 \begin{align*}
 \bigl|g^{(m)}(x) - g^{(m)}(x')\bigr| &= \tilde{\gamma}^{\frac{m+1}{\beta+1}}\bigl|K^{(m)}\bigl(\tilde{\gamma}^{1/(\beta+1)}x\bigr) - K^{(m)}\bigl(\tilde{\gamma}^{1/(\beta+1)}x'\bigr)\bigr| \\
 &\leq \tilde{\gamma}^{\frac{m+1}{\beta+1}}\|K\|_{C^\beta}\bigl|\tilde{\gamma}^{1/(\beta+1)}x - \tilde{\gamma}^{1/(\beta+1)}x'\bigr|^{\beta-m} = \gamma|x-x'|^{\beta-m}.
 \end{align*}
Similarly, $\bigl|h^{(m)}(x) - h^{(m)}(x')\bigr| \leq \gamma|x-x'|^{\beta-m}$ for $x,x' \in \mathbb{R}$, and $\int_{-\infty}^\infty h = \lambda^{(\beta+1)/\beta}$.  Moreover,
\[
h(x) \leq \lambda \tilde{\gamma}^{1/(\beta+1)} K\bigl(\tilde{\gamma}^{1/(\beta+1)}x\bigr) = \lambda g(x)
\]
for all $x \in \mathbb{R}$.  Now define $f_0,f_1: \mathbb{R} \rightarrow [0,\infty)$ by
\begin{align}
\label{Eq:f0f1}
f_0(x) &\coloneqq \frac{(g+h)(x)+g(2\tilde{\gamma}^{-1/(\beta+1)}-x)}{2+\lambda^{(\beta+1)/\beta}}, \\ \label{Eq:f0f12}
f_1(x) &\coloneqq\frac{g(x)+(g+h)(2\tilde{\gamma}^{-1/(\beta+1)}-x)}{2+\lambda^{(\beta+1)/\beta}};
\end{align}
see Figure~\ref{Fig:LeCam}.
\begin{figure}
    \centering
    \includegraphics[width=\textwidth]{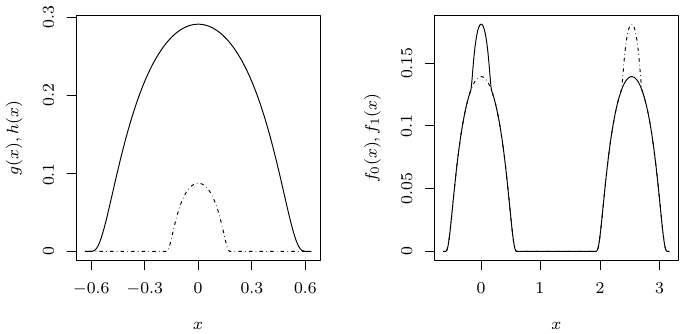}
    \caption{\label{Fig:LeCam}Illustration of the construction in Proposition~\ref{prop:KDE}.  Left: The functions $g$ (solid) and $h$ (dot-dashed) from~\eqref{Eq:gh}.  Right: The functions $f_0$ (solid) and $f_1$ (dot-dashed) from~\eqref{Eq:f0f1} and~\eqref{Eq:f0f12}.}
\end{figure}
Since $g,h$ are zero outside the interval $[-\tilde{\gamma}^{-1/(\beta+1)}/2,\tilde{\gamma}^{-1/(\beta+1)}/2]$, we deduce that $f_0,f_1\in\mathcal{F}(\beta,\gamma)$ and since $K(0) \geq 1$,
\[
f_0(0)-f_1(0)=\frac{h(0)}{2+\lambda^{(\beta+1)/\beta}}=\frac{\lambda \tilde{\gamma}^{1/(\beta+1)} K(0)}{2+\lambda^{(\beta+1)/\beta}}\geq\frac{\lambda \tilde{\gamma}^{1/(\beta+1)}}{3}.
\]
In addition, writing $I\coloneqq[-\lambda^{1/\beta}\tilde{\gamma}^{-1/(\beta+1)}/2,\lambda^{1/\beta}\tilde{\gamma}^{-1/(\beta+1)}/2]$ for the support of $h$ and $I' \coloneqq 2\tilde{\gamma}^{-1/(\beta+1)} + I$, we have $f_0 = f_1$ on $\mathbb{R}\setminus (I\cup I')$. Thus, since $\log(1+\lambda) \leq \lambda$, it follows that the product densities $\bigl\{f_j^n:j \in \{0,1\}\bigr\}$ given by $f_j^n(x_1,\ldots,x_n) \coloneqq \prod_{i=1}^n f_j(x_i)$ satisfy
\begin{align*}
\mathrm{KL}(f_0^n,f_1^n) &= n\mathrm{KL}(f_0,f_1) = n\int_{I\cup I'} f_0 \log \Bigl(\frac{f_0}{f_1}\Bigr) \\
&= \frac{n}{2+\lambda^{(\beta+1)/\beta}}\int_I \,\biggl\{(g+h)\log \Bigl(\frac{g+h}{g}\Bigr) + g \log\Bigl(\frac{g}{g+h}\Bigr)\biggr\} \\
&=\frac{n}{2+\lambda^{(\beta+1)/\beta}}\int_I h\log\biggl(1+\frac{h}{g}\biggr)\\
&\leq\frac{n}{2}\log(1+\lambda)\int_I h \leq \frac{n}{2}\lambda^{(2\beta+1)/\beta}.
\end{align*}
Taking $\lambda = \bigl\{\frac{1}{n}\log\bigl(\frac{1}{4\delta(1-\delta)}\bigr)\bigr\}^{\beta/(2\beta+1)} \wedge 1 \in (0,1]$, we have that $\mathrm{KL}(f_0^n,f_1^n) \leq \frac{1}{2}\log\bigl(\frac{1}{4\delta(1-\delta)}\bigr) < \log\bigl(\frac{1}{4\delta(1-\delta)}\bigr)$, so by Theorem~\ref{lemma:important-lower-bound} and Corollary~\ref{cor:high-prob-le-cam-kl}, for $\delta \in (0,1/4]$,
\begin{align*}
\mathcal{M}(\delta, \mathcal{P}_{\Theta}, L) \geq \mathcal{M}_-(\delta, \mathcal{P}_{\Theta}, L) &\geq \biggl(\frac{\lambda\tilde{\gamma}^{1/(\beta+1)}}{2 \cdot 3}\biggr)^2 \\
&\geq \frac{\gamma^{2/(\beta+1)}}{36 \cdot 5^{2\beta/(2\beta+1)}\|K\|_{C^\beta}^{2/(\beta+1)}} \biggl\{ \biggl( \frac{\log(1/\delta)}{n} \biggr)^{2\beta/(2\beta+1)} \wedge 1\biggr\}.
\end{align*}
This completes the proofs of the lower bound.  The upper bound follows from Lemma~\ref{lemma:kde-ub-lepski}, after fixing a bounded kernel supported on $[-1,1]$.  
\end{proof}

\subsection{Isotonic regression}

Other natural nonparametric estimation problems are provided by shape-constrained models \citep{groeneboom2014nonparametric,samworth2018editorial}, and here we consider isotonic regression.  The matching upper bound is attained by the projected isotonic least squares estimator, which is $\delta$-independent.
\begin{prop}\label{prop:isotonic-book}
    Let $n\geq 2$, $\delta\in(0,1/4]$, $\Theta \coloneqq \bigl\{ \theta=(\theta_1,\ldots,\theta_n)^\top \in[0,1]^n : \theta_1\leq \cdots \leq \theta_n \bigr\}$ and $\mathcal{P}_{\theta} \coloneqq \bigl\{ \mathsf{N}_n(\theta, I_n) \bigr\}$ for $\theta \in \Theta$. Then
    \begin{align*}
        \mathcal{M}(\delta,\mathcal{P}_{\Theta}, n^{-1}\|\cdot\|_2^2) \asymp \biggl( \frac{1}{n^{2/3}} + \frac{\log(1/\delta)}{n} \biggr) \wedge 1.
    \end{align*}
\end{prop}

\begin{proof}
    First consider the case where $\delta\in[e^{-n}/3, 1/4]$.
    Let $\theta^{(1)} \coloneqq 0$, $\theta^{(2)} \coloneqq \sqrt{\frac{1}{n}\log\bigl(\frac{1}{4\delta(1-\delta)}\bigr)} \,\bm{1}_n$ and $P^{(\ell)} \coloneqq \mathsf{N}_n(\theta^{(\ell)}, I_n)$ for $\ell\in\{1,2\}$. Then $\theta^{(\ell)} \in \Theta$ and $P^{(\ell)} \in \mathcal{P}_{\theta^{(\ell)}}$ for $\ell\in\{1,2\}$ and 
    \begin{align*}
        \mathrm{KL}(P^{(1)},P^{(2)}) = \frac{1}{2}\log\biggl(\frac{1}{4\delta(1-\delta)}\biggr) < \log\biggl( \frac{1}{4\delta(1-\delta)} \biggr).
    \end{align*}
    Therefore, by Corollary~\ref{cor:high-prob-le-cam-kl}, 
    \begin{align} \label{eq:isotonic-regression-delta-term1}
        \mathcal{M}(\delta,\mathcal{P}_{\Theta}, n^{-1}\|\cdot\|_2^2) \geq \frac{\|\theta^{(1)} - \theta^{(2)}\|_2^2}{4n} \geq \frac{\log(1/\delta)}{20n},
    \end{align}
    for $\delta\in[e^{-n}/3, 1/4]$. Next, for the case where $\delta\in(0, e^{-n}/3)$, we have by~\eqref{eq:isotonic-regression-delta-term1} that
    \begin{align}\label{eq:isotonic-regression-delta-term2}
        \mathcal{M}(\delta,\mathcal{P}_{\Theta}, n^{-1}\|\cdot\|_2^2) \geq \mathcal{M}(e^{-n},\mathcal{P}_{\Theta}, n^{-1}\|\cdot\|_2^2) \geq \frac{1}{20}.
    \end{align}
    Combining~\eqref{eq:isotonic-regression-delta-term1} and~\eqref{eq:isotonic-regression-delta-term2} yields
    \begin{align}\label{eq:isotonic-regression-delta-term}
        \mathcal{M}(\delta,\mathcal{P}_{\Theta}, n^{-1}\|\cdot\|_2^2) \geq \frac{1}{20} \biggl( \frac{\log(1/\delta)}{n} \wedge 1 \biggr),
    \end{align}
    for $\delta\in(0,1/4]$.

    For the second part of the lower bound, we let $\Theta_0 \subseteq \Theta$ denote the finite set constructed in the proof of \citet[][Theorem~5.3]{chatterjee2015risk} to obtain
    \begin{align*}
        \inf_{\hat{\theta} \in \hat{\Theta}} \max_{\theta_0\in\Theta_0} \sup_{P_{\theta_0} \in \mathcal{P}_{\theta_0}} \mathbb{E}_{P_{\theta_0}} \bigl( n^{-1} \|\hat{\theta} - \theta_0\|_2^2\bigr) \geq \frac{1}{2^9 \cdot 3^{1/3}} \cdot n^{-2/3}.
    \end{align*}
    In this construction, $D \coloneqq \sup_{\theta,\theta' \in \Theta_0} \| \theta - \theta' \|_2 \in \bigl[\frac{n^{1/6}}{2^3\cdot 3^{1/6}}, n^{1/6}\bigr]$.  Hence by Theorem~\ref{lemma:expectation-lb-to-high-prob-lb} with $\epsilon = 2^{-6}$, we obtain that for $\delta\in(0,1/1000]$,
    \begin{align*} 
        \mathcal{M}_- (\delta,\mathcal{P}_{\Theta},n^{-1}\|\cdot\|_2^2) \geq \frac{1}{2^{18} \cdot 3^{1/3}}\cdot n^{-2/3}.
    \end{align*}  
    Then, applying Proposition~\ref{cor:boost-delta} with $k=4$ in conjunction with Theorem~\ref{lemma:important-lower-bound}, we deduce that for $\delta\in(0,1/4]$, 
    \begin{align} \label{eq:isotonic-regression-expectation-term-boosted}
        \mathcal{M} (\delta,\mathcal{P}_{\Theta},n^{-1}\|\cdot\|_2^2) \geq \mathcal{M}_- (\delta,\mathcal{P}_{\Theta},n^{-1}\|\cdot\|_2^2) \geq  \frac{1}{2^{26} \cdot 3^{3}}\cdot n^{-2/3}.
    \end{align}
    Combining \eqref{eq:isotonic-regression-delta-term} and \eqref{eq:isotonic-regression-expectation-term-boosted} proves the lower bound.
  

    For the upper bound, let $\hat{\theta}$ denote the isotonic least squares estimator.  By \citet[Corollary~3.3]{bellec2018sharp}, for every $\delta \in (0,1]$, we have with probability at least $1-\delta$ that 
    \begin{align*}
        n^{-1}\|\hat{\theta} - \theta\|_2^2 \lesssim \frac{1}{n^{2/3}} + \frac{\log(1/\delta)}{n}.
    \end{align*}
    Now define $\tilde{\theta} = (\tilde{\theta}_1,\ldots,\tilde{\theta}_n)^\top \in \mathbb{R}^n$ by $\tilde{\theta}_i \coloneqq 0 \vee \hat{\theta}_i \wedge 1$ for $i\in[n]$. Then, since $n^{-1}\|\tilde{\theta} - \theta\|_2^2 \leq n^{-1}\|\hat{\theta} - \theta\|_2^2$ and $n^{-1}\|\tilde{\theta} - \theta\|_2^2 \leq 1$, we deduce that 
    \begin{align*}
        \mathcal{M}(\delta,\mathcal{P}_{\Theta}, n^{-1}\|\cdot\|_2^2) \lesssim \biggl( \frac{1}{n^{2/3}} + \frac{\log(1/\delta)}{n} \biggr) \wedge 1
    \end{align*}
    for $\delta\in(0,1]$.
\end{proof}

\subsection{Stochastic convex optimisation} \label{sec:SCO}
As our final example, we move away from the traditional setting of statistical estimation and consider stochastic convex optimisation \citep{duchi2018introductory}, where the goal is to understand the iteration complexity of minimising a convex function in expectation, given access only to queries at random observations.  To do so, we require some preliminaries.  Let $R > 0$, let $\mathcal{X} \subseteq \mathbb{R}^d$ be closed and convex such that $\mathbb{B}_2^d(R) \subseteq \mathcal{X}$, and let $\mathcal{Y}$ be a measurable space.  For $P \in \mathcal{Q}(\mathcal{Y})$ and $\gamma > 0$, let $\mathcal{F}(P,\gamma)$ denote the set of functions $f: \mathbb{R}^d \times \mathcal{Y} \rightarrow \mathbb{R}$ for which 
\begin{enumerate}
    \item[\emph{(i)}]  $f(x,\cdot): \mathcal{Y} \to \mathbb{R}$ is $P$-integrable for each $x \in \mathcal{X}$;
    \item[\emph{(ii)}] $f(\cdot,y): \mathbb{R}^d \to \mathbb{R}$ is convex for each $y \in \mathcal{Y}$ with subdifferential at $x \in \mathbb{R}^d$ denoted by $\partial_x f(x,y)$;
    \item[\emph{(iii)}] for each $x \in \mathcal{X}$, define the function $g_x:\mathcal{Y} \rightarrow \mathbb{R}^d$ by $g_x(y) \coloneqq \argmin_{x' \in \partial_x f(x, y)} \| x' \|_2$, and assume that $\sup_{x\in\mathcal{X}} \mathbb{E}_P\bigl\{ \|g_x(Y)\|_2^2 \bigr\} \leq \gamma^2$.
\end{enumerate}
The selected subgradient $g_x$ in part~\emph{(iii)} of this definition is well-defined as the subdifferential is compact and convex \citep[][Theorem~23.4 and p.~215]{rockafellar1997convex}, and moreover $g_x$ is measurable on $\mathcal{Y}$ for each $x \in \mathcal{X}$, by~\citet[][Theorem~18.19]{aliprantis2007infinite}.  For $f \in \mathcal{F}(P,\gamma)$, define $F_P: \mathcal{X} \rightarrow \mathbb{R}$ by $F_P(x) \coloneqq \mathbb{E}_{P}\{f(x; Y)\}$, so that $F_P$ is convex on $\mathcal{X}$.  The \emph{optimality gap} $L^{\mathsf{opt}}(\cdot; P, f): \mathcal{X} \rightarrow [0, \infty)$ is defined to be $L^{\mathsf{opt}}(x; P, f) \coloneqq F_P(x) - \inf_{x' \in \mathcal{X}} F_P(x')$. For $P_1,P_2 \in \mathcal{Q}(\mathcal{Y})$ and $f \in \mathcal{F}(P_1,\gamma) \cap \mathcal{F}(P_2,\gamma)$, the \emph{$f$-separation} between $P_1$ and $P_2$ is defined as 
\begin{align*}
    \mathsf{sep}(P_1, P_2; f) \coloneqq \sup\Bigl\{\Delta \geq 0:\; & L^{\mathsf{opt}}(x; P_2, f) > \Delta \text{ for all $x \in \mathcal{X}$ with } L^{\mathsf{opt}}(x; P_1, f) \leq \Delta\Bigr\}. 
\end{align*}
Now define the set 
\begin{align*}
\mathcal{G}(\gamma,R) \coloneqq \Bigl\{(P,f) : P \in \mathcal{Q}(\mathcal{Y}),\, f \in \mathcal{F}(P,\gamma)\text{ and } \| x^{\star} \|_2 \leq R \text{ for some } x^{\star} \in \argmin_{x' \in \mathcal{X}} F_P(x')\Bigr\}.
\end{align*}
For each $t \in \mathbb{N}$, let $\hat{\mathcal{X}}_t$ denote the set of all measurable functions from $\mathcal{Y}^t$ to $\mathcal{X}$; in particular, this includes methods such as stochastic (sub)gradient descent (SGD) or stochastic versions of Newton's algorithm.  For $\delta \in (0,1]$ and $T \in \mathbb{N}$, the minimax quantile $\mathcal{M}^{\mathsf{opt}}$ for the optimality gap~\citep[cf.][Eq.~(1)]{carmon2024price} is given by
\begin{align*}
    \mathcal{M}^{\mathsf{opt}}(\delta;T) \coloneqq \inf_{\hat{x}_{\delta} \in \hat{\mathcal{X}}_T} \sup_{(P, f) \in \mathcal{G}(\gamma, R)} \mathrm{Quantile}\bigl(1 - \delta; L^{\mathsf{opt}}(\hat{x}_{\delta}; P,f), P^{\otimes T}\bigr).
\end{align*}
By a virtually identical argument to that given in the proof of Theorem~\ref{lemma:important-lower-bound}, 
\begin{align}\label{ineq:M-opt-lb}
    \mathcal{M}^{\mathsf{opt}}(\delta; T) \geq \inf\Bigl\{r \in [0, \infty):  \inf_{\hat{x}_{\delta} \in \hat{\mathcal{X}}_T} \sup_{(P, f) \in \mathcal{G}(\gamma, R)} P^{\otimes T}\bigl\{L^{\mathsf{opt}}(\hat{x}_{\delta};P,f) > r\bigr\} \leq \delta \Bigr\} \eqqcolon \mathcal{M}_{-}^{\mathsf{opt}}(\delta; T).
\end{align}
We next present a variant of Lemma~\ref{lemma:high-prob-le-cam-tv}, specialised to the current optimisation setting.  It can also be understood as a high-probability variant of, e.g.,~\citet[][Eq. (5.2.3)]{duchi2018introductory}.

\begin{lemma}\label{lem:SCO-Le-Cam}
Let $T \in \mathbb{N}$, $\gamma, R > 0$ and $\delta\in(0,1/2)$.  Suppose $(P_1,f), (P_2,f) \in \mathcal{G}(\gamma, R)$ and $\mathrm{TV}(P_1^{\otimes T}, P_2^{\otimes T}) < 1 - 2\delta$.  Then $\mathcal{M}^{\mathsf{opt}}_{-}(\delta; T) \geq \mathsf{sep}(P_1, P_2; f)$.
\end{lemma}
\begin{proof}
    Given $\epsilon > 0$, choose $r \in [0,\infty)$ such that $r > \mathsf{sep}(P_1, P_2; f) - \epsilon$ and such that for every $x \in \mathcal{X}$ we have $L^{\mathsf{opt}}(x; P_2, f) > r$ whenever $L^{\mathsf{opt}}(x; P_1, f) \leq r$.  Then
    \begin{align*}
        \inf_{\hat{x}_{\delta} \in \hat{\mathcal{X}}_T} \sup_{(P, f) \in \mathcal{G}(\gamma, R)} &P^{\otimes T}\bigl\{L^{\mathsf{opt}}(\hat{x}_{\delta};P,f) > r\bigr\} \\
        &\geq \frac{1}{2} \inf_{\hat{x}_{\delta} \in \hat{\mathcal{X}}_T} \Bigl\{P_1^{\otimes T}\bigl\{L^{\mathsf{opt}}(\hat{x}_{\delta};P_1,f) > r\bigr\} + P_2^{\otimes T}\bigl\{L^{\mathsf{opt}}(\hat{x}_{\delta};P_2,f) > r\bigr\}\Bigr\}\\
        &\geq \frac{1}{2} \inf_{\hat{x}_{\delta} \in \hat{\mathcal{X}}_T} \Bigl\{ 1 - P_1^{\otimes T}\bigl\{L^{\mathsf{opt}}(\hat{x}_{\delta};P_2,f) > r\bigr\} + P_2^{\otimes T}\bigl\{L^{\mathsf{opt}}(\hat{x}_{\delta};P_2,f) > r\bigr\}\Bigr\}\\
        &\geq \frac{1}{2} \bigl\{1 - \mathrm{TV}(P_1^{\otimes T}, P_2^{\otimes T})\bigr\} > \delta.
    \end{align*}
    We deduce that $\mathcal{M}_{-}(\delta,T) \geq r > \mathsf{sep}(P_1, P_2; f) - \epsilon$.  Since $\epsilon > 0$ was arbitrary, the result follows.
\end{proof}
\begin{figure}
    \begin{center}
\begin{tikzpicture}{scale=0.9}
\node[text width=60pt] at (7.65,3.9) {{\footnotesize gradient $\gamma s$}};
  \draw [-latex] (7.3,4.1) to [out=70,in=300] (7.18,4.5);
\node[text width=60pt] at (8.7,7.2) {{\footnotesize gradient $\gamma$}};
\draw [-latex] (8.55,7) to [out=290,in=160] (9.05,6.7);
\begin{axis}[
    legend pos=north west,
    axis lines = left,
    x label style={at={(axis description cs:1.04,0.084)},anchor=north},
    xlabel = \(x\),
    xmin = -2.4, xmax = 2.4,
    ymin = -0.3, ymax = 4.9,
    axis x line shift=-0.3,
    yticklabel={\empty},
    xticklabel={\empty},
    extra x ticks={-1.5, 0, 1.5},
    extra x tick labels={$-R$, $0$, $R$},
    extra y ticks = {0, 1.5},
    extra y tick labels ={$0$, $\gamma s R$},
    legend style={at={(0.12,1)},anchor=north west},
]
\addplot [
    domain=-2.5:2.5, 
    samples=1000, 
    color=black,
    line width = 1.25pt,
]
{(x < - 1.5) * (-2)*x + (-1.5 < x) * (x < 1.5) * (-1) * (x - 1.5) + (x > 1.5) * (x - 1.5)};
\addlegendentry{$L^{\mathsf{opt}}(x; P_2, f)$}

\addplot [
    domain=-2.5:2.5, 
    samples=1000, 
    color=gray,
    line width = 1.25pt,
]
{(x < - 1.5) * (-1) * (x + 1.5) + (-1.5 < x) * (x < 1.5)  * (x + 1.5) + (x > 1.5) * 2 * x};
\addlegendentry{$L^{\mathsf{opt}}(x; P_1, f)$}
\addplot [
    domain=-2.5:2.5, 
    samples=1000, 
    color=gray,
    line width = 1.25pt,
    dashed,
]
{1.5};
\addlegendentry{$\mathsf{sep}(P_1, P_2;f)$}
\end{axis}
\draw[-Latex,thick] (-0.1,0.5)--(10.48,0.5);
\draw[-Latex,thick] (0,-0.1)--(0,8.82);
\end{tikzpicture}
    \end{center}
    \caption{\label{Fig:Sep} Illustration of the functions in the proof of Proposition~\ref{prop:SCO-rate} with $f:\mathcal{X} \times \{-1,1\}$ given by $f(x,y) = \gamma|x+yR|$.}
\end{figure}
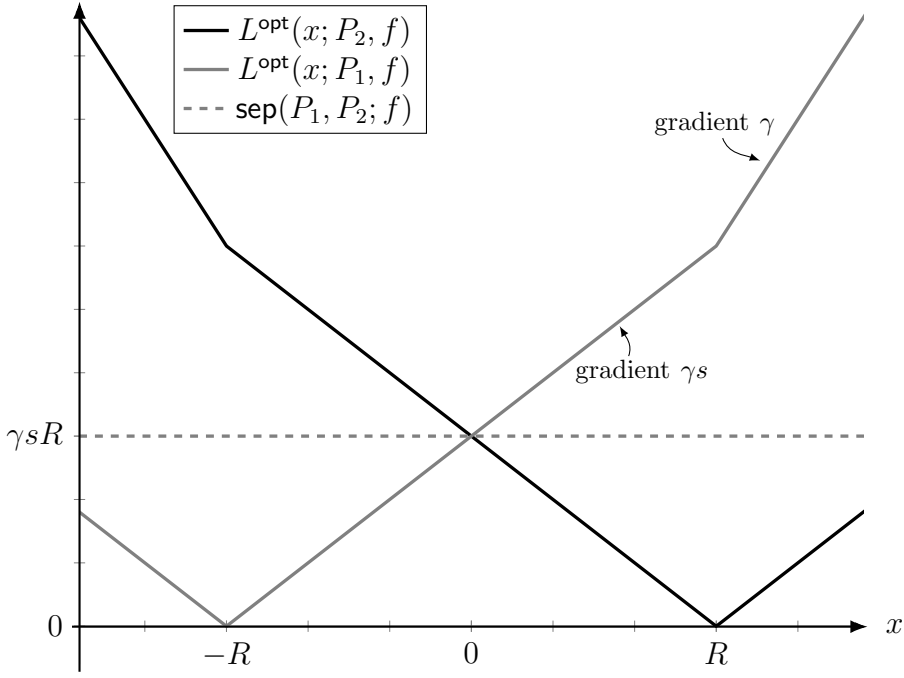

The following result, first obtained by~\citet[][Proposition 1]{carmon2024price}, then provides tight bounds on the minimax $(1-\delta)$th quantile over $\mathcal{G}(\gamma, R)$.  The upper bound is attained by clipped SGD, where the truncation level is $\delta$-dependent.  We provide here a short proof of the lower bound using our techniques in conjunction with the lower bound construction of~\citet[Eq. (5.2.7)]{duchi2018introductory}.
\begin{prop}\label{prop:SCO-rate}
Let $T \in \mathbb{N}$, $\gamma, R > 0$ and $\delta \in (0, 1/4]$.  Then
\[
\mathcal{M}^{\mathsf{opt}}(\delta; T) \asymp \gamma R \cdot\biggl\{\sqrt{\frac{\log(1/\delta)}{T}} \wedge 1\biggr\}.
\]
\end{prop}
\begin{proof}
For the lower bound, we initially consider the case $d = 1$ where $\mathcal{X} \supseteq [-R,R]$.  Take $\mathcal{Y} = \{-1, 1\}$ and consider the function $f: \mathbb{R} \times \mathcal{Y} \rightarrow \mathbb{R}$ given by $f(x, y) \coloneqq \gamma  \lvert x + yR \rvert$.  Let $g: [0,1) \rightarrow \mathbb{R}$ be the strictly increasing function given by $g(x) \coloneqq x\log\bigl(\frac{1 + x}{1 - x}\bigr)$, and set $s \coloneqq g^{-1}\bigl(\frac{1}{2T}\log\bigl(\frac{1}{4\delta(1 - \delta)}\bigr)\bigr) \in [0,1)$.  Now let $P_1 = \mathsf{Ber}(\frac{1 + s}{2})$ and $P_2 = \mathsf{Ber}(\frac{1 - s}{2})$, so that $\mathrm{KL}(P_1, P_2) = g(s)$, 
\[
\inf_{x \in \mathcal{X}} F_{P_1}(x) = F_{P_1}(-R) = \gamma R(1 - s) = F_{P_2}(R) = \inf_{x \in \mathcal{X}} F_{P_2}(x),
\]
and $(P_1,f), (P_2,f) \in \mathcal{G}(\gamma,R)$.  Further, $\mathsf{sep}(P_1, P_2;f) = F_{P_1}(0) - \gamma R (1 - s) = \gamma s R$; see Figure~\ref{Fig:Sep}.    Applying the Bretagnolle--Huber inequality, we deduce that
\[
\mathrm{TV}\bigl(P_1^{\otimes T}, P_2^{\otimes T}\bigr) \leq \sqrt{1 - \exp\bigl\{-T \cdot \mathrm{KL}(P_1, P_2)\bigr\}} = \bigl\{1 - \sqrt{4\delta(1-\delta)}\bigr\}^{1/2}
< 1 - 2\delta.
\]
Thus, by Lemma~\ref{lem:SCO-Le-Cam}, $\mathcal{M}_{-}^{\mathsf{opt}}(\delta;T) \geq \gamma s R$.  Now consider the case where $\log(1/\delta) \leq T$.  Then, since $\delta \in (0,1/4]$ and $g^{-1}$ is a strictly increasing function, 
\[
s \geq g^{-1}\biggl(\frac{\log(1/\delta)}{10T}\biggr) \geq  \sqrt{\frac{\log(1/\delta)}{30T}},
\]
where the final inequality follows since $g(x) \leq 3x^2$ for $x \in [0,4/5]$, so that $g^{-1}(w) \geq (w/3)^{1/2}$ for $w \in [0,7/4]$.  On the other hand, when $\log(1/\delta) > T$, we have $s > g^{-1}(1/10) \geq 1/5$.  Thus, by~\eqref{ineq:M-opt-lb},
\[
\mathcal{M}^{\mathsf{opt}}(\delta;T) \geq \mathcal{M}_{-}^{\mathsf{opt}}(\delta;T) \geq \frac{\gamma R}{30^{1/2}} \cdot \biggl\{\sqrt{\frac{\log(1/\delta)}{T}} \wedge 1\biggr\}.
\]
Turning to the more general case where $d > 1$ and $\mathcal{X} \supseteq \mathbb{B}_2^{d}(R)$, we consider the function $\tilde{f}: \mathbb{R}^d \times \mathcal{Y}$ given by $\tilde{f}(x, y) \coloneqq f(x^{\top} e_1, y)$, where $f$ is as defined above.  The same lower bound then follows.   

To obtain a matching upper bound, we set $\hat{x}_{T} = 0$ if $\log(1/\delta) > T$ and apply clipped SGD with a $\delta$-dependent clipping parameter otherwise.  The conclusion then follows from~\citet[Theorem 4]{carmon2024price}.
\end{proof}

\section{Proofs from Section~\ref{sec:setup}} \label{sec:proofs-setup}

\begin{proof}[Proof of Proposition~\ref{lemma:minimax-quantile-to-minimax-risk}]
    For $\theta\in\Theta$, $P_{\theta} \in \mathcal{P}_{\theta}$ and $\hat{\theta} \in \hat{\Theta}$, define $r(P_{\theta},\hat{\theta}) \coloneqq \mathrm{Quantile} \bigl( 1-\delta; P_{\theta}, L(\hat{\theta}, \theta) \bigr)$. Then 
    \begin{align*}
        \inf_{\hat{\theta} \in \hat{\Theta}} \sup_{\theta \in \Theta} \sup_{P_{\theta} \in \mathcal{P}_{\theta}} \mathbb{E}_{P_{\theta}} L(\hat{\theta}, \theta) &\geq \inf_{\hat{\theta} \in \hat{\Theta}} \sup_{\theta \in \Theta} \sup_{P_{\theta} \in \mathcal{P}_{\theta}}  P_{\theta}\bigl\{ L(\hat{\theta},\theta) \geq r(P_{\theta},\hat{\theta}) \bigr\} \cdot r(P_{\theta},\hat{\theta}) \\
        &\geq \inf_{\hat{\theta} \in \hat{\Theta}} \sup_{\theta \in \Theta} \sup_{P_{\theta} \in \mathcal{P}_{\theta}} \delta \cdot r(P_{\theta},\hat{\theta}) = \delta \mathcal{M}(\delta),
    \end{align*}
    as desired.
\end{proof}

\begin{proof}[Proof of Theorem~\ref{lemma:important-lower-bound}]
We begin with the lower bound on $\mathcal{M}(\delta)$.  For $\theta\in\Theta$, $P_{\theta} \in \mathcal{P}_{\theta}$ and $\hat{\theta} \in \hat{\Theta}$,
    define $r(P_{\theta},\hat{\theta}) \coloneqq \mathrm{Quantile} \bigl( 1-\delta; P_{\theta} , L(\hat{\theta}, \theta) \bigr)$, so that $P_{\theta} \bigl\{ L(\hat{\theta}, \theta) > r(P_{\theta},\hat{\theta}) \bigr\} \leq \delta$.  Then
    \begin{align*}
        P_{\theta} \biggl\{ L(\hat{\theta}, \theta) > \sup_{\theta'\in\Theta} \sup_{P_{\theta'}\in\mathcal{P}_{\theta'}} r(P_{\theta'},\hat{\theta}) \biggr\} \leq \delta,
    \end{align*}
    for all $\theta\in\Theta$ and $P_{\theta} \in \mathcal{P}_{\theta}$, so that
    \begin{align} \label{eq:exchange-inf-and-sup}
        \sup_{\theta\in\Theta} \sup_{P_{\theta} \in \mathcal{P}_{\theta}} r(P_{\theta},\hat{\theta}) \geq \inf \biggl\{ r \in [0,\infty):\sup_{\theta\in\Theta} \sup_{P_{\theta} \in \mathcal{P}_{\theta}} P_{\theta} \bigl\{ L(\hat{\theta}, \theta) > r \bigr\} \leq \delta \biggr\} \eqqcolon r(\hat{\theta}).
    \end{align}
    Further, if $r > \inf_{\hat{\theta} \in \hat{\Theta}} r(\hat{\theta})$, then there exists $\tilde{\theta} \in \hat{\Theta}$ and $\epsilon>0$ such that $r(\tilde{\theta}) \leq r-\epsilon$, so 
    \begin{align*}
        \inf_{\hat{\theta} \in \hat{\Theta}} \sup_{\theta\in\Theta} \sup_{P_{\theta} \in \mathcal{P}_{\theta}} P_{\theta} \bigl\{ L(\hat{\theta}, \theta) > r \bigr\} &\leq \sup_{\theta\in\Theta} \sup_{P_{\theta} \in \mathcal{P}_{\theta}} P_\theta \bigl\{ L(\tilde{\theta}, \theta) > r \bigr\}\\
        &\leq \sup_{\theta\in\Theta} \sup_{P_{\theta} \in \mathcal{P}_{\theta}} P_\theta \bigl\{ L(\tilde{\theta}, \theta) > r(\tilde{\theta}) + \epsilon \bigr\} \leq \delta. \numberthis \label{eq:exchange-inf-and-inf}
    \end{align*}    
    Hence, 
    \begin{align*}
        \mathcal{M}_-(\delta) &= \inf \Bigl\{ r\in[0,\infty) : \inf_{\hat{\theta} \in \hat{\Theta}} \sup_{\theta\in\Theta} \sup_{P_{\theta} \in \mathcal{P}_{\theta}} P_{\theta} \bigl\{ L(\hat{\theta}, \theta) > r \bigr\}\leq \delta \Bigr\}\\
        \overset{(i)}&{\leq} \inf\Bigl\{ r \in [0,\infty): r > \inf_{\hat{\theta} \in \hat{\Theta}} r(\hat{\theta}) \Bigr\}\\ &= \inf_{\hat{\theta}\in \hat{\Theta}} r(\hat{\theta})
        \overset{(ii)}{\leq} \inf_{\hat{\theta} \in \hat{\Theta}} \sup_{\theta\in\Theta} \sup_{P_{\theta} \in \mathcal{P}_{\theta}} r(P_{\theta},\hat{\theta}) = \mathcal{M}(\delta),
    \end{align*}
    where inequality $(i)$ follows from~\eqref{eq:exchange-inf-and-inf} and inequality $(ii)$ follows from~\eqref{eq:exchange-inf-and-sup}.

For the upper bound, let $\epsilon>0$ and $\xi\in(0,\delta)$.  There exists $r\leq \mathcal{M}_-(\delta-\xi) + \epsilon$ such that
    \begin{align*}
        \inf_{\hat{\theta} \in \hat{\Theta}} \sup_{\theta\in\Theta} \sup_{P_{\theta} \in \mathcal{P}_{\theta}} P_{\theta} \bigl\{ L(\hat{\theta}, \theta) > r \bigr\} \leq \delta - \xi.
    \end{align*}
    It follows that there exists $\tilde{\theta} \in \hat{\Theta}$ such that
    \begin{align*}
        \sup_{\theta\in\Theta} \sup_{P_{\theta} \in \mathcal{P}_{\theta}} P_{\theta} \bigl\{ L(\tilde{\theta}, \theta) > r \bigr\} \leq \delta.
    \end{align*}
    Hence, $\mathcal{M}(\delta) \leq r \leq \mathcal{M}_-(\delta-\xi)+\epsilon$.
    Since $\epsilon>0$ was arbitrary, we deduce that $\mathcal{M}(\delta) \leq \mathcal{M}_-(\delta-\xi)$.

For the final claim, observe that $\delta \mapsto \mathcal{M}_-(\delta)$ is a decreasing function on $(0,1]$, so it has at most countably many discontinuities.  But if $\mathcal{M}_-(\cdot)$ is continuous at $\delta$, then we must have $\mathcal{M}(\delta) = \mathcal{M}_-(\delta)$, by~\eqref{Eq:LowerUpperBound}.
\end{proof}

\begin{proof}[Proof of Lemma~\ref{lemma:high-prob-le-cam-tv}]
    There is nothing to prove in the case $g(\eta)=0$, so we may assume without loss of generality that $g(\eta) > 0$.
    For $a\in(0,1)$, define $\tilde{g}_a : [0,\infty) \to \{0,1\}$ by $\tilde{g}_a(x) \coloneqq \mathbbm{1}_{\{g(x) > ag(\eta)\}}$ for $x\in[0,\infty)$, so that $\tilde{g}_a$ is an increasing function. By applying Le Cam's two-point lemma (stated as Lemma~\ref{lemma:le-cam} for convenience), we obtain
    \begin{align*}
        \inf_{\hat{\theta} \in \hat{\Theta}} \sup_{\theta\in\Theta} \sup_{P_{\theta}\in\mathcal{P}_{\theta}} P_{\theta} \bigl\{ L(\hat{\theta}, \theta) > ag(\eta) \bigr\} &= \inf_{\hat{\theta} \in \hat{\Theta}} \sup_{\theta\in\Theta} \sup_{P_{\theta}\in\mathcal{P}_{\theta}} \mathbb{E}_{P_{\theta}} \bigl\{ \tilde{g}_a\bigl( d(\hat{\theta}, \theta) \bigr) \bigr\}\\
        &\geq \frac{\tilde{g}_a(\eta)}{2} \bigl( 1- \mathrm{TV}(P_1, P_2) \bigr) > \delta.
    \end{align*}
    We deduce that $\mathcal{M}_-(\delta)\geq ag(\eta)$.  Since $a\in(0,1)$ was arbitrary, the conclusion follows.
\end{proof}

\begin{proof}[Proof of Lemma~\ref{lemma:fano's-minimax-quantile}]
    Again, we may assume without loss of generality that $g(\eta) > 0$.  
    For $a \in(0,1)$, define $\tilde{g}_a : [0,\infty) \to \{0,1\}$ by $\tilde{g}_a(x) \coloneqq \mathbbm{1}_{\{g(x) > ag(\eta)\}}$. By Fano's lemma (stated as Lemma~\ref{lemma:fano} for convenience), 
    \begin{align*}
        \inf_{\hat{\theta} \in \hat{\Theta}} \sup_{\theta\in\Theta} \sup_{P_{\theta}\in\mathcal{P}_{\theta}} &P_{\theta} \bigl\{ L(\hat{\theta}, \theta) > ag(\eta) \bigr\} = \inf_{\hat{\theta} \in \hat{\Theta}} \sup_{\theta\in\Theta} \sup_{P_{\theta}\in\mathcal{P}_{\theta}} \mathbb{E}_{P_{\theta}} \bigl\{ \tilde{g}_a\bigl( d(\hat{\theta}, \theta) \bigr) \bigr\}\\
        &\geq \tilde{g}_a(\eta) \biggl\{ 1 - \frac{M^{-1} \inf_{Q\in\mathcal{Q}(\mathcal{X})} \sum_{j=1}^M \mathrm{KL}(P_j,Q) + \log(2-M^{-1})}{\log M} \biggr\} \geq \epsilon > \delta.
    \end{align*}
    Since $a\in(0,1)$ was arbitrary, the conclusion follows.
\end{proof}

\begin{figure}
    \centering
    \begin{tikzpicture}[thick,scale=1, every node/.style={scale=1}]
    
    \begin{scope}
    \clip (-8,2) rectangle (8,8);
    \draw[fill=gray!20] (0,2) circle(3);
    \end{scope}

    \begin{scope}
    \clip (-8,2) rectangle (8,8);
    \draw[fill=white] (0,-2) circle(5);
    \end{scope}
    
    \begin{scope}
    \clip (-8,2) rectangle (8,8);
    \draw[densely dashed](0,2) circle(4);
    \end{scope}
    
    \begin{scope}
        \clip (-5,0.2) rectangle (-2.4, 2);
        \draw[densely dashed] (-3,2) circle (1);
    \end{scope}
    
    \begin{scope}
        \clip (5,0.2) rectangle (2.4, 2);
        \draw[densely dashed] (3,2) circle (1);
    \end{scope}

    \begin{scope}
    \clip (-8,1.2) rectangle (8,8);
    \draw[densely dashed] (0,-2) circle(4);
    \end{scope}

    \draw[>=stealth, <->] (0,5) -- (0, 6);
    \node[fill=white, scale=0.75] at (0,5.5){$\epsilon D$};

    \node at (-1.2,3.7) {$\Theta_0$};
    \node at (-3.2,1.4) {$\Theta_1$};

    \filldraw[black] (2.8,4) circle (1.5pt);
    \node[scale=0.75, above] at (2.8,4) {$\theta^{(1)}$};

    \filldraw[black] (1,3.5) circle (1.5pt);
    \node[scale=0.75, left] at (0.95,3.5) {$\theta_0$};

    \filldraw[black] (2,7) circle (1.5pt);
    \node[scale=0.75, left] at (1.95,7.07) {$\hat{\theta}(x)$};

    \draw[dashed] (2,7) -- (1, 3.5);
    \draw[dashed] (1, 3.5) -- (2.8,4);
\end{tikzpicture}
\caption{\label{fig:projection} Illustration of the construction of $\hat{\theta}^{(1)}$ in the proof of Theorem~\ref{lemma:expectation-lb-to-high-prob-lb}.  If $x\in S(\hat{\theta})$, then we set $\hat{\theta}^{(1)}(x) \coloneqq \hat{\theta}(x)$; otherwise we set $\hat{\theta}^{(1)}(x) \coloneqq \theta^{(1)} \in \Theta_1$.}
\end{figure}
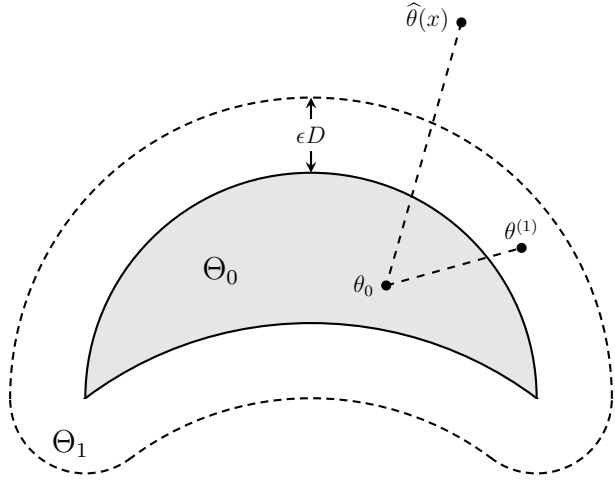


\begin{proof}[Proof of Theorem~\ref{lemma:expectation-lb-to-high-prob-lb}]
    For $\theta\in\Theta$ and $r\geq 0$, let $B(\theta, r) \coloneqq \bigl\{ \theta'\in\Theta : d(\theta,\theta') < r \bigr\}$ be the open ball in $\Theta$ centred at $\theta$ with radius $r$.  Fix $\epsilon>0$ and define $\Theta_1 \coloneqq \bigcup_{\theta_0 \in \Theta_0} B(\theta_0, \epsilon  D)$, so that $d(\theta_0,\theta_1) < (1+\epsilon )D$ for all $\theta_0\in\Theta_0$ and $\theta_1 \in \Theta_1$.  Let $\hat{\Theta}_1$ denote the set of all measurable functions from $\mathcal{X}$ to $\Theta_1$.  For any $\tilde{\theta} \in \hat{\Theta}_1$, we have
    \begin{align*}
        \Delta &\leq \sup_{\theta_0 \in \Theta_0} \sup_{P_{\theta_0} \in \mathcal{P}_{\theta_0}} \mathbb{E}_{P_{\theta_0}} L( \tilde{\theta}, \theta_0)\\
        &=  \sup_{\theta_0 \in \Theta_0} \sup_{P_{\theta_0} \in \mathcal{P}_{\theta_0}} \Bigl\{ \mathbb{E}_{P_{\theta_0}} \bigl\{ L( \tilde{\theta}, \theta_0) \mathbbm{1}_{\{L( \tilde{\theta}, \theta_0) \leq g(\epsilon  D) \}} \bigr\} + \mathbb{E}_{P_{\theta_0}} \bigl\{ L( \tilde{\theta}, \theta_0) \mathbbm{1}_{\{L( \tilde{\theta}, \theta_0) > g(\epsilon  D) \}} \bigr\} \Bigr\}\\
        &\leq g(\epsilon  D) + g\bigl((1+\epsilon )D\bigr)  \sup_{\theta_0 \in \Theta_0} \sup_{P_{\theta_0} \in \mathcal{P}_{\theta_0}} P_{\theta_0} \bigl(L(\tilde{\theta}, \theta_0) > g(\epsilon  D) \bigr).
    \end{align*}
    Hence,
    \begin{align}
        \inf_{\tilde{\theta} \in \hat{\Theta}_1} \sup_{\theta_0 \in \Theta_0} \sup_{P_{\theta_0} \in \mathcal{P}_{\theta_0}} P_{\theta_0} \bigl(L(\tilde{\theta}, \theta_0) > g(\epsilon  D) \bigr) \geq \frac{\Delta - g(\epsilon  D)}{g\bigl((1+\epsilon )D\bigr)}. \label{eq:reduction-lemma-lb1}
    \end{align}
    Now for any $\hat{\theta} \in \hat{\Theta}$, let
    \begin{align*}
        S(\hat{\theta}) \coloneqq \bigl\{ x\in\mathcal{X} : \hat{\theta}(x) \in \Theta_1 \bigr\} \in \mathcal{A}.
    \end{align*}
    Fix $\theta^{(1)} \in \Theta_1$ and $\hat{\theta} \in \hat{\Theta}$, and define a new estimator $\hat{\theta}^{(1)} \in \hat{\Theta}_1$ by
    \begin{align*}
        \hat{\theta}^{(1)}(x) \coloneqq \begin{cases}
            \hat{\theta}(x) \quad&\text{if } x\in S(\hat{\theta})\\
            \theta^{(1)} \quad&\text{if } x\notin S(\hat{\theta});
        \end{cases}
    \end{align*}
    see Figure~\ref{fig:projection}.  Then for any $\xi>0$, $\theta_0\in\Theta_0$ and $P_{\theta_0} \in \mathcal{P}_{\theta_0}$, we have
    \begin{align*}
        P_{\theta_0} \bigl(L(\hat{\theta},\theta_0) > g(\epsilon  D) - \xi \bigr) &\overset{(i)}{=} P_{\theta_0}\bigl(S(\hat{\theta})^{\mathrm{c}}\bigr) + P_{\theta_0} \bigl( S(\hat{\theta})  \cap \bigl\{L(\hat{\theta},\theta_0) > g(\epsilon  D) - \xi \bigr\}\bigr) \\
        &= P_{\theta_0}\bigl(S(\hat{\theta})^{\mathrm{c}}\bigr) + P_{\theta_0} \bigl(S(\hat{\theta}) \cap \bigl\{L(\hat{\theta}^{(1)},\theta_0) > g(\epsilon  D) - \xi\bigr\} \bigr) \\
        &\geq P_{\theta_0} \bigl( L(\hat{\theta}^{(1)},\theta_0) > g(\epsilon  D) - \xi \bigr),
    \end{align*}
    where $(i)$ follows since $g\bigl(d(\hat{\theta}(x), \theta_0)\bigr) \geq g(\epsilon  D) > g(\epsilon  D) - \xi$ for all $x\in S(\hat{\theta})^{\mathrm{c}}$.  Therefore, 
    \begin{align*}
        \inf_{\hat{\theta} \in \hat{\Theta}} \sup_{\theta \in \Theta} &\sup_{P_{\theta} \in \mathcal{P}_{\theta}}  P_{\theta} \bigl(L(\hat{\theta}, \theta) > g(\epsilon  D) - \xi \bigr)\\
        &\geq \inf_{\hat{\theta} \in \hat{\Theta}} \sup_{\theta_0 \in \Theta_0} \sup_{P_{\theta_0} \in \mathcal{P}_{\theta_0}}  P_{\theta_0} \bigl(L(\hat{\theta}^{(1)}, \theta_0) > g(\epsilon  D) - \xi \bigr)\\
        &\geq \inf_{\tilde{\theta} \in \hat{\Theta}_1} \sup_{\theta_0 \in \Theta_0} \sup_{P_{\theta_0} \in \mathcal{P}_{\theta_0}} P_{\theta_0} \bigl(L(\tilde{\theta}, \theta_0) > g(\epsilon  D) - \xi \bigr) \geq \frac{\Delta - g(\epsilon  D)}{g\bigl((1+\epsilon )D\bigr)},
    \end{align*}
    where the final inequality follows from~\eqref{eq:reduction-lemma-lb1}.  We conclude that $\mathcal{M}_{-}(\delta) \geq g(\epsilon  D) - \xi$ for all $\delta \in \bigl(0, \frac{\Delta - g(\epsilon  D)}{g((1+\epsilon )D)} \bigr)$, and since $\xi>0$ was arbitrary, the result follows.
\end{proof}

We now turn to the proof of our boosting result, Proposition~\ref{cor:boost-delta}, which requires two preliminary lemmas.  Let $(\mathcal{Y},\mathcal{B})$ be a measurable space, and suppose that $\Theta$ and $\mathcal{Z}$ are topological spaces.  A function $f:\mathcal{Y} \times \Theta \rightarrow \mathcal{Z}$ is a \emph{Carath\'{e}odory function} if:
\begin{enumerate}[(i)]
\item[\emph{(i)}] the function $f(\cdot,\theta):\mathcal{Y} \rightarrow \mathcal{Z}$ is measurable (when $\mathcal{Z}$ is equipped with its Borel $\sigma$-algebra) for each $\theta \in \Theta$;
\item[\emph{(ii)}] the function $f(y,\cdot):\Theta \rightarrow \mathcal{Z}$ is continuous for each $y \in \mathcal{Y}$.
\end{enumerate}

\begin{lemma}\label{lemma:boost-delta}
    For $m\in\mathbb{N}$, let $\mathcal{P}_{\theta}^{\otimes m} \coloneqq \{ P_{\theta}^{\otimes m} \in \mathcal{Q}(\mathcal{X}^m) : P_{\theta} \in \mathcal{P}_{\theta} \}$ and $\mathcal{P}_{\Theta}^{\otimes m} \coloneqq \{\mathcal{P}_{\theta}^{\otimes m} : \theta\in\Theta\}$. Suppose that $\Theta$ is a non-empty Polish space, that $L:\Theta \times \Theta \rightarrow [0,\infty)$ is a Carath\'eodory function, and there exists $A>0$ such that $L(\theta_1,\theta_2) \leq A\bigl\{ L(\theta_1,\theta_3) + L(\theta_2,\theta_3) \bigr\}$ for all $\theta_1,\theta_2,\theta_3\in\Theta$.
    Then, for any $\delta\in(0,1]$, we have
    \begin{align*}
         \mathcal{M}_-(\delta, \mathcal{P}_{\Theta}, L) \geq \frac{1}{2A}\mathcal{M}_-\bigl(\delta^3+3\delta^2(1-\delta), \mathcal{P}_{\Theta}^{\otimes 3}, L\bigr).
    \end{align*}
\end{lemma}

\begin{proof}
    Fix $\epsilon>0$, there exists a measurable function $\tilde{\theta} : \mathcal{X} \to \Theta$ such that for all $\theta\in\Theta$,
    \begin{align*}
        P_{\theta}\bigl\{ L(\tilde{\theta}, \theta) < \mathcal{M}_-(\delta,\mathcal{P}_{\Theta}, L) + \epsilon \bigr\} \geq 1-\delta.
    \end{align*}
    Now fix $\theta\in\Theta$, and define the event
    \begin{align*}
    E \coloneqq \biggl\{ (x_1,x_2,x_3) \in\mathcal{X}^3 : \sum_{\ell=1}^3 \mathbbm{1}_{\{L(\tilde{\theta}(x_{\ell}), \theta) < \mathcal{M}_-(\delta,\mathcal{P}_{\Theta}, L) + \epsilon\}} \geq 2 \biggr\}
    \end{align*}
    so that $P_{\theta}^{\otimes 3}(E) \geq 1 - \delta^3 - 3\delta^2 ( 1- \delta)$.  Further, define $\psi : \mathcal{X}^3 \rightarrow \mathrm{Pow}(\Theta)$ by
    \begin{align*}
        \psi(x_1,x_2,x_3) \coloneqq \biggl\{ \theta'\in\Theta : \sum_{\ell=1}^3 \mathbbm{1}_{\{L(\tilde{\theta}(x_{\ell}), \theta') < \mathcal{M}_-(\delta,\mathcal{P}_{\Theta}, L) + \epsilon\}} \geq 2 \biggr\}.
    \end{align*}
    For $(x_1,x_2,x_3)\in\mathcal{X}^3$, let $\bar{\psi}(x_1,x_2,x_3)$ denote the closure of $\psi(x_1,x_2,x_3)$ in $\Theta$, and define $S : \mathcal{X}^3 \rightarrow \mathrm{Pow}(\Theta)$ by
    \begin{align*}
        S(x_1,x_2,x_3) \coloneqq \begin{cases}
            \bar{\psi}(x_1,x_2,x_3) \quad&\text{when } \psi(x_1,x_2,x_3)\neq \emptyset\\
            \Theta &\text{otherwise}.
        \end{cases}
    \end{align*}
    By Lemma~\ref{lemma:measurable-selector}, there exists a measurable function $\hat{\theta} : \mathcal{X}^3 \to \Theta$ such that $\hat{\theta}(x_1,x_2,x_3) \in S(x_1,x_2,x_3)$ for all $(x_1,x_2,x_3) \in \mathcal{X}^3$.
    For any $(x_1,x_2,x_3) \in E$, we have $\theta\in \psi(x_1,x_2,x_3)$, so $\psi(x_1,x_2,x_3)$ is non-empty; for such $(x_1,x_2,x_3)$, it follows that  $\hat{\theta}(x_1,x_2,x_3) \in \bar{\psi}(x_1,x_2,x_3)$, so 
    \[
    \sum_{\ell=1}^3 \mathbbm{1}_{\{L(\tilde{\theta}(x_{\ell}), \hat{\theta}(x_1,x_2,x_3)) \leq \mathcal{M}_-(\delta,\mathcal{P}_{\Theta}, L) + \epsilon\}} \geq 2.
    \]
    Therefore, for all $(x_1,x_2,x_3) \in E$, there exists $k\in \{1,2,3\}$ such that both $L\bigl(\tilde{\theta}(x_k), \theta \bigr) < \mathcal{M}_-(\delta,\mathcal{P}_{\Theta}, L) + \epsilon$ and $L\bigl(\tilde{\theta}(x_k), \hat{\theta}(x_1,x_2,x_3)\bigr) \leq \mathcal{M}_-(\delta,\mathcal{P}_{\Theta}, L) + \epsilon$.  Thus, for all $(x_1,x_2,x_3) \in E$,
    \begin{align*}
        L\bigl(\hat{\theta}(x_1,x_2,x_3),\theta\bigr) \leq A\bigl\{ L(\tilde{\theta}(x_k), \theta) + L\bigl(\tilde{\theta}(x_k), \hat{\theta}(x_1,x_2,x_3)\bigr) \bigr\} \leq 2A \mathcal{M}_-(\delta,\mathcal{P}_{\Theta}, L) + 2A\epsilon,
    \end{align*}
    and we deduce that
    \begin{align*}
        P_{\theta}^{\otimes 3} \bigl\{ L(\hat{\theta},\theta) \leq 2A \mathcal{M}_-(\delta,\mathcal{P}_{\Theta}, L) + 2A\epsilon \bigr\} \geq P_{\theta}^{\otimes 3}(E) \geq 1 - \delta^3 - 3\delta^2(1-\delta).
    \end{align*}
    This holds for all $\theta\in\Theta$, so
    \begin{align*}
        \mathcal{M}_-\bigl(\delta^3+3\delta^2(1-\delta), \mathcal{P}_{\Theta}^{\otimes 3}, L\bigr) \leq 2A \mathcal{M}_-(\delta,\mathcal{P}_{\Theta}, L) + 2A\epsilon,
    \end{align*}
    and the result follows since $\epsilon > 0$ was arbitrary.
\end{proof}
The following technical lemma on the existence of a measurable selector is based on the theory of correspondences, so we now present some basic definitions, following \citet{aliprantis2007infinite}.  A \emph{correspondence} $\varphi$ from a set $\mathcal{Y}$ to a set $\mathcal{Z}$, written $\varphi: \mathcal{Y} \twoheadrightarrow \mathcal{Z}$, assigns to each $y \in \mathcal{Y}$ a subset $\varphi(y)$ of $\mathcal{Z}$; in other words, it is a function from $\mathcal{Y}$ to $\mathrm{Pow}(\mathcal{Z})$.  Its \emph{lower inverse} $\varphi^{\mathrm{L}}:\mathrm{Pow}(\mathcal{Z}) \rightarrow \mathrm{Pow}(\mathcal{Y})$ is defined by
\[
\varphi^{\mathrm{L}}(B) \coloneqq \{y \in \mathcal{Y}: \varphi(y) \cap B \neq \emptyset\}.
\]
When $\mathcal{Z}$ is a topological space, the \emph{closure correspondence} $\bar{\varphi}: \mathcal{Y} \twoheadrightarrow \mathcal{Z}$ of $\varphi$ assigns to each $y \in \mathcal{Y}$ the subset of $\mathcal{Z}$ given by the closure of $\varphi(y)$.  If, in addition, $(\mathcal{Y},\mathcal{B})$ is a measurable space, we say that the correspondence $\varphi: \mathcal{Y} \twoheadrightarrow \mathcal{Z}$ is \emph{measurable} if $\varphi^{\mathrm{L}}(F) \in \mathcal{B}$ for every closed subset $F$ of $\mathcal{Z}$, and \emph{weakly measurable} if $\varphi^{\mathrm{L}}(G) \in \mathcal{B}$ for every open subset $G$ of $\mathcal{Z}$. 
\begin{lemma} \label{lemma:measurable-selector}
    Let $(\mathcal{X},\mathcal{A})$ be a measurable space, let $\Theta$ be a non-empty Polish space and suppose that $L:\Theta \times \Theta \rightarrow [0,\infty)$ is a Carath\'{e}odory function.  Let $\tilde{\theta}: \mathcal{X} \to \Theta$ be measurable, let $r \in (0,\infty)$ and define a correspondence $\psi: \mathcal{X}^3 \twoheadrightarrow \Theta$ by
    \begin{align*}
        \psi(x_1,x_2,x_3) \coloneqq \biggl\{ \theta'\in\Theta : \sum_{\ell=1}^3 \mathbbm{1}_{\{L(\tilde{\theta}(x_{\ell}), \theta') < r\}} \geq 2 \biggr\}.
    \end{align*}
    Further, define a non-empty correspondence $S : \mathcal{X}^3 \twoheadrightarrow \Theta$ by
    \begin{align*}
        S(x_1,x_2,x_3) \coloneqq \begin{cases}
            \bar{\psi}(x_1,x_2,x_3) \quad&\text{when } \psi(x_1,x_2,x_3)\neq \emptyset\\
            \Theta &\text{otherwise}.
        \end{cases}
    \end{align*}
    Then there exists a measurable function $\hat{\theta} : \mathcal{X}^3 \to \Theta$ such that $\hat{\theta}(x_1,x_2,x_3) \in S(x_1,x_2,x_3)$ for all $(x_1,x_2,x_3) \in \mathcal{X}^3$.
\end{lemma}
\begin{proof}
    The map $(x_1,x_2,x_3,\theta') \mapsto \bigl(L\bigl(\tilde{\theta}(x_1), \theta'\bigr),L\bigl(\tilde{\theta}(x_2), \theta'\bigr),L\bigl(\tilde{\theta}(x_3), \theta'\bigr)\bigr)$ is a Carath\'{e}odory function from $\mathcal{X}^3 \times \Theta$ to $[0,\infty)^3$. Define $G\coloneqq \bigl\{(v_1,v_2,v_3)^\top\in[0,\infty)^3 : \sum_{\ell=1}^3 \mathbbm{1}_{\{v_{\ell} < r\}} \geq 2\bigr\}$, which is an open subset of $[0,\infty)^3$. Then, by \citet[Lemma~18.7]{aliprantis2007infinite}, the correspondence $\psi$, which can be equivalently expressed as $\psi(x_1,x_2,x_3) = \bigl\{ \theta'\in\Theta : \bigl(L(\tilde{\theta}(x_{\ell}), \theta')\bigr)_{\ell=1}^3 \in G \bigr\}$, is measurable, so by \citet[Lemma~18.2]{aliprantis2007infinite} it is also weakly measurable.  It follows that the lower inverse $\psi^{\mathrm{L}}: \mathrm{Pow}(\Theta) \rightarrow \mathrm{Pow}(\mathcal{X}^3)$ satisfies $\psi^{\mathrm{L}}(\Theta) = \{(x_1,x_2,x_3) \in \mathcal{X}^3 : \psi(x_1,x_2,x_3) \neq \emptyset\} \in \mathcal{A}^{\otimes 3}$, so $U \coloneqq  \psi^{\mathrm{L}}(\Theta)^{\mathrm{c}} \in \mathcal{A}^{\otimes 3}$.
    Now define a non-empty correspondence $\varphi:\mathcal{X}^3 \twoheadrightarrow \Theta$ by
    \begin{align*}
        \varphi(x_1,x_2,x_3) \coloneqq \begin{cases}
            \psi(x_1,x_2,x_3) \quad&\text{if } (x_1,x_2,x_3) \in U^{\mathrm{c}} \\
            \Theta \quad&\text{if } (x_1,x_2,x_3) \in U.
        \end{cases}
    \end{align*}
    Then, for any non-empty open set $V\subseteq \Theta$, we have $\psi^{\mathrm{L}}(V) \in \mathcal{A}^{\otimes 3}$, so \begin{align*}
        \varphi^{\mathrm{L}}(V) &= \bigl\{ (x_1,x_2,x_3) \in \mathcal{X}^3 : \varphi(x_1,x_2,x_3) \cap V \neq \emptyset \bigr\}\\
        &= \bigl\{ (x_1,x_2,x_3) \in \mathcal{X}^3 : \psi(x_1,x_2,x_3) \cap V \neq \emptyset \bigr\} \cup U = \psi^{\mathrm{L}}(V) \cup U \in \mathcal{A}^{\otimes 3}.
    \end{align*}
    Since we also have $\varphi^{\mathrm{L}}(\emptyset) = \emptyset \in \mathcal{A}^{\otimes 3}$, we deduce that $\varphi$ is weakly measurable, and by \citet[Lemma~18.3]{aliprantis2007infinite}, its closure correspondence $\bar{\varphi}$ is also weakly measurable. Therefore, by \citet[Theorem~18.13]{aliprantis2007infinite}, $\bar{\varphi}$ admits a measurable selector $\hat{\theta}:\mathcal{X}^3 \to \Theta$; in other words, $\hat{\theta}$ is measurable, with $\hat{\theta}(x_1,x_2,x_3) \in \bar{\varphi}(x_1,x_2,x_3)$ for all $(x_1,x_2,x_3) \in \mathcal{X}^3$.  Since $\bar{\varphi} = S$, the conclusion follows.
\end{proof}
We are now in a position to complete the proof of Proposition~\ref{cor:boost-delta}.
\begin{proof}[Proof of Proposition~\ref{cor:boost-delta}]
    The sequence $x, h(x), h^{\circ 2}(x), h^{\circ 3}(x),\ldots$ is strictly decreasing and converges to zero for $x\in(0,1/2)$.  It follows that there exists $k \equiv k(\delta_-,\delta_+) \in\mathbb{N}$ such that $h^{\circ k}(\delta_+) \leq \delta_-$. Moreover, since $h^{\circ k}(\cdot)$ is increasing on $(0,1/2)$, we have $h^{\circ k}(\delta) \leq h^{\circ k}(\delta_+) \leq \delta_-$ for all $\delta\in(0,\delta_+]$.  Since $d(\cdot,\theta)$ and $d(\theta,\cdot)$ are continuous for all $\theta\in\Theta$, and since $g$ is continuous, we have that $L$ is a Carath\'{e}odory function.  Hence, by iteratively applying Lemma~\ref{lemma:boost-delta} $k$ times, we deduce that
    \begin{align*}
        \mathcal{M}_{-}(\delta,\mathcal{P}_{\Theta},L) \geq \frac{1}{(2A)^k} \cdot \mathcal{M}_{-}\bigl(h^{\circ k}(\delta), \mathcal{P}_{\Theta}^{\otimes 3^k},L\bigr) \geq \frac{1}{(2A)^k} \cdot \mathcal{M}_{-}\bigl(\delta_-, \mathcal{P}_{\Theta}^{\otimes 3^k},L\bigr),
    \end{align*}
as required.
\end{proof}

\noindent \textbf{Acknowledgements:} The research of all three authors was supported by the third author's European Research Council Advanced Grant 101019498.  We are very grateful to Thomas Berrett, Saminul Haque and Alexandre Tsybakov for helpful feedback on an earlier draft.

{
\bibliographystyle{custom}
\bibliography{bibliography}

\begin{thebibliography}{53}
\providecommand{\natexlab}[1]{#1}
\providecommand{\url}[1]{\texttt{#1}}
\providecommand{\urlprefix}{URL }
\providecommand{\eprint}[2][]{\url{#2}}

\bibitem[{Adler and Taylor(2009)}]{adler2009random}
Adler, R.~J. and Taylor, J.~E. (2009) \emph{Random Fields and Geometry}.
  Springer.

\bibitem[{Aliprantis and Border(2007)}]{aliprantis2007infinite}
Aliprantis, C.~D. and Border, K.~C. (2007) \emph{{Infinite Dimensional
  Analysis: A Hitchhiker's Guide}}. Springer.

\bibitem[{Bellec(2018)}]{bellec2018sharp}
Bellec, P.~C. (2018) Sharp oracle inequalities for least squares estimators in
  shape restricted regression. \emph{The Annals of Statistics}, \textbf{46},
  745--780.

\bibitem[{Bellec, Lecu{\'e} and Tsybakov(2018)}]{bellec2018slope}
Bellec, P.~C., Lecu{\'e}, G. and Tsybakov, A.~B. (2018) Slope meets {L}asso:
  improved oracle bounds and optimality. \emph{The Annals of Statistics},
  \textbf{46}, 3603--3642.

\bibitem[{Bogdan et~al.(2015)Bogdan, Van Den~Berg, Sabatti, Su and
  Cand{\`e}s}]{bogdan2015slope}
Bogdan, M., Van Den~Berg, E., Sabatti, C., Su, W. and Cand{\`e}s, E.~J. (2015)
  {SLOPE---Adaptive variable selection via convex optimization}. \emph{The
  Annals of Applied Statistics}, \textbf{9}, 1103.

\bibitem[{Bretagnolle and Huber(1979)}]{bretagnolle1979estimation}
Bretagnolle, J. and Huber, C. (1979) Estimation des densit{\'e}s: {R}isque
  minimax. \emph{Zeitschrift f{\"u}r Wahrscheinlichkeitstheorie und verwandte
  Gebiete}, \textbf{47}, 119--137.

\bibitem[{Cai(2012)}]{cai2012minimax}
Cai, T.~T. (2012) Minimax and adaptive inference in nonparametric function
  estimation. \emph{Statistical Science}, \textbf{27}, 31--50.

\bibitem[{Cand{\`e}s and Davenport(2013)}]{candes2013well}
Cand{\`e}s, E.~J. and Davenport, M.~A. (2013) How well can we estimate a sparse
  vector? \emph{Applied and Computational Harmonic Analysis}, \textbf{34},
  317--323.

\bibitem[{Carmon and Hinder(2024)}]{carmon2024price}
Carmon, Y. and Hinder, O. (2024) The price of adaptivity in stochastic convex
  optimization. \emph{arXiv preprint arXiv:2402.10898}.

\bibitem[{Catoni(2012)}]{catoni2012challenging}
Catoni, O. (2012) Challenging the empirical mean and empirical variance: {A}
  deviation study. \emph{Annales de l'Institut Henri Poincaré, Probabilités
  et Statistiques}, \textbf{48}, 1148--1185.

\bibitem[{Chatterjee, Guntuboyina and Sen(2015)}]{chatterjee2015risk}
Chatterjee, S., Guntuboyina, A. and Sen, B. (2015) {On risk bounds in isotonic
  and other shape restricted regression problems}. \emph{The Annals of
  Statistics}, \textbf{43}, 1774--1800.

\bibitem[{Chen, Gao and Ren(2016)}]{chen2016general}
Chen, M., Gao, C. and Ren, Z. (2016) {A general decision theory for Huber’s
  $\epsilon$-contamination model}. \emph{Electronic Journal of Statistics},
  \textbf{10}, 3752--3774.

\bibitem[{Chen, Gao and Ren(2018)}]{chen2018robust}
Chen, M., Gao, C. and Ren, Z. (2018) Robust covariance and scatter matrix
  estimation under Huber’s contamination model. \emph{The Annals of
  Statistics}, \textbf{46}, 1932--1960.

\bibitem[{Chhor, Klopp and Tsybakov(2024)}]{chhor2024generalized}
Chhor, J., Klopp, O. and Tsybakov, A. (2024) Generalized multi-view model:
  {A}daptive density estimation under low-rank constraints. \emph{arXiv
  preprint arXiv:2404.17209}.

\bibitem[{Davis et~al.(2021)Davis, Drusvyatskiy, Xiao and Zhang}]{davis2021low}
Davis, D., Drusvyatskiy, D., Xiao, L. and Zhang, J. (2021) From low probability
  to high confidence in stochastic convex optimization. \emph{Journal of
  Machine Learning Research}, \textbf{22}, 1--38.

\bibitem[{Depersin and Lecu{\'e}(2022{\natexlab{a}})}]{depersin2022optimal}
Depersin, J. and Lecu{\'e}, G. (2022{\natexlab{a}}) Optimal robust mean and
  location estimation via convex programs with respect to any pseudo-norms.
  \emph{Probability Theory and Related Fields}, \textbf{183}, 997--1025.

\bibitem[{Depersin and Lecu{\'e}(2022{\natexlab{b}})}]{depersin2022robust}
Depersin, J. and Lecu{\'e}, G. (2022{\natexlab{b}}) {Robust sub-Gaussian
  estimation of a mean vector in nearly linear time}. \emph{The Annals of
  Statistics}, \textbf{50}, 511--536.

\bibitem[{DeVore et~al.(2006)DeVore, Kerkyacharian, Picard and
  Temlyakov}]{devore2006approximation}
DeVore, R., Kerkyacharian, G., Picard, D. and Temlyakov, V. (2006)
  Approximation methods for supervised learning. \emph{Foundations of
  Computational Mathematics}, \textbf{6}, 3--58.

\bibitem[{Devroye et~al.(2016)Devroye, Lerasle, Lugosi and
  Oliveira}]{devroye2016subgaussian}
Devroye, L., Lerasle, M., Lugosi, G. and Oliveira, R.~I. (2016) {Sub-Gaussian
  mean estimators}. \emph{The Annals of Statistics}, \textbf{44}, 2695--2725.

\bibitem[{Devroye and Lugosi(2001)}]{devroye2001combinatorial}
Devroye, L. and Lugosi, G. (2001) \emph{{Combinatorial Methods in Density
  Estimation}}. Springer.

\bibitem[{Donoho(1994)}]{donoho1994statistical}
Donoho, D.~L. (1994) Statistical estimation and optimal recovery. \emph{The
  Annals of Statistics}, \textbf{22}, 238--270.

\bibitem[{Donoho and Liu(1991)}]{donoho1991geometrizing}
Donoho, D.~L. and Liu, R.~C. (1991) Geometrizing rates of convergence, II.
  \emph{The Annals of Statistics}, 633--667.

\bibitem[{Duchi(2018)}]{duchi2018introductory}
Duchi, J.~C. (2018) {Introductory Lectures on Stochastic Optimization}. In
  \emph{The Mathematics of Data}, 99--185, American Mathematical Society.

\bibitem[{Duchi and Haque(2024)}]{duchi2024information}
Duchi, J.~C. and Haque, S. (2024) An information-theoretic lower bound in
  time-uniform estimation. \emph{arXiv preprint arXiv:2402.08794}.

\bibitem[{Ehrenfeucht et~al.(1989)Ehrenfeucht, Haussler, Kearns and
  Valiant}]{ehrenfeucht1989general}
Ehrenfeucht, A., Haussler, D., Kearns, M. and Valiant, L. (1989) A general
  lower bound on the number of examples needed for learning. \emph{Information
  and Computation}, \textbf{82}, 247--261.

\bibitem[{El~Hanchi, Maddison and Erdogdu(2024)}]{elhanchi2024minimax}
El~Hanchi, A., Maddison, C. and Erdogdu, M. (2024) Minimax Linear Regression
  under the Quantile Risk. In \emph{Proceedings of Conference on Learning
  Theory}, vol. 247, 1516--1572.

\bibitem[{Gerchinovitz, M{\'e}nard and Stoltz(2020)}]{gerchinovitz2020fano}
Gerchinovitz, S., M{\'e}nard, P. and Stoltz, G. (2020) Fano’s inequality for
  random variables. \emph{Statistical Science}, \textbf{35}, 178--201.

\bibitem[{Groeneboom and Jongbloed(2014)}]{groeneboom2014nonparametric}
Groeneboom, P. and Jongbloed, G. (2014) \emph{Nonparametric Estimation under
  Shape Constraints}. Cambridge University Press.

\bibitem[{Hopkins(2020)}]{hopkins2020mean}
Hopkins, S.~B. (2020) Mean estimation with sub-{G}aussian rates in polynomial
  time. \emph{The Annals of Statistics}, \textbf{48}, 1193--1213.

\bibitem[{Huber(1964)}]{huber1964robust}
Huber, P.~J. (1964) {Robust estimation of a location parameter}. \emph{The
  Annals of Mathematical Statistics}, \textbf{35}, 73--101.

\bibitem[{Juditsky and Nemirovski(2009)}]{juditsky2009convex}
Juditsky, A.~B. and Nemirovski, A.~S. (2009) {Nonparametric estimation by
  convex programming}. \emph{The Annals of Statistics}, \textbf{37}, 2278 --
  2300.

\bibitem[{Kerkyacharian et~al.(2014)Kerkyacharian, Tsybakov, Temlyakov, Picard
  and Koltchinskii}]{kerkyacharian2014optimal}
Kerkyacharian, G., Tsybakov, A.~B., Temlyakov, V., Picard, D. and Koltchinskii,
  V. (2014) Optimal exponential bounds on the accuracy of classification.
  \emph{Constructive Approximation}, \textbf{39}, 421--444.

\bibitem[{Koltchinskii and Lounici(2017)}]{koltchinskii2017concentration}
Koltchinskii, V. and Lounici, K. (2017) Concentration inequalities and moment
  bounds for sample covariance operators. \emph{Bernoulli}, \textbf{23},
  110--133.

\bibitem[{Lecu{\'e} and Mendelson(2013)}]{lecue2013learning}
Lecu{\'e}, G. and Mendelson, S. (2013) Learning sub-{G}aussian classes: {U}pper
  and minimax bounds. In \emph{Topics in Learning Theory}, Soci{\'e}t{\'e}
  Math{\'e}matique de France.

\bibitem[{Lugosi and Mendelson(2019{\natexlab{a}})}]{lugosi2019mean}
Lugosi, G. and Mendelson, S. (2019{\natexlab{a}}) Mean estimation and
  regression under heavy-tailed distributions: {A} survey. \emph{Foundations of
  Computational Mathematics}, \textbf{19}, 1145--1190.

\bibitem[{Lugosi and Mendelson(2019{\natexlab{b}})}]{lugosi2019sub-gaussian}
Lugosi, G. and Mendelson, S. (2019{\natexlab{b}}) {Sub-Gaussian estimators of
  the mean of a random vector}. \emph{The Annals of Statistics}, \textbf{47},
  783--794.

\bibitem[{Lugosi and Mendelson(2021)}]{lugosi21robust}
Lugosi, G. and Mendelson, S. (2021) {Robust multivariate mean estimation: The
  optimality of trimmed mean}. \emph{The Annals of Statistics}, \textbf{49},
  393--410.

\bibitem[{Massart(2007)}]{massart2007concentration}
Massart, P. (2007) \emph{Concentration Inequalities and Model Selection}.
  Springer.

\bibitem[{Nemirovski and Yudin(1983)}]{nemirovski1983problem}
Nemirovski, A. and Yudin, D. (1983) \emph{Problem Complexity and Method
  Efficiency in Optimization}. Wiley-Interscience.

\bibitem[{Polyanskiy and Wu(2021)}]{polyanskiy2021dualizing}
Polyanskiy, Y. and Wu, Y. (2021) Dualizing {L}e {C}am's method for functional
  estimation, with applications to estimating the unseens. \emph{arXiv preprint
  arXiv:1902.05616}.

\bibitem[{Robbins(1955)}]{robbins1955remark}
Robbins, H. (1955) A remark on {S}tirling's formula. \emph{The American
  Mathematical Monthly}, \textbf{62}, 26--29.

\bibitem[{Rockafellar(1997)}]{rockafellar1997convex}
Rockafellar, R.~T. (1997) \emph{Convex Analysis}. Princeton University Press.

\bibitem[{Samworth and Sen(2018)}]{samworth2018editorial}
Samworth, R.~J. and Sen, B. (2018) {Editorial: Special issue on
  “Nonparametric inference under shape constraints”}. \emph{Statistical
  Science}, \textbf{33}, 469--472.

\bibitem[{Samworth and Shah(2024+)}]{samworth2024statistics}
Samworth, R.~J. and Shah, R.~D. (2024+) \emph{Modern Statistical Methods and
  Theory}. Cambridge University Press, to appear.

\bibitem[{Temlyakov(2008)}]{temlyakov2008approximation}
Temlyakov, V. (2008) Approximation in learning theory. \emph{Constructive
  Approximation}, \textbf{27}, 33--74.

\bibitem[{Tropp(2012)}]{tropp2012user}
Tropp, J.~A. (2012) User-friendly tail bounds for sums of random matrices.
  \emph{Foundations of computational mathematics}, \textbf{12}, 389--434.

\bibitem[{Tsybakov(2009)}]{tsybakov2009introduction}
Tsybakov, A.~B. (2009) \emph{Introduction to Nonparametric Estimation}.
  Springer.

\bibitem[{Tukey(1975)}]{tukey1975mathematics}
Tukey, J.~W. (1975) Mathematics and the picturing of data. In \emph{Proceedings
  of the International Congress of Mathematicians}, vol.~2, 523--531.

\bibitem[{Vershynin(2018)}]{vershynin2018high}
Vershynin, R. (2018) \emph{{High-Dimensional Probability: An Introduction with
  Applications in Data Science}}. Cambridge University Press.

\bibitem[{Wainwright(2019)}]{wainwright2019high}
Wainwright, M.~J. (2019) \emph{High-Dimensional Statistics: A Non-Asymptotic
  Viewpoint}. Cambridge University Press.

\bibitem[{Wang and Ramdas(2024)}]{wang2024anytime}
Wang, H. and Ramdas, A. (2024) Anytime-valid {$t$}-tests and confidence
  sequences for {G}aussian means with unknown variance. \emph{arXiv preprint
  arXiv:2310.03722}.

\bibitem[{Zhivotovskiy(2024)}]{zhivotovskiy2024dimension}
Zhivotovskiy, N. (2024) Dimension-free bounds for sums of independent matrices
  and simple tensors via the variational principle. \emph{Electronic Journal of
  Probability}, \textbf{29}, 1--28.

\bibitem[{Zhu, Wang and Samworth(2022)}]{zhu2022high}
Zhu, Z., Wang, T. and Samworth, R.~J. (2022) High-dimensional principal
  component analysis with heterogeneous missingness. \emph{Journal of the Royal
  Statistical Society Series B: Statistical Methodology}, \textbf{84},
  2000--2031.

\end{thebibliography}
}

\appendix

\section{Minimax risk lower bounds}

In this section, we state some general techniques for minimax risk lower bounds from \cite{samworth2024statistics}; similar statements can be found in the books by \citet{tsybakov2009introduction} and \citet{wainwright2019high}.  We work in the setting of Section~\ref{sec:setup}. 

\begin{lemma}[Le Cam's two-point lemma] \label{lemma:le-cam}
    Let $\theta_1,\theta_2 \in \Theta$, $P_1 \in \mathcal{P}_{\theta_1}$ and $P_2 \in \mathcal{P}_{\theta_2}$. Then, writing $\eta \coloneqq \frac{1}{2} d(\theta_1,\theta_2)$ and $\Theta_0 \coloneqq \{\theta_1,\theta_2\}$, we have
    \begin{align*}
        \inf_{\hat{\theta} \in \hat{\Theta}} \sup_{\theta \in \Theta} \sup_{P_{\theta} \in \mathcal{P}_{\theta}} \mathbb{E}_{P_{\theta}} L(\hat{\theta},\theta) \geq \inf_{\hat{\theta} \in \hat{\Theta}} \max_{\theta_0 \in \Theta_0} \sup_{P_{\theta_0} \in \mathcal{P}_{\theta_0}} \mathbb{E}_{P_{\theta_0}} L(\hat{\theta},\theta_0) \geq \frac{g(\eta)}{2} \bigl\{ 1 - \mathrm{TV}(P_1,P_2) \bigr\}.
    \end{align*}
\end{lemma}

\begin{lemma}[Fano's lemma] \label{lemma:fano}
    Let $M\geq 2$, $\theta_1,\ldots,\theta_M \in \Theta$, and $P_j \in \mathcal{P}_{\theta_j}$ for $j\in[M]$. Then, writing $\eta \coloneqq \frac{1}{2} \min_{1\leq j < k \leq M} d(\theta_j,\theta_k)$ and $\Theta_0 \coloneqq \{\theta_j : j\in[M]\}$, we have
    \begin{align*}
        \inf_{\hat{\theta} \in \hat{\Theta}} \sup_{\theta \in \Theta} \sup_{P_{\theta} \in \mathcal{P}_{\theta}} \mathbb{E}_{P_{\theta}} L(\hat{\theta},\theta) &\geq \inf_{\hat{\theta} \in \hat{\Theta}} \max_{\theta_0 \in \Theta_0} \sup_{P_{\theta_0} \in \mathcal{P}_{\theta_0}} \mathbb{E}_{P_{\theta_0}} L(\hat{\theta},\theta_0)\\
        &\geq g(\eta) \biggl\{ 1 - \frac{M^{-1} \inf_{Q\in\mathcal{Q}(\mathcal{X})}\sum_{j=1}^M \mathrm{KL}(P_j,Q) + \log(2-M^{-1})}{\log M} \biggr\}.
    \end{align*}
\end{lemma}

\begin{lemma}[Assouad's lemma] \label{lemma:assouad}
    Let $m\in\mathbb{N}$, $\Phi\coloneqq \{0,1\}^m$ For $\phi\in\Phi$, let $\theta_{\phi} \in \Theta$ and $P_{\phi} \in \mathcal{P}_{\theta_{\phi}}$. For $\phi,\phi' \in \Phi$, write $\phi \sim \phi'$ whenever $\phi$ and $\phi'$ differ in precisely one coordinate, and $\phi \sim_j \phi'$ when that coordinate is the $j$th. Suppose now that the loss function is of the form
    \begin{align*}
        L(\theta_1,\theta_2) \coloneqq \sum_{j=1}^m g\bigl(d_j(\theta_1,\theta_2)\bigr),
    \end{align*}
    for $\theta_1,\theta_2 \in \Theta$, where $d_1,\ldots,d_m$ are pseudo-metrics on $\Theta$ with $d_j(\theta_{\phi}, \theta_{\phi'}) \geq \alpha_j$ whenever $\phi \sim_j \phi'$, and where $g$ is an increasing function satisfying $g(x+y) \leq A\{g(x) + g(y)\}$ for all $x,y\in[0,\infty)$ and some $A>0$. Then, writing $\Theta_0 \coloneqq \{ \theta_{\phi} : \phi\in\Phi \}$, we have
    \begin{align*}
        \inf_{\hat{\theta} \in \hat{\Theta}} \sup_{\theta \in \Theta} \sup_{P_{\theta} \in \mathcal{P}_{\theta}} \mathbb{E}_{P_{\theta}} L(\hat{\theta},\theta) &\geq \inf_{\hat{\theta} \in \hat{\Theta}} \max_{\theta_0 \in \Theta_0} \sup_{P_{\theta_0} \in \mathcal{P}_{\theta_0}} \mathbb{E}_{P_{\theta_0}} L(\hat{\theta},\theta_0)\\
        &\geq \frac{1}{2A} \Bigl\{ 1 - \max_{\phi,\phi'\in\Phi : \phi\sim\phi'} \mathrm{TV}(P_{\phi}, P_{\phi'}) \Bigr\} \sum_{j=1}^d g(\alpha_j).
    \end{align*}
\end{lemma}

\section{Auxiliary lemmas}

Our first auxiliary lemma is a minor modification of~\citet[Theorem 5.1]{chen2018robust} that we apply in the proof of Corollary~\ref{cor:huber-mean}.  To introduce the setting, let $\mathcal{Y}$ be a measurable space; for $\theta \in \Theta$, let $\mathcal{R}_{\theta} \subseteq \mathcal{Q}(\mathcal{Y})$ and let $\mathcal{R}_{\Theta} \coloneqq \{\mathcal{R}_{\theta} : \theta\in\Theta\}$.

\begin{defn}
\label{defn:modulus-of-continuity}
    Let $\varepsilon \in [0,1)$. We define the \emph{modulus of continuity} of $\mathcal{R}_{\Theta}$ with respect to $\varepsilon$ and $g$ to be \
    \begin{align*} 
        \omega(\varepsilon, \mathcal{R}_{\Theta}, g) \coloneqq \sup\, \biggl\{g\biggl( \frac{d(\theta_1, \theta_2)}{2} \biggr) :\; &\theta_1, \theta_2 \in \Theta \text{ and there exist } R_{\theta_1} \in \mathcal{R}_{\theta_1}, \nonumber\\
        &R_{\theta_2} \in \mathcal{R}_{\theta_2} \text{ such that } \mathrm{TV}(R_{\theta_1}, R_{\theta_2}) \leq \frac{\varepsilon}{1 - \varepsilon} \biggr\}.
    \end{align*}
\end{defn}
\begin{lemma} \label{thm:general-lower-bound-huber}
    For $n \in \mathbb{N}$, $\theta \in \Theta$ and $\varepsilon \in [0,1)$, let $\mathcal{X} = \mathcal{Y}^n$ and
    \begin{align*}
    \mathcal{P}_{\theta} \equiv \mathcal{P}_{\theta, \varepsilon} \coloneqq \bigl\{P^{\otimes n} \in \mathcal{Q}(\mathcal{X}):\, P = (1-\varepsilon)R_{\theta} + \varepsilon Q: R_{\theta} \in \mathcal{R}_{\theta}\, \text{ and } Q \in \mathcal{Q}(\mathcal{Y}) \bigr\}
    \end{align*}
    denote the class of contaminated distributions.  Then, for all $\delta \in (0, 1/2)$, 
    \[
    \mathcal{M}_{-}(\delta) \geq \omega(\varepsilon, \mathcal{R}_{\Theta}, g). 
    \]

\end{lemma}

\begin{proof}
    We apply the same proof strategy as \citet[Theorem 5.1]{chen2018robust}.  Consider any $R_1 \in \mathcal{R}_{\theta_1}$ and $R_2 \in \mathcal{R}_{\theta_2}$ that satisfy $\mathrm{TV} (R_1, R_2) \leq \varepsilon / (1 - \varepsilon)$. Then there exists $\varepsilon' \in [0,\varepsilon]$ such that $\mathrm{TV}(R_1, R_2) = \varepsilon'/(1 - \varepsilon')$.  
    Next, for $\ell \in \{1, 2\}$, define densities 
    \begin{align*}
        r_\ell \coloneqq \frac{\mathrm{d}R_\ell}{\mathrm{d}R_1 + \mathrm{d}R_2}. 
    \end{align*} 
    We also define probability distributions $Q_1$ and $Q_2$ on $\mathcal{X}$ through their densities 
    \begin{align*}
        &\frac{\mathrm{d}Q_1}{\mathrm{d}R_1 + \mathrm{d}R_2} \coloneqq \frac{(r_2 - r_1) \mathbbm{1}_{\{r_2 \geq r_1\}}}{\mathrm{TV}(R_1, R_2)} \quad\text{and} \quad\frac{\mathrm{d}Q_2}{\mathrm{d}R_1 + \mathrm{d}R_2} \coloneqq \frac{(r_1 - r_2) \mathbbm{1}_{\{r_1 \geq r_2\}}}{\mathrm{TV}(R_1, R_2)}.
    \end{align*}
    Now define $P_\ell \coloneqq (1-\varepsilon') R_\ell + \varepsilon' Q_\ell$ for $\ell \in \{1,2\}$.  Then $P_{\ell} \in \mathcal{P}_{\theta_{\ell}}$ for $\ell \in \{1,2\}$. Moreover, direct calculation yields 
    \begin{align*}
        \frac{\mathrm{d} P_1}{\mathrm{d}R_1 + \mathrm{d}R_2} &= (1-\varepsilon')r_1 + \varepsilon' \frac{(r_2 - r_1) \mathbbm{1}_{\{r_2 \geq r_1\}}}{\varepsilon' / (1-\varepsilon')} = (1-\varepsilon')\bigl\{r_1 + (r_2 - r_1) \mathbbm{1}_{\{r_2 \geq r_1\}}\bigr\} \\
        &= (1-\varepsilon')\bigl\{r_2 + (r_1 - r_2) \mathbbm{1}_{\{r_1 \geq r_2\}}\bigr\} = \frac{\mathrm{d} P_2}{\mathrm{d} R_1 + \mathrm{d} R_2},
    \end{align*}
    so $P_1 = P_2$.  Thus, by Lemma~\ref{lemma:high-prob-le-cam-tv},
    \[
    \inf_{\hat{\theta} \in \hat{\Theta}} \sup_{\theta \in \Theta} \sup_{P \in \mathcal{P}_{\theta}} \; P \biggl\{ L(\hat{\theta}, \theta) \geq g\biggl(\frac{d(\theta_1, \theta_2)}{2}\biggr)\biggr\} \geq \frac{1}{2} \bigl\{ 1 - \mathrm{TV}(P_1^{\otimes n}, P_2^{\otimes n}) \bigr\} = \frac{1}{2},
    \]
    where the inequality follows since $P_1^{\otimes n}, P_2^{\otimes n} \in \mathcal{P}_{\Theta}$.  Further, since $R_1, R_2 \in \mathcal{R}_{\Theta}$ with $\mathrm{TV} (R_1, R_2) \leq \varepsilon / (1 - \varepsilon)$ were arbitrary, for any $t > 0$, we can find $\theta^{(t)}_1, \theta^{(t)}_2 \in \Theta$, $R_1^{(t)} \in \mathcal{R}_{\theta_1}^{(t)}$ and $R_2^{(t)} \in \mathcal{R}_{\theta_2}^{(t)}$ such that $\mathrm{TV}(R_1^{(t)}, R_2^{(t)}) \leq \varepsilon/(1-\varepsilon)$ and $g\bigl(d(\theta_{1}^{(t)}, \theta_2^{(t)})/2\bigr) \geq \omega(\varepsilon, \mathcal{R}_{\Theta}, g) - t$.  Hence, by taking $t \to 0$, we deduce that
    \begin{align*}
        \inf_{\hat{\theta} \in \hat{\Theta}} \sup_{\theta \in \Theta} \sup_{P \in \mathcal{P}_{\theta}} P \biggl\{ L\big(\hat{\theta},\, \theta \big) \geq \omega(\varepsilon, \mathcal{R}_{\Theta}, g) \biggr\} \geq \frac{1}{2},
    \end{align*}
    so the result follows.
\end{proof}

Lemma~\ref{lemma:matrix-bernstein} below applies a non-central version of the matrix Bernstein inequality \citep{tropp2012user,zhu2022high} to provide a high-probability upper bound on the operator norm of the difference between a sample covariance matrix and its population counterpart.  
\begin{lemma} \label{lemma:matrix-bernstein}
    Let $n,d\in\mathbb{N}$, $\Sigma\in\mathcal{S}_+^{d\times d}$ and $X_1,\ldots,X_n \overset{\mathrm{iid}}{\sim} \mathsf{N}_d(0,\Sigma)$. Then 
    \begin{align} \label{eq:matrix-bernstein-variance-term}
        \biggl\| \frac{1}{n}\sum_{i=1}^n \mathbb{E} \bigl\{(X_iX_i^\top - \Sigma)^2\bigr\} \biggr\|_{\mathrm{op}} \asymp \biggl\| \frac{1}{n}\sum_{i=1}^n \mathbb{E} \bigl\{(X_iX_i^\top)^2\bigr\} \biggr\|_{\mathrm{op}} \asymp \|\Sigma\|_{\mathrm{op}}^2 \mathbf{r}(\Sigma).
    \end{align}
    Moreover, writing $\hat{\Sigma} \coloneqq n^{-1}\sum_{i=1}^n X_iX_i^\top$, if $\frac{\mathbf{r}(\Sigma)\{\log(1/\delta) + \log (8d)\}}{n} \leq 1$, then with probability at least $1-\delta$, we have
    \begin{align} \label{eq:matrix-bernstein-ub}
        \|\hat{\Sigma} - \Sigma\|_{\mathrm{op}} \leq 513\|\Sigma\|_{\mathrm{op}} \sqrt{\frac{\mathbf{r}(\Sigma)\{\log(1/\delta) + \log (8d)\}}{n}}.
    \end{align}
\end{lemma}
\begin{remark}
From the bounds on the variance parameter in~\eqref{eq:matrix-bernstein-variance-term}, we see that the upper bound in~\eqref{eq:matrix-bernstein-ub} is the tightest (up to multiplicative constants) that one can obtain from the matrix Bernstein inequality in the sub-Gaussian regime (i.e. when $\frac{\mathbf{r}(\Sigma)\{\log(1/\delta) + \log (8d)\}}{n} \leq 1$).  In particular, the $\log(1/\delta)$ term in~\eqref{eq:matrix-bernstein-ub} appears multiplicatively with the effective rank, in stark contrast with Proposition~\ref{prop:covariance-matrix-estimation}, where the $\log(1/\delta)$ term appears only additively.      
\end{remark}

\begin{proof}
    Since the operator norm is invariant to orthogonal transformations, we may assume without loss of generality that $\Sigma$ is diagonal, so that $\Sigma = \mathrm{diag}(\lambda_1,\ldots,\lambda_d)$ where $\lambda_1 \geq \cdots \geq \lambda_d \geq 0$. Then, writing $X_1 = (X_{11}, \ldots, X_{1d}) \in \mathbb{R}^d$, we have 
    \begin{align*}
        \mathbb{E} \bigl\{(X_1X_1^\top - \Sigma)^2\bigr\} &= \mathbb{E} \bigl\{(X_1X_1^\top)^2\bigr\} - \Sigma^2\\
        &= \mathbb{E} \bigl(\|X_1\|_2^2 X_1X_1^\top\bigr) - \Sigma^2\\
        &= \mathrm{diag}\bigl( \mathbb{E}(\|X_1\|_2^2X_{11}^2), \ldots, \mathbb{E}(\|X_1\|_2^2 X_{1d}^2) \bigr) - \mathrm{diag}\bigl( \lambda_1^2, \ldots, \lambda_d^2 \bigr)\\
        &= \mathrm{diag}\biggl( 2\lambda_1^2 + \lambda_1 \sum_{j\neq 1}\lambda_j, \ldots, 2\lambda_d^2 + \lambda_d \sum_{j\neq d} \lambda_j \biggr).
    \end{align*}
    Therefore,
    \begin{align*}
        \|\Sigma\|_{\mathrm{op}} \tr(\Sigma) \leq \biggl\| \frac{1}{n}\sum_{i=1}^n \mathbb{E} \bigl\{(X_iX_i^\top - \Sigma)^2\bigr\} \biggr\|_{\mathrm{op}} \leq 2\|\Sigma\|_{\mathrm{op}} \tr(\Sigma).
    \end{align*}
    Similarly,
    \begin{align*}
        \|\Sigma\|_{\mathrm{op}} \tr(\Sigma) \leq \biggl\| \frac{1}{n}\sum_{i=1}^n \mathbb{E} \bigl\{(X_iX_i^\top)^2\bigr\} \biggr\|_{\mathrm{op}} \leq 3\|\Sigma\|_{\mathrm{op}} \tr(\Sigma),
    \end{align*}
    and~\eqref{eq:matrix-bernstein-variance-term} follows.

    
    For the second claim, observe that for $p\geq 2$, we have
    \begin{align*}
        \mathbb{E} \bigl\{(X_1X_1^\top)^p\bigr\} 
        &= \mathbb{E} \bigl(\|X_1\|_2^{2p-2} X_1X_1^\top \bigr) \\
        &= \mathrm{diag}\bigl( \mathbb{E}(\|X_1\|_2^{2p-2}X_{11}^2), \ldots, \mathbb{E}( \|X_1\|_2^{2p-2} X_{1d}^2) \bigr) \\
        &\preceq \bigl\{\mathbb{E}(\|X_1\|_2^{4p-4})\bigr\}^{1/2} \mathrm{diag}\bigl(\bigl\{\mathbb{E}( X_{11}^4)\bigr\}^{1/2},\ldots,\bigl\{\mathbb{E}( X_{1d}^4)\bigr\}^{1/2}\bigr) \\
        \overset{(i)}&{\preceq} \Bigl(\frac{8}{3}\Bigr)^{p-1} \cdot (2p-2)^{p-1} \tr(\Sigma)^{p-1} \cdot 3^{1/2}\Sigma \\
        \overset{(ii)}&{\preceq}  \Bigl(\frac{16e\tr(\Sigma)}{3}\Bigr)^{p-1} \cdot \frac{1}{p\sqrt{p-1}} \cdot p! \cdot \Sigma \\
        &\preceq \frac{p!}{2}\cdot \bigl(16\tr(\Sigma)\bigr)^{p-2} \cdot  16\tr(\Sigma) \cdot \Sigma.
    \end{align*}
    Here, Step $(i)$ follows by recalling the sub-exponential norm of a random variable $Z$, given by $\| Z \|_{\psi_1} \coloneqq \inf\{t > 0: \mathbb{E}(e^{|Z|/t}) \leq 2\}$, so that $\bigl\| \|X_1\|_2^2 \bigr\|_{\psi_1} \leq \sum_{j=1}^d \|X_{1j}^2\|_{\psi_1} = 8\tr(\Sigma)/3$.  Moreover, $(y/r)^r \leq e^y$ for all $y,r \geq 0$, so 
    \[
    \mathbb{E} \biggl\{\biggl( \frac{\|X_1\|_2^2}{(2p-2)\cdot 8\tr(\Sigma)/3} \biggr)^{2p-2}\biggr\} \leq \mathbb{E} \exp\biggl(\frac{\|X_1\|_2^2}{8\tr(\Sigma)/3}\biggr) \leq 2.
    \]
    Step $(ii)$ follows from Stirling's formula~\citep[e.g.,][Eq.~(1)]{robbins1955remark}. 
Noting that $X_1X_1^\top$ is positive semi-definite, we may apply a non-central version of the matrix Bernstein inequality~\citep[Lemma~3]{zhu2022high} to both $\hat{\Sigma}$ and $-\hat{\Sigma}$ to deduce that, for all $t > 0$,
\[
\mathbb{P}\bigl( \| \hat{\Sigma} - \Sigma \|_{\mathrm{op}} \geq t \bigr) \leq 8d \exp\biggl\{-\frac{nt^2}{2^9(\| \Sigma \|_{\mathrm{op}} + t) \tr(\Sigma)}\biggr\}.
\]
The result then follows on setting $t = (2^9+1)\|\Sigma\|_{\mathrm{op}} \sqrt{\frac{\mathbf{r}(\Sigma)\{\log(1/\delta) + \log (8d)\}}{n}}$ and invoking the assumption $\frac{\mathbf{r}(\Sigma)\{\log(1/\delta) + \log (8d)\}}{n} \leq 1$. 
\end{proof}

The following lemma provides a lower bound on the packing number of the set of $s$-sparse vectors in $\mathbb{S}^{d-1}$.  Very similar results exist in the literature (e.g.~\citet[Lemma~3]{candes2013well} or~\citet[Exercise~5.8]{wainwright2019high}); we present a short, constructive proof without imposing a parity restriction on the sparsity level $s$.
\begin{lemma}
    \label{lemma:GilbertVarshamov}
Fix $d \in \mathbb{N}$ and $s \in [d]$, and let $\mathcal{V} \coloneqq \{v \in \mathbb{S}^{d-1}:\|v\|_0 \leq s\}$.  There exist $\log M \geq \frac{3s}{4}\log\bigl(\frac{d}{4s}\bigr)$ and $v_1,\ldots,v_M \in \mathcal{V}$ such that $\|v_j - v_k\|_2 > 1/2$ for all distinct $j,k \in [M]$.    
\end{lemma}
\begin{proof}
Let $L \coloneqq \lfloor s/4 \rfloor$ and let $\Omega_1 \coloneqq \bigl\{\omega \in \{0,1\}^d:\|\omega\|_0 \leq s\bigr\}$.  For each $j \in \mathbb{N}$, if $\Omega_j \neq \emptyset$, fix an arbitrary $\omega_j \in \Omega_j$, let
\[
  \Omega_{j+1} \coloneqq \bigl\{\omega \in
\Omega_j:d_{\mathrm{H}}(\omega,\omega_j) > s/4\bigr\},
\]
and let $M \coloneqq \max\{j \in \mathbb{N}:\Omega_j \neq \emptyset\}$.  Then    
$|\Omega_j \setminus \Omega_{j+1}| \leq \sum_{\ell=0}^L  \binom{d}{\ell}$
for each $j \in [M]$.  Since the sets $\{\Omega_j \setminus \Omega_{j+1}:j
\in [M]\}$ form a partition of $\Omega_1$, we deduce that
\[
  \sum_{\ell=0}^s \binom{d}{\ell} = \sum_{j=1}^M |\Omega_j \setminus \Omega_{j+1}| \leq
M\sum_{\ell=0}^L  \binom{d}{\ell}.
\]
Hence
\[
M \geq \frac{\sum_{\ell=0}^s \binom{d}{\ell}}{\sum_{\ell=0}^L  \binom{d}{\ell}} \geq \frac{\binom{d}{s}}{(L+1)\binom{d}{L}} \geq \frac{(d/s)^s}{(L+1)(ed/L)^L}.
\]
It follows that
\begin{align*}
\log M &\geq s \log\Bigl(\frac{d}{s}\Bigr) - L\log\Bigl(\frac{ed}{L}\Bigr) - \log(L+1) \geq s \log\Bigl(\frac{d}{s}\Bigr) - \frac{s}{4}\log\Bigl(\frac{4ed}{s}\Bigr) - \log\Bigl(\frac{s}{4}+1\Bigr) \\
&\geq \frac{3s}{4}\log\Bigl(\frac{d}{s}\Bigr) - \frac{s}{4}\log(4e^2) = \frac{3s}{4}\log\Bigl(\frac{d}{(4e^2)^{1/3}s}\Bigr) \geq \frac{3s}{4}\log\Bigl(\frac{d}{4s}\Bigr).
\end{align*}
The points $\omega_1,\ldots,\omega_M$ satisfy
$d_{\mathrm{H}}(\omega_j,\omega_k) > s/4$ for all distinct $j,k \in [M]$,
and the result follows by taking, $v_j = \omega_{j}/\sqrt{s}$ for $j \in [M]$, so that $v_j \in \mathcal{V}$ for all $j$.
\end{proof}
Recall that in the context of kernel density estimation, a \emph{kernel} is a Borel measurable function $K:\mathbb{R} \rightarrow \mathbb{R}$ with $\int_{-\infty}^\infty K(u) \, du = 1$.  It is of \emph{order} $\ell \in \mathbb{N}$ if $\int_{-\infty}^\infty u^j K(u) \, du = 0$ for all $j \in [\ell-1]$.
\begin{lemma} \label{lemma:kde-ub}
    Let $\beta, \gamma > 0$, let $X_1,\ldots,X_n \overset{\mathrm{iid}}{\sim} f \in \mathcal{F}(\beta, \gamma)$, let $x_0 \in \mathbb{R}$ and let $K:\mathbb{R} \to \mathbb{R}$ be a kernel of order $\ell \coloneqq \lceil \beta \rceil$ satisfying $\|K\|_{\infty} < \infty$ and
    \begin{align*}
        \mu_{\beta}(K) \coloneqq \int_{-\infty}^{\infty} |u|^{\beta} |K(u)| \,\mathrm{d}u < \infty.
    \end{align*}
    There exist $c(\beta,K), C(\beta,K) > 0$, depending only on $\beta,K$, and a choice of bandwidth $h \equiv h_{n,\beta,\gamma,K,\delta}$ such that if $\delta \in \bigl[ 2e^{-c(\beta,K) \cdot n}, 1\bigr]$,
    then with probability at least $1-\delta$,
    \begin{align*}
        \bigl\{ \hat{f}_n(x_0) - f(x_0) \bigr\}^2 \leq C(\beta,K) \cdot \gamma^{2/(\beta+1)} \biggl( \frac{\log(2/\delta)}{n} \biggr)^{2\beta/(2\beta+1)},
    \end{align*}
    where $\hat{f}_n \equiv \hat{f}_{n,h,K}$ is the kernel density estimator with bandwidth $h$ and kernel $K$.
\end{lemma}

\begin{proof}
    By~\citet[Proposition 1.2]{tsybakov2009introduction}, we have
    \begin{align} \label{eq:bias-kde}
        \bigl| \mathbb{E} \hat{f}_n(x_0) - f(x_0) \bigr| \leq \frac{\gamma}{(\ell-1)!} \mu_{\beta}(K) h^{\beta}.
    \end{align}
    Moreover, by~\citet[Proposition 1.1]{tsybakov2009introduction} and the first display in the proof of \citet[Theorem 4.6]{samworth2024statistics}, we have
    \begin{align*}
        \Var \bigl( K_h(X_1 - x_0) \bigr) \leq \frac{1}{h} \|f\|_{\infty} R(K) \leq \frac{1}{h} C_1(\beta) \gamma^{1/(\beta+1)} R(K),
    \end{align*}
    where $R(K) \coloneqq \int_{-\infty}^{\infty} K^2(u)\,\mathrm{d}u$ and where $C_1(\beta) > 0$ depends only on $\beta$.  Now by Bernstein's inequality
    , we have for any $\delta \in (0,1]$ that
    \[
    \mathbb{P}\biggl\{\bigl| \hat{f}_n(x_0) - \mathbb{E} \hat{f}_n(x_0) \bigr| > \biggl(\frac{2C_1(\beta)\gamma^{1/(\beta+1)}R(K)\log(2/\delta)}{nh}\biggr)^{1/2} + \frac{\|K\|_\infty \log(2/\delta)}{3nh}\biggr\} \leq \delta.
    \]
    Hence, when
    \begin{align} \label{eq:kde-ub-condition-on-delta}
        \log(2/\delta) \leq nh\biggl( \frac{18C_1(\beta) \gamma^{1/(\beta+1)} R(K)}{\|K\|_{\infty}^2} \biggr),
    \end{align} 
    we have with probability at least $1-\delta$ that
    \begin{equation}
    \label{Eq:BernsteinCombined}
        \bigl| \hat{f}_n(x_0) - \mathbb{E} \hat{f}_n(x_0) \bigr| \leq 2\sqrt{\frac{2C_1(\beta) \gamma^{1/(\beta+1)} R(K) \log(2/\delta)}{nh}}.
    \end{equation}
    Combining~\eqref{Eq:BernsteinCombined} with \eqref{eq:bias-kde}, we deduce that when $\delta$ satisfies~\eqref{eq:kde-ub-condition-on-delta}, we have with probability at least $1-\delta$ that
    \begin{align*}
        \bigl| \hat{f}_n(x_0) - f(x_0) \bigr| &\leq \bigl| \hat{f}_n(x_0) - \mathbb{E} \hat{f}_n(x_0) \bigr| + \bigl| \mathbb{E} \hat{f}_n(x_0) - f(x_0) \bigr|\\
        &\leq 2\sqrt{\frac{2C_1(\beta) \gamma^{1/(\beta+1)} R(K) \log(2/\delta)}{nh}} + \frac{\gamma}{(\ell-1)!} \mu_{\beta}(K) h^{\beta}.
    \end{align*}
    Now set
    \begin{align*}
        C_2(\beta,K) \coloneqq \biggl( \frac{C_1(\beta)R(K)\{(\ell-1)!\}^2}{\mu_{\beta}^2(K)} \biggr)^{1/(2\beta+1)},
    \end{align*}
    and 
    \begin{align*}
        h \equiv h_{n,\beta,\gamma,K,\delta} \coloneqq C_2(\beta,K) \cdot \gamma^{-1/(\beta+1)} \biggl( \frac{\log(2/\delta)}{n} \biggr)^{1/(2\beta+1)}.
    \end{align*}
    Then there exist $c(\beta,K), C(\beta,K)>0$ such that 
    \begin{align*}
        \bigl\{ \hat{f}_n(x_0) - f(x_0) \bigr\}^2 \leq C(\beta,K) \cdot \gamma^{2/(\beta+1)} \biggl( \frac{\log(2/\delta)}{n} \biggr)^{2\beta/(2\beta+1)},
    \end{align*}
    with probability at least $1-\delta$, whenever $\delta \geq 2e^{-c(\beta,K) \cdot n}$.
\end{proof}

The bandwidth in Lemma~\ref{lemma:kde-ub} depends on $\delta$, which is undesirable. The following lemma uses a Lepski type method to obtain a $\delta$-independent estimator for $f(x_0)$. The same method was employed by \citet[Theorem 4.2]{devroye2016subgaussian} for heavy-tailed mean estimators.
\begin{lemma} \label{lemma:kde-ub-lepski}
    Let $\beta, \gamma > 0$, let $X_1,\ldots,X_n \overset{\mathrm{iid}}{\sim} f \in \mathcal{F}(\beta, \gamma)$, let $x_0 \in \mathbb{R}$ and let $K:\mathbb{R} \to \mathbb{R}$ be a kernel of order $\ell \coloneqq \lceil \beta \rceil$ satisfying $\|K\|_{\infty} < \infty$ and
    \begin{align*}
        \mu_{\beta}(K) \coloneqq \int_{-\infty}^{\infty} |u|^{\beta} |K(u)| \,\mathrm{d}u < \infty.
    \end{align*}
    There exists a $\delta$-independent estimator $\check{f}(x_0) \equiv \check{f}_{n,\beta,\gamma,K}(x_0)$ and $\check{C}(\beta,K) > 0$ depending only on $\beta,K$ such that for all $\delta\in (0,1]$, we have with probability at least $1-\delta$ that
    \begin{align*}
        \bigl\{ \check{f}_n(x_0) - f(x_0) \bigr\}^2 \leq \check{C}(\beta,K) \cdot \gamma^{2/(\beta+1)} \cdot \biggl\{ \biggl( \frac{\log(8/\delta)}{n} \biggr)^{2\beta/(2\beta+1)} \wedge 1\biggr\}.
    \end{align*}
\end{lemma}

\begin{proof}
    Let $c(\beta,K) > 0$ be as in Lemma~\ref{lemma:kde-ub}.
    First consider the case $n\leq 3\log(2)/c(\beta,K)$. We set $\tilde{f}_n(x_0) = 0$, and recall from the first display in the proof of \citet[Theorem 4.6]{samworth2024statistics} that $\|f\|_{\infty} \leq C_1(\beta) \gamma^{1/(\beta+1)}$ for some $C_1(\beta) > 0$ depending only on $\beta$. Therefore, when $n\leq 3\log(2)/c(\beta,K)$, we have
    \begin{align} \label{eq:kde-lepski-small-n}
        \bigl\{ \tilde{f}_n(x_0) - f(x_0) \bigr\}^2 \leq C_1(\beta)^2 \gamma^{2/(\beta+1)} \leq C_2(\beta,K) \cdot \gamma^{2/(\beta+1)} \cdot \biggl\{ \biggl( \frac{\log(8/\delta)}{n} \biggr)^{2\beta/(2\beta+1)} \wedge 1\biggr\},
    \end{align}
    for some $C_2(\beta,K) > 0$ depending only on $\beta$ and $K$. This proves the claim for $n\leq 3\log(2)/c(\beta,K)$.

    Next consider the case where $n> 3\log(2)/c(\beta,K)$.
    Let $\hat{f}_{n,\delta}(x_0)$ denote the kernel density estimator of $f(x_0)$ with bandwidth $h_{n,\beta,\gamma,K,\delta}$ as in Lemma~\ref{lemma:kde-ub}.
    Define $k_+ \coloneqq \lfloor c(\beta,K)\cdot n / \log(2) - 1 \rfloor \in \mathbb{N}$, which by assumption, satisfies $k_+ \geq c(\beta,K)\cdot n / (2\log 2)$.
    For $k\in[k_+]$, define the interval
    \begin{align*}
        \hat{I}_{k} \coloneqq \Biggl[ \hat{f}_{n,2^{-k}}(x_0) - &C(\beta,K)^{1/2} \cdot \gamma^{1/(\beta+1)} \biggl( \frac{\log(2^{k+1})}{n} \biggr)^{\beta/(2\beta+1)},\\
        &\hat{f}_{n,2^{-k}}(x_0) + C(\beta,K)^{1/2} \cdot \gamma^{1/(\beta+1)} \biggl( \frac{\log(2^{k+1})}{n} \biggr)^{\beta/(2\beta+1)} \Biggr].
    \end{align*}
    By Lemma~\ref{lemma:kde-ub}, we have for each $k \in [k_+]$ that $f(x_0) \in \hat{I}_k$ with probability at least $1-2^{-k}$.  Now define $\hat{k}\coloneqq \min\bigl\{ k\in[k_+] : \bigcap_{j=k}^{k_+} \hat{I}_j \neq \emptyset \bigr\}$.  Define the estimator 
    \[
    \tilde{f}_n(x_0) \coloneqq \begin{cases} 
    \inf \bigl\{\bigcap_{j=\hat{k}}^{k_+} \hat{I}_j \bigr\} & \text{ if } \bigcap_{j=\hat{k}}^{k_+} \hat{I}_j \neq \emptyset \\
    0 & \text{ otherwise},
    \end{cases}
    \]
    and set $\check{f}_n(x_0) \coloneqq 0 \vee \tilde{f}_n(x_0) \wedge C_1(\beta) \gamma^{1/(\beta+1)}$.  For $\delta \in [2^{-(k_+-1)}, 1]$, we set $k_{\delta}\coloneqq \min \bigl\{ k\in[k_+] : \delta \geq 2^{-(k-1)} \bigr\}$, so that $2^{k_\delta - 2} < 1/\delta \leq 2^{k_\delta-1}$. Then, by a union bound,
    \begin{align*}
        \mathbb{P}\biggl( f(x_0) \in \bigcap_{j=k_{\delta}}^{k_+} \hat{I}_j \biggr) \geq 1 - \sum_{j=k_{\delta}}^{k_+} 2^{-j} \geq 1 - 2^{-(k_{\delta}-1)} \geq 1-\delta.
    \end{align*}
    On this event, $\bigcap_{j=k_{\delta}}^{k_+} \hat{I}_j \neq \emptyset$, so $\hat{k} \leq k_{\delta}$. Hence, for all $\delta \in [2^{-(k_+-1)}, 1]$, we have with probability at least $1-\delta$ that $\tilde{f}_n(x_0) \in \hat{I}_{k_{\delta}}$, so that with probability at least $1-\delta$,
    \begin{align*} 
        \bigl\{\tilde{f}_n(x_0) - f(x_0) \bigr\}^2 &\leq 4C(\beta,K) \cdot \gamma^{2/(\beta+1)} \biggl( \frac{\log(2^{k_{\delta}+1})}{n} \biggr)^{2\beta/(2\beta+1)}\nonumber\\
        &\leq 4C(\beta,K) \cdot \gamma^{2/(\beta+1)} \biggl( \frac{\log(8/\delta)}{n} \biggr)^{2\beta/(2\beta+1)}.
    \end{align*}
    Since $f(x_0) \in \bigl[0,\, C_1(\beta)\gamma^{1/(\beta+1)}\bigr]$, it follows that when $\delta \in [2^{-(k_+-1)}, 1]$, we have with probability at least $1-\delta$ that
    \begin{align} \label{eq:kde-ub-term-1}
    \bigl\{\check{f}_n(x_0) - f(x_0)\bigr\}^2 \leq \bigl\{\tilde{f}_n(x_0) - f(x_0)\bigr\}^2 \leq 4C(\beta,K) \cdot \gamma^{2/(\beta+1)} \biggl( \frac{\log(8/\delta)}{n} \biggr)^{2\beta/(2\beta+1)}.
    \end{align}
    On the other hand, for $\delta \in (0, 2^{-(k_+-1)})$, 
    \begin{align} \label{eq:kde-ub-term-2}
        \bigl\{\check{f}_n(x_0) - f(x_0)\bigr\}^2 \leq \|f\|_{\infty}^2 \leq C_1^2(\beta) \cdot \gamma^{2/(\beta+1)}. 
    \end{align}
    Note that when $\delta = 2^{-(k_+-1)}$ and $n> 3\log(2)/c(\beta,K)$, we have that
    \[
    \frac{\log(8/\delta)}{n} = \frac{1}{n}\log\bigl(2^{k_+ + 2}\bigr) \in \biggl[\frac{c(\beta,K)}{2},\frac{4c(\beta,K)}{3}\biggr].
    \]
    Thus, there exists $C_3(\beta,K) \geq 1$ depending only on $\beta,K$ such that when $\delta\in(0,2^{-(k_+-1)})$, we have $1 \leq C_3(\beta,K) \bigl\{ \bigl( \frac{\log(8/\delta)}{n} \bigr)^{2\beta/(2\beta+1)} \wedge 1 \bigr\}$, and when $\delta\in[2^{-(k_+-1)}, 1]$, we have 
    \[
    \biggl( \frac{\log(8/\delta)}{n} \biggr)^{2\beta/(2\beta+1)} \leq C_3(\beta,K) \biggl\{ \biggl( \frac{\log(8/\delta)}{n} \biggr)^{2\beta/(2\beta+1)} \wedge 1 \biggr\}.
    \]
    The final conclusion follows by combining~\eqref{eq:kde-lepski-small-n}, ~\eqref{eq:kde-ub-term-1} and~\eqref{eq:kde-ub-term-2}, and setting $\check{C}(\beta,K) \coloneqq C_2(\beta,K) \vee \bigl[C_3(\beta,K) \bigl\{4C(\beta,K) \vee C_1^2(\beta) \bigr\}\bigr]$.
\end{proof}

\end{document}